\documentclass[11pt]{article} 

\usepackage[utf8]{inputenc} 

\usepackage{geometry} 
\usepackage{graphicx} 
\usepackage[font=sf]{caption,subcaption}
\usepackage{hyperref}
\hypersetup{
    colorlinks=false,
    linkcolor=blue,
    filecolor=blue,      
    urlcolor=blue,
}

\usepackage{array} 
\usepackage{paralist} 
\usepackage{verbatim} 
\usepackage{lipsum}
\usepackage{amsmath, amsfonts, amssymb, amsxtra, amsthm, mathrsfs}
\usepackage[inline]{enumitem}
\usepackage{color, colortbl}
\usepackage{tikz,tikz-3dplot,pgfplots}
\usepackage{multirow}
\usepackage{colortbl}
\usepackage{makecell}

\usepackage{pdfpages}
\usepackage[titles]{tocloft}

\newtheoremstyle{mystyle}{}{}{\itshape}{}{\bfseries}{.}{ }{\thmname{#1}\thmnumber{ #2}\thmnote{ (\bfseries #3)}}
\theoremstyle{plain}
\newtheorem{thm}{Theorem}[section]

\newtheorem{prop}[thm]{Proposition}
\newtheorem{lm}[thm]{Lemma}
\newtheorem{cor}[thm]{Corollary}
\newtheorem{conj}[thm]{Conjecture}
\theoremstyle{definition}

\newtheorem{rmk}[thm]{Remark}
\newtheorem{eg}[thm]{Example}

\DeclareMathOperator{\realpart}{Re}
\DeclareMathOperator{\impart}{Im}

\DeclareMathOperator{\isom}{Isom}

\DeclareMathOperator{\hull}{hull}
\DeclareMathOperator{\core}{core}
\DeclareMathOperator{\diam}{diam}

\DeclareMathOperator{\psl}{PSL}

\DeclareMathOperator{\stab}{Stab}
\DeclareMathOperator{\so}{SO}

\DeclareMathOperator{\re}{Re}
\DeclareMathOperator{\Cr}{Cr}
\DeclareMathOperator{\symb}{\mathsf{Symb}}
\definecolor{cof}{RGB}{219,144,71}
\definecolor{pur}{RGB}{186,146,162}
\definecolor{greeo}{RGB}{91,173,69}
\definecolor{greet}{RGB}{52,111,72}
\usetikzlibrary{calc}

\usepackage[calcwidth]{titlesec}
\titlespacing*{\paragraph}{0pt}{0pt}{1em}
\newcommand{\periodafter}[1]{#1.}
\titleformat{\paragraph}[runin]{\bfseries}{\theparagraph}{}{\periodafter}

\pgfplotsset{compat=1.14}

\numberwithin{equation}{section}

\title{\bfseries Elementary planes in the Apollonian orbifold}
\author{Yongquan Zhang}
\date{October 6, 2022}

\begin{document}
\maketitle

\begin{abstract}
In this paper, we study the topological behavior of elementary planes in the Apollonian orbifold $M_A$, whose limit set is the classical Apollonian gasket. The existence of these elementary planes leads to the following failure of equidistribution: there exists a sequence of closed geodesic planes in $M_A$ limiting only on a finite union of closed geodesic planes. This contrasts with other acylindrical hyperbolic 3-manifolds analyzed in \cite{MMO1, MMO2, acy_geom_finite}.

On the other hand, we show that certain rigidity still holds: the area of an elementary plane in $M_A$ is uniformly bounded above, and the union of all elementary planes is closed. This is achieved by obtaining a complete list of elementary planes in $M_A$, indexed by their intersection with the convex core boundary. The key idea is to recover information on a closed geodesic plane in $M_A$ from its boundary data; requiring the plane to be elementary in turn puts restrictions on these data.
\end{abstract}
\begin{center}
\begin{minipage}{0.85\textwidth}
    \tableofcontents
\end{minipage}
\end{center}

\vfill
\rule{0.5\textwidth}{0.5pt}
\begin{center}
\begin{minipage}{0.9\textwidth}

\end{minipage}
\end{center}

\thispagestyle{empty}
\newpage

\clearpage
\pagenumbering{arabic} 

\section{Introduction}
In this paper, we explore the following question: what are all the circles that intersect the Apollonian gasket in countably many points, and how are they distributed? Equivalently, what are all the elementary planes in the Apollonian orbifold, and how do they behave geometrically? This problem is motivated by the study of topological behavior of geodesic planes in geometrically finite acylindrical hyperbolic 3-manifolds, in search for generalizations of Ratner's theorem in this setting.

\paragraph{The Apollonian gasket}
A \emph{Descartes configuration} is a collection of 4 mutually tangent circles in $S^2$ that bound disjoint disks. Given a Descartes configuration, we can add four more circles to the triangular regions, each tangent to 3 circles in the original configuration. If we continue to fill the triangular regions with circles \emph{ad infinitum}, we obtain an \emph{Apollonian circle packing}. An \emph{Apollonian gasket} is the complement in $S^2$ of the union of the open disks bounded by the circles.

It is clear that any two Apollonian gaskets are conformally equivalent. Let $\mathcal{A}$ be the one obtained from the Descartes configuration consisting of $y=0$, $y=-1$, $x^2+(y+1/2)^2=1/4$, and $(x-1)^2+(y+1/2)^2=1/4$ in $\hat{\mathbb{C}}$; see Figure~\ref{fig: apollonian_gasket}. We will refer to this particular normalization as \emph{the} Apollonian gasket. We remark that in $\mathcal{A}$, the circles tangent to $\mathbb{R}$ are precisely the \emph{Ford circles} (see e.g.~\cite{ford1938fractions} and \cite[\S5.5]{modular_functions_apostol} for details) reflected across $\mathbb{R}$.
\begin{figure}[htp]
\centering
\includegraphics[trim={0cm 0.8cm 0cm 0.8cm},clip,width=0.9\textwidth]{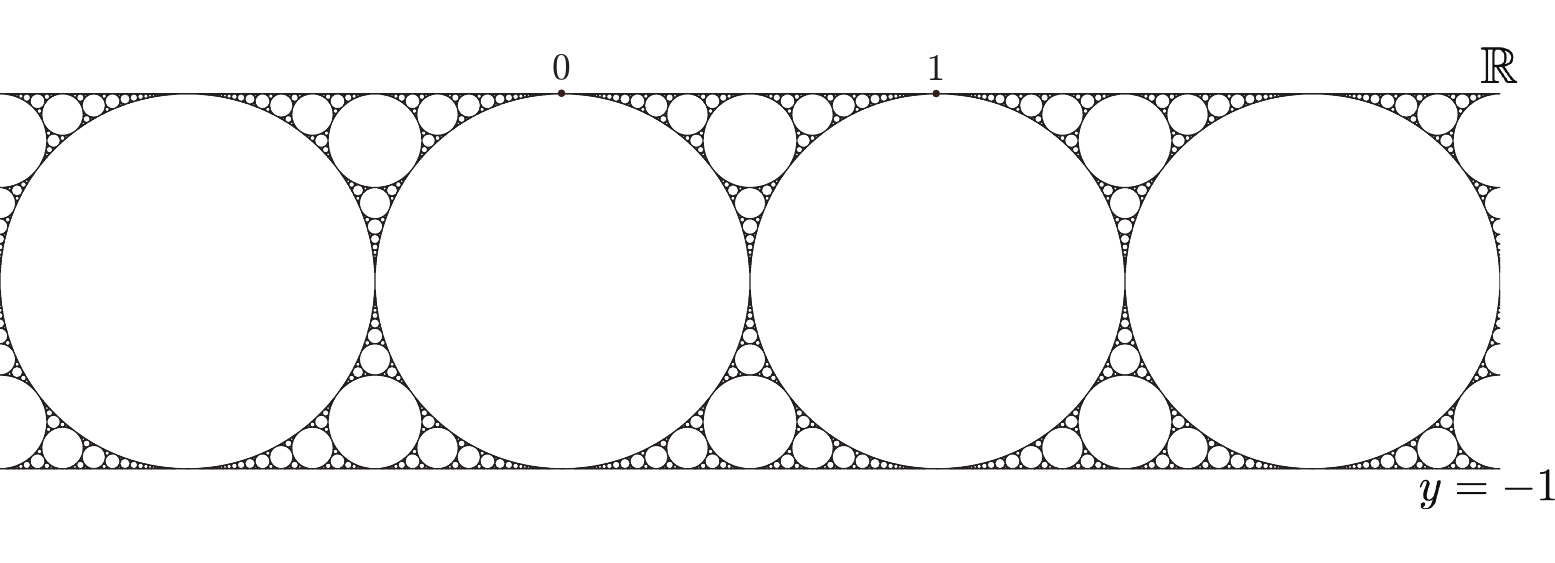}
\caption{The Apollonian gasket $\mathcal{A}$}
\label{fig: apollonian_gasket}
\end{figure}

\paragraph{The Apollonian group}
Consider the stabilizer $\Gamma_A$ of $\mathcal{A}$ in $G=\psl(2,\mathbb{C})$. This is a Kleinian group with torsion, whose limit set is precisely $\mathcal{A}$. It acts transitively on the circles in $\mathcal{A}$, and the stabilizer of $\mathbb{R}$ in $\Gamma_A$ is $\psl(2,\mathbb{Z})$. Finally, $\Gamma_A$ acts transitively on all the points of tangency in the packing. See \S\ref{sec: apollonian_group} for details.

\paragraph{Space of circles}
Let $S^2\cong\hat{\mathbb{C}}$ be the sphere at infinity of the hyperbolic 3-space $\mathbb{H}^3$. The space $\mathcal{C}$ of oriented circles on $S^2$ can be identified with the homogeneous space $G/H$, where $G=\psl(2,\mathbb{C})$ and $H=\psl(2,\mathbb{R})$. Elements of $G$ act on $G/H$ by left multiplication, which corresponds to the action of M\"obius transforms on oriented circles.

\paragraph{Elementary circles}
A circle $C\subset S^2$ is called \emph{elementary} if $C\cap\mathcal{A}$ is countable, and both components of $\hat{\mathbb{C}}-C$ intersect $\mathcal{A}$. Given an elementary circle $C$, note that $\gamma\cdot C$ is also elementary for any $\gamma\in\Gamma_A$.

Let $\Gamma_A^C$ be the stabilizer of $C$ in the Apollonian group $\Gamma_A$. Then $\Gamma_A^C$ acts on the countably many points in $C\cap\mathcal{A}$. Let $\mathcal{O}_C$ be the set of orbits of this action.

Our main results on elementary circles can be formulated as follows:
\begin{thm}[Uniformly finite orbit space]\label{thm: finite_orbit}
For any elementary circle $C$, the set of orbits $\mathcal{O}_C$ is finite, and in fact $|\mathcal{O}_C|\le10$.
\end{thm}

\begin{thm}[Elementary circles are closed]\label{thm: elem_circles_closed}
The $\Gamma_A$-invariant set of all elementary circles is closed in the space of circles $\mathcal{C}=G/H$.
\end{thm}

These results on elementary circles can be restated in terms of elementary planes in the corresponding orbifold $M_A=\Gamma_A\backslash\mathbb{H}^3$, as we will explain below.

\paragraph{Geodesic planes in hyperbolic 3-manifold}
Let $M=\Gamma\backslash\mathbb{H}^3$ be a complete hyperbolic 3-manifold (or orbifold), where $\Gamma\subset\psl(2,\mathbb{C})$ is a Kleinian group. A geodesic plane in $M$ is a totally geodesic isometric immersion
$$f:\mathbb{H}^2\to M.$$
We often identify $f$ with its image $P:=f(\mathbb{H}^2)$ and call the latter a geodesic plane as well.

We are mostly interested in geodesic planes that intersect the \emph{convex core} of $M$, which is the smallest closed convex subset $\core(M)\subset M$ containing all closed geodesics. For simplicity, let $\core(P)=P\cap\core(M)$ for any geodesic plane $P$.

From the perspective of homogeneous dynamics, $\Gamma\backslash G$ can be identified with the frame bundle $FM$ over $M$, and any oriented geodesic plane lifts to an $H$-orbit in $\Gamma\backslash G$. Since the projection from $FM$ to $M$ is proper, to understand the topological behavior of $P$, it often suffices to study the corresponding $H$-orbit in $\Gamma\backslash G$.

Geodesic planes and circles are closed related. Given any (oriented) geodesic plane $P$, take a lift $\tilde P\subset\mathbb{H}^3$ with respect to the covering map $\mathbb{H}^3\to M=\Gamma\backslash\mathbb{H}^3$. The boundary at infinity of $\tilde P$ is a circle $C\subset\hat{\mathbb{C}}\cong S^2$. Conversely, any circle $C$ determines a geodesic plane in $\mathbb{H}^3$, and in turn a geodesic plane $P$ in $M$. We call $C$ a \emph{$\Gamma$-boundary circle} (or just a boundary circle if the group is understood) of $P$. Note that $\Gamma\cdot C$ gives all the boundary circles of $P$. To study the topological behavior of $P$, we can therefore study the corresponding $\Gamma$-orbit in the space of circles $\mathcal{C}=G/H$, and vice versa.

\paragraph{The Apollonian orbifold}
We mostly focus on the orbifold $M_A:=\Gamma_A\backslash\mathbb{H}^3$, which we call the \emph{Apollonian orbifold}. The convex core of $M_A$ has finite volume and totally geodesic boundary, which is isometric to the modular surface $X:=\psl(2,\mathbb{Z})\backslash\mathbb{H}^2$. The orbifold $M_A$ has a unique cusp of rank $1$. These properties can be deduced from those of the group $\Gamma_A$ listed above; see \S\ref{sec: apollonian_orbifold}.

\paragraph{Elementary planes}
By an \emph{elementary plane} in $M$, we mean a closed geodesic plane $P$ intersecting $\core(M)$ whose fundamental group is virtually abelian. In the case of the Apollonian orbifold $M_A$, for any elementary plane $P$, $\core(P)=P\cap\core(M_A)$ is a properly immersed convex elementary surface of finite volume whose boundary consists of complete geodesics, and thus is either an ideal polygon, a punctured ideal polygon, a single crown or a double crown (see Figure~\ref{fig: crowns and double crowns} and \S\ref{sec: elem_acylindrical}).

\begin{figure}[htp]
\captionsetup{width=.75\linewidth}
\centering
\begin{minipage}{0.38\linewidth}
\centering
\begin{subfigure}[b]{\linewidth}
\includegraphics[width=0.8\textwidth]{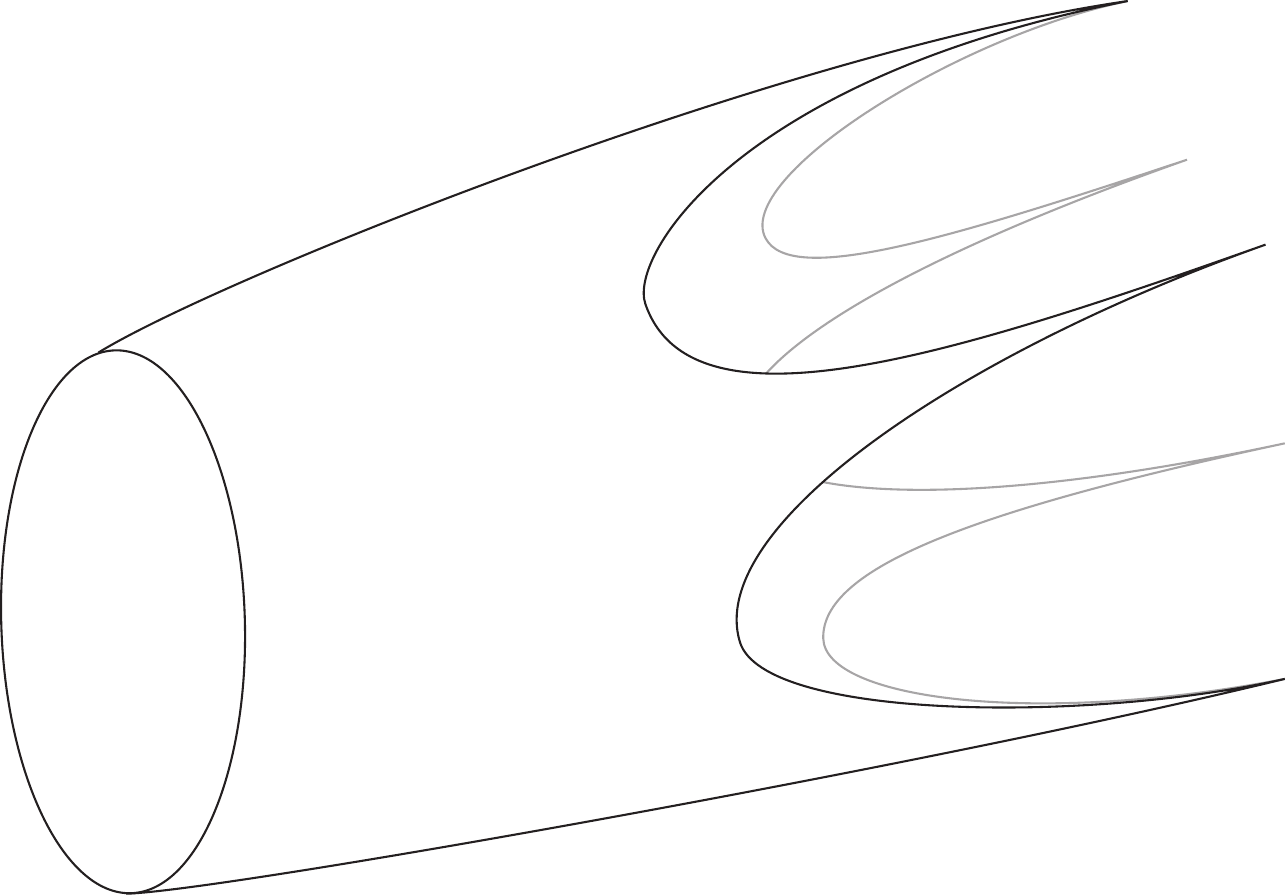}
\caption{A single crown}
\label{fig: crown}
\end{subfigure}
\end{minipage}
\begin{minipage}{0.4\linewidth}
\centering
\begin{subfigure}[b]{\linewidth}
\includegraphics[width=\textwidth]{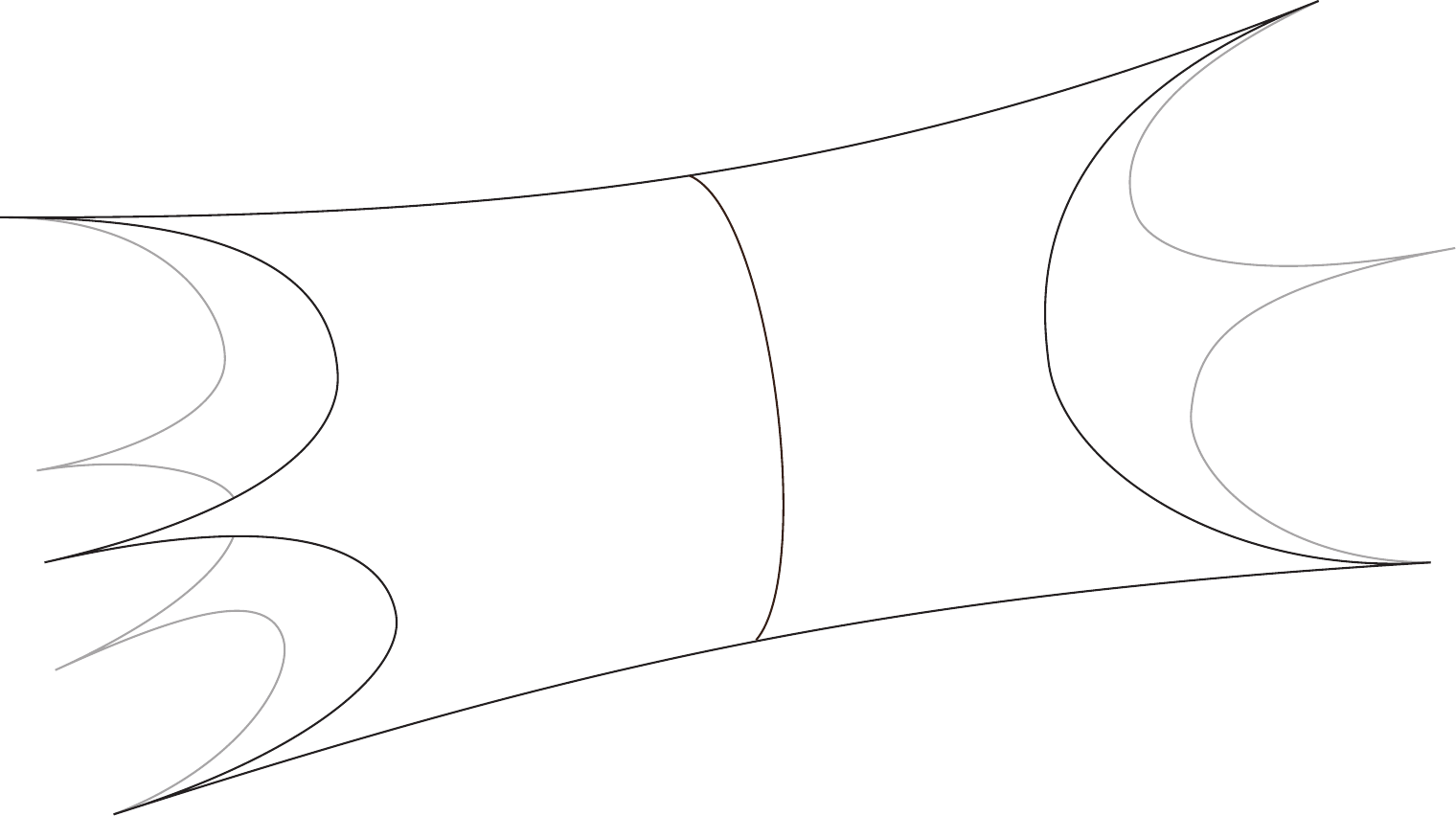}
\caption{A double crown}
\label{fig: double_crown}
\end{subfigure}
\end{minipage}
\caption{Finite-volume convex subsurfaces of a hyperbolic cylinder. Each side is either a closed geodesic, or complete geodesics forming spikes}
\label{fig: crowns and double crowns}
\end{figure}

The following proposition states that we can recognize elementary planes in $M_A$ from their boundary circles:
\begin{prop}\label{prop: elementary_circle}
A geodesic plane $P$ in the Apollonian orbifold $M_A$ is elementary if and only if any boundary circle $C$ of $P$ is elementary. In fact, we have the following possibilities: $\core(P)$ is a properly immersed
\begin{enumerate}[label=\normalfont{(\arabic*)}, topsep=0mm, itemsep=0mm]
\item\label{item: polygon} ideal polygon if and only if $|C\cap\mathcal{A}|<\infty$;
\item\label{item: ideal_polygon} punctured ideal polygon if and only if $C\cap\mathcal{A}$ has a unique accumulation point;
\item\label{item: crown} single crown if and only if $C\cap\mathcal{A}$ has exactly two accumulation points $p,q$, and $\mathcal{A}$ only intersects one component of $C-\{p, q\}$;
\item\label{item: double_crown} double crown if and only if $C\cap\mathcal{A}$ has exactly two accumulation points $p,q$, and $\mathcal{A}$ intersects both components of $C-\{p, q\}$.
\end{enumerate}
\end{prop}
For simplicity, we refer to the first type of elementary planes in the proposition above as \emph{elliptic}, the second type as \emph{parabolic}, and the last two as \emph{hyperbolic} elementary planes.

\paragraph{Topology and geometry of elementary planes}
The main results, stated in terms of elementary planes, are as follows:
\begin{thm}\label{thm: topologies}
Let $P$ be an elementary plane in the Apollonian orbifold $M_A$. Then $\core(P)$ is a properly immersed
\begin{itemize}[topsep=0mm, itemsep=0mm]
\item ideal triangle, quadrilateral, or hexagon; or
\item punctured ideal monogon or bigon; or
\item single crown with 2, 4, or 6 spikes; or
\item double crown with (2,2) or (6,2) spikes on each side.
\end{itemize}
In particular, $\core(P)$ has area $\le8\pi$.
\end{thm}
\begin{thm}\label{thm: elem_closed}
The union of all elementary planes is closed in $M_A$. That is, if a sequence of elementary planes converges, the limit is a union of elementary planes.
\end{thm}
It is clear that Theorems~\ref{thm: finite_orbit} and \ref{thm: elem_circles_closed} follow from \ref{thm: topologies} and \ref{thm: elem_closed} respectively, considering the correspondence between circles and planes as explained above.

\paragraph{Modular symbols and elementary planes in $M_A$}
The key idea in the proof of the main results is to probe the geometry and distribution of elementary planes with their boundary data. The intersection of an elementary plane $P$ with $\partial\core(M_A)\cong X=\psl(2,\mathbb{Z})\backslash\mathbb{H}^2$ consists of closed geodesics and complete geodesics from cusp to cusp. In fact, at least one component must be a complete geodesic from cusp to cusp, and these geodesics form one (for ideal polygons, punctured ideal polygons, single crowns) or two (for double crowns) cycles.

We use the language of \emph{modular symbols} to describe these cycles of geodesics from cusp to cusp; see \S\ref{sec: marking_crown}. A natural question is then what are all the modular symbols coming from elementary planes. It turns out that different planes may share the same symbol, so we need to introduce a related, but more precise way to index the planes and their boundary circles.

\paragraph{Markings of elementary planes}
Note that any elementary circle $C$ must pass through parabolic fixed points, i.e.\ points of tangency of circles in $\mathcal{A}$. Since $\Gamma_A$ acts on these points transitively, we may assume $C$ passes through $\infty$. It then intersects $y=-1$ and $y=0$ at two rational numbers (since these are exactly the points of tangency), say $\alpha$ and $\beta$ respectively. Conversely, given a pair of rational numbers $\alpha,\beta$, let $l(\alpha,\beta)$ be the line passing through $\alpha-i$ and $\beta$. This gives a closed geodesic plane $P$ in $M_A$, although not necessarily elementary. Note that for any integer $k$, $l(\alpha,\beta)$ and $l(\alpha+k,\beta+k)$ represent the same plane in $M$. We call the pair $(\alpha,\beta)\in (\mathbb{Q}\times\mathbb{Q})/\mathbb{Z}$ a \emph{marking} of $P$. Unlike modular symbols, markings \emph{do} determine the plane.

Note that markings are not unique; we made a choice putting a tangency point at infinity. In the orbifold $M_A$, this is tantamount to choosing an orientation and a ``spike" of $\core(P)$. Define the change of marking map
$$T:(\mathbb{Q}\times\mathbb{Q})/\mathbb{Z}\to(\mathbb{Q}\times\mathbb{Q})/\mathbb{Z}$$
as follows. Geometrically, $l(T(\alpha,\beta))$ gives the same plane as $l(\alpha,\beta)$, but the next spike along the direction of the orientation is lifted to $\infty$.

The map $T$ is periodic on any input, as crowns have finitely many spikes. Markings of an elementary plane $P$ thus form one or two periodic orbits of $T$.

A concrete way to classify elementary planes is then determining what orbits of markings give each type of elementary surfaces. Using a symbolic coding scheme analogous to cutting sequences for $\psl(2,\mathbb{Z})$ (see \S\ref{sec: diophantine}), we give a complete list of symbols for each type of elementary surfaces:
\begin{thm}\label{thm: collected}
The line $l(\alpha,\beta)$ gives an elementary plane if and only if the $T$-orbit of $(\alpha,\beta)$ or $(1-\alpha,1-\beta)$ contains a marking appearing in Table~\ref{tab: collected_symbols}.
\end{thm}

\begin{table}[htp]
\centering
\begin{tabular}{c|c|c|c}
$(\alpha,\beta)$&Parameter(s)&$d$&Core geodesic\\
\hline\hline
\multicolumn{4}{c}{Elliptic crowns}\\
\hline\hline
\arrayrulecolor{gray}
$(0,0)$&&1&\\
\hline
$(0,1/n)$&$n\ge1$&2&\\
\hline
$(0,n/(n^2-1))$&$n\ge2$&2&\\
\arrayrulecolor{black}\hline\hline
\multicolumn{4}{c}{Parabolic crowns}\\
\arrayrulecolor{black}\hline\hline
\arrayrulecolor{gray}
$(1/2,1/2)$&&1&\\
\hline
$\left(0,\frac{2n+1}{2n^2+2n}\right)$&$n\ge1$&2&\\
\arrayrulecolor{black}\hline\hline
\multicolumn{4}{c}{Single hyperbolic crowns}\\
\hline\hline
\arrayrulecolor{gray}
$(1/m,-1/n)$&$m,n\ge1$&2&$V_3^{m}V_1^n$\\
\hline
$(1/m,1/n)$&$m,n\ge2, (m,n)\neq(2,2)$&2&$V_3V_2^{m-2}V_3V_2^{n-2}$\\
\hline
$\left(0,\frac{mn+1}{n^2m+2n}\right)$&$m\ge3,n\ge1$&2&$V_3V_1^{m-2}$\\
\hline
$\left(0,\frac{mn-1}{n^2m-2n}\right)$&$m\ge3,n\ge2$&2&$V_3V_2^{m-2}$\\
\hline
$\left(\frac1n,\frac{\zeta+m}{m^2+m\zeta-1}\right)$&\makecell{$m\ge2$, $\zeta$ a positive divisor\\of $m^2-1$, $n=m+(m^2-1)/\zeta$}&4&$V_3V_2^{m-2}V_3V_2^{m-2}V_3V_2^{n-2}$\\
\hline
$\left(\frac{1}n,\frac{\zeta-m}{-m^2+m\zeta+1}\right)$&\makecell{$m\ge3$, $\zeta\ge(m^2-1)/(m-2)$\\a positive divisor of $m^2-1$,\\$n=m-(m^2-1)/\zeta$}&4&$V_3V_1^{m-2}V_3V_1^{m-2}V_3V_1^{n-2}$\\
\hline
$\left(\frac{1}n,-\frac{m-\zeta}{m^2-m\zeta-1}\right)$&\makecell{$m\ge3$, $\zeta\le m-2$\\a positive divisor of $m^2-1$, \\$n=-m+(m^2-1)/\zeta$}&4&$V_3V_2^{m-1}V_3V_2^{m-1}V_3^{n-1}$\\
\hline
$\left(\frac{t}{2t^2-1},-\frac{t}{2t^2-1}\right)$&$t\ge2$&6&$V_2^{2t-1}V_3V_2^{2t-1}V_3^{2t-1}V_2V_3^{2t-1}$\\
\arrayrulecolor{black}\hline\hline
\multicolumn{4}{c}{Double hyperbolic crowns}\\
\hline\hline
\arrayrulecolor{gray}
\makecell{$(1/n,1-1/n)$\\$\left(\frac1{n-2},-1-\frac1{n-2}\right)$}&$n\ge3$&\makecell{2\\2}&$V_3V_2^{n-2}V_3V_1^{n-2}$\\
\hline
\makecell{$(2/7,5/7)$\\$(1/3,8/3)$}&&\makecell{6\\2}&$V_3V_2V_3V_2V_3V_1V_3V_1$\\
\hline
\makecell{$\left(\frac{n}{nm-1},-\frac{m}{nm-1}\right)$\\$\left(\frac{m}{nm+1},\frac{n}{nm+1}\right)$}&$m,n\ge2$&\makecell{2\\2}&$V_2V_1^{m-1}V_3^{n-1}V_2V_3^{m-1}V_1^{n-1}$\\
\hline
\makecell{$\left(\frac{2n+1}{4n},-\frac1n\right)$\\$\left(\frac{n}{n+1},\frac{2n+3}{4n+4}\right)$}&$n\ge1$&\makecell{2\\2}&$V_1^nV_3V_2V_3^{n-1}V_2V_3$\\
\hline
\makecell{$\left(0,\frac{nm^2-m-n}{m^2n^2-n^2-2mn+1}\right)$\\$\left(0,\frac{km^2-m-k}{m^2k^2-k^2-2mk+1}\right)$}&\makecell{$n,k\ge2$ such that\\$m=\frac{(k+n)^2+2k^2n^2+\sqrt{(k^2-n^2)^2+4k^4n^4}}{2kn(n+k)}$\\is an integer}&\makecell{2\\2}&$V_1V_2^{m-2}V_3V_2^{m-2}$\\
\hline
\makecell{$\left(0,\frac{nm^2-m-n}{m^2n^2-n^2-2mn+1}\right)$\\$\left(0,\frac{km^2+m-k}{m^2k^2-k^2+2mk+1}\right)$}&\makecell{$n>k\ge2$ such that\\$m=\frac{(k-n)^2+2k^2n^2+\sqrt{(k^2-n^2)^2+4k^4n^4}}{2kn(n-k)}$\\is an integer}&\makecell{2\\2}&$V_1V_2^{m-2}V_3V_2^{m-2}$\\
\arrayrulecolor{black}\hline
\end{tabular}
\caption{\small Elementary planes. One marking is listed from each $T$-orbit. Markings for double crowns come in pairs, one from each side. The third column lists the period $d$ of $T$. The words in the fourth column give representatives of conjugacy classes, in terms of a set of generators $V_1, V_2, V_3$ of the Apollonian group; see \S\ref{sec: diophantine_apollonian} for details.}
\label{tab: collected_symbols}
\end{table}

See Figure~\ref{fig: elementary_limit_gasket} and Example~\ref{eg: ideal_quadrilateral} for an illustration of the second family in the list. We refer to Theorems~\ref{thm: elliptic}, \ref{thm: parabolic}, \ref{thm: crown}, and \ref{thm: double_crown} for complete lists of symbols for each type of elementary surfaces.

\begin{figure}[htp]
\centering
\captionsetup{width=.8\linewidth}
\includegraphics[width=0.7\linewidth]{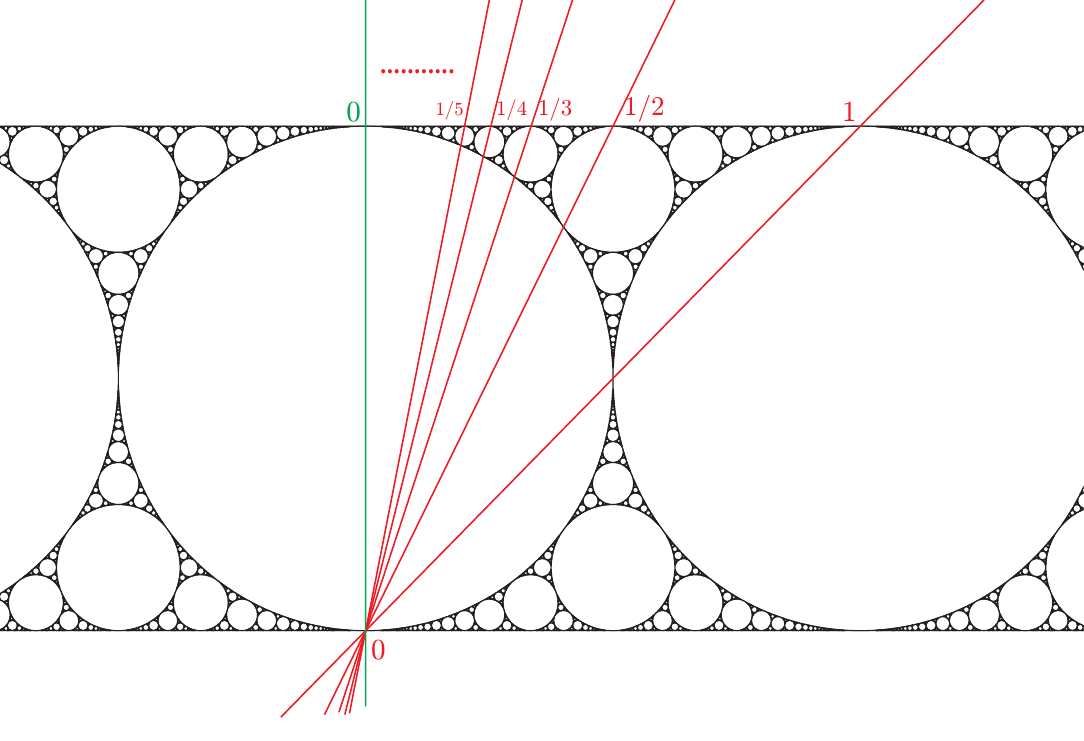}
\caption[A sequence of elementary planes determined by $l(0,1.n)$]{A sequence of ideal quadrilaterals determined by $l(0,1/n)$, tending to two copies of the ideal triangle determined by $l(0,0)$. Two opposite sides of the quadrilaterals are getting closer and closer, as in Figure~\ref{fig: elementary_limit}}
\label{fig: elementary_limit_gasket}
\end{figure}

\paragraph{Topology and geometry of elementary planes, revisited}
By going through the lists, we can then describe all possible topologies of $\core(P)$ for an elementary plane $P$, as in Theorem~\ref{thm: topologies}. In particular, the possible topological types of $\core(P)$ are very limited. Moreover, each elementary plane is immersed in $M_A$ in a very controlled way:
\begin{thm}\label{thm: complexity}
Let $P$ be an elementary plane in the Apollonian orbifold $M_A$ and $(\alpha,\beta)$ be any marking of $P$. Then the continued fractions of $\alpha$ and $\beta$ have length $\le8$. Geometrically, we have
\begin{itemize}[topsep=0mm, itemsep=0mm]
\item Each component of the modular symbol of $P$ makes an excursion into the unique cusp of $\partial\core(M_A)$ at most $8$ times;
\item When $\core(P)$ is a crown or a double crown, the core geodesic of $P$ makes an excursion into the unique cusp of $M_A$ at most $8$ times.
\end{itemize}
\end{thm}
Here, we say a geodesic \emph{makes an excursion into the cusp} of $M_A$ if it enters and then leaves a fixed small enough cusp cylinder $M_\epsilon$.

The uniformly controlled complexity of the geometry of elementary planes in $M_A$ detailed in Theorems~\ref{thm: topologies} and \ref{thm: complexity} suggests that the collective distribution of elementary planes in $M_A$ is also controlled. For example, given any sequence $P_i$ of elementary planes, the portion $\core(P_i)=\core(M_A)\cap P$ cannot become uniformly distributed in $\core(M_A)$ with respect to the hyperbolic volume measure, for that necessarily implies that the area of $\core(P_i)$ goes to infinity. In \S\ref{sec: geom_top_elem_planes} we will prove Theorem~\ref{thm: elem_closed} using these results.

We can describe how sequences of elementary planes converge. Theorem~\ref{thm: complexity} states that for an elementary plane $P$, each component of $P\cap\partial\core(M_A)$ makes excursion into the unique cusp of the modular surface at most 8 times. Thus for a sequence $\{P_i\}$ of elementary planes, boundary components of $\core(P_i)$ can go further into the cusp for each excursion, but cannot make arbitrarily many excursions. Deeper excursion into the cusp pushes part of the elementary surface into the cusp as well. In the limit, the process is simply pinching a few sides of $\core(P)$; see Example~\ref{eg: ideal_quadrilateral}. By choosing suitable base points, we can view each limiting elementary plane as a Gromov-Hausdorff limit of the sequence.

\begin{eg}\label{eg: ideal_quadrilateral}
The sequence of planes $P_n$ determined by $l_n:=l(0,1/n)$ tends to the plane $P_0$ determined by $l_0:=l(0,0)$ as $n\to\infty$. More precisely, $\core(P_n)$ is finitely covered by an ideal quadrilateral $Q_n$, and $\core(P_0)$ is finitely covered by an ideal triangle $T$. As $n\to\infty$, two opposite sides of $Q_n$ is pinched together, and in the limit, $Q_n$ tends to two copies of the ideal triangle $T$. See Figures~\ref{fig: elementary_limit_gasket} and \ref{fig: elementary_limit} for an illustration of this sequence.\qed
\end{eg}
\begin{figure}[htp]
\centering
\includegraphics[width=0.8\linewidth]{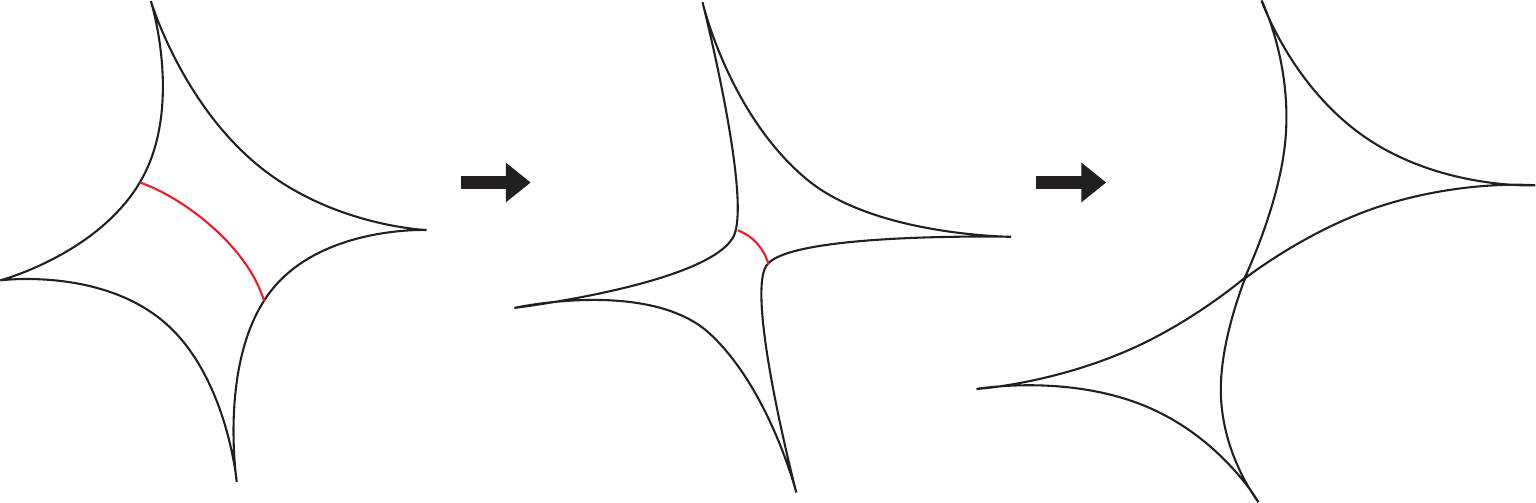}
\caption{Topological picture of the limiting sequence}
\label{fig: elementary_limit}
\end{figure}

\paragraph{Motivation: Generalization of Ratner's theorem}
As mentioned above, we are motivated by the study of topological behavior of geodesic planes in hyperbolic 3-manifolds. When the hyperbolic 3-manifold $M$ has finite volume, any geodesic plane in $M$ is either closed or dense, independently due to Ratner \cite{ratner} and Shah \cite{shah}. Using the correspondence between planes, circles and frames, this result can also be stated in the following stronger form: if $\Gamma\subset G$ is a lattice, then any $H$-invariant subset of $\Gamma\backslash G$ (equivalently, any $\Gamma$-invariant subset of $G/H$) is either closed or dense.

Recent works have generalized this rigidity to convex cocompact acylindrical hyperbolic 3-manifolds of infinite volume. Here, we say a hyperbolic 3-manifold is \emph{convex cocompact} if $\core(M)$ is compact. Then the topological condition of being acylindrical means that the compact topological manifold $\core(M)$ has incompressible boundary, and any essential cylinder in $\core(M)$ is boundary parallel (see \cite{hyperbolization1} and also \S\ref{sec: acy_def}). We have
\begin{thm}[\cite{MMO2}]\label{rigidity_acy}
Let $M\cong\Gamma\backslash\mathbb{H}^3$ be a convex cocompact, acylindrical hyperbolic 3-manifold, and set $M^*$ to be the interior of $\core(M)$. Then any geodesic plane $P$ intersecting $M^*$ is either closed or dense in $M^*$.
\end{thm}
Let $\mathcal{P}^*_M\subset\mathcal{P}_M$ be the collection of geodesic planes that meet $M^*$, and $\mathcal{C}^*\subset G/H$ be the corresponding set of boundary circles on $S^2$. Let $F^*M$ be the set of frames tangent to a geodesic plane in $\mathcal{P}^*_M$. Similarly, we have the following stronger form of Theorem~\ref{rigidity_acy}:
{\theoremstyle{plain}
\newtheorem*{thm_ratner_acy}{Theorem~\ref{rigidity_acy}'}
\begin{thm}[\cite{MMO2}]\label{thm: rigidity_homo}
Let $M\cong\Gamma\backslash\mathbb{H}^3$ be a convex cocompact, acylindrical hyperbolic 3-manifold. Then any $H$-invariant subset of $F^*M$ is either closed or dense in $F^*M$. Equivalently, any $\Gamma$-invariant subset of $\mathcal{C}^*$ is closed or dense in $\mathcal{C}^*$.
\end{thm}

It is a natural question to ask if these results can be generalized to other hyperbolic 3-manifolds. In particular, Thurston's definition of acylindrical manifolds in \cite{hyperbolization1} includes a larger family of \emph{geometrically finite} hyperbolic 3-manifolds, where we only require that the unit neighborhood of the convex core has finite volume. This definition in particular includes the Apollonian manifold $M_A$. Recently, Benoist and Oh have generalized Theorems~\ref{rigidity_acy} and \ref{thm: rigidity_homo} to include certain geometrically finite acylindrical manifolds with cusps, but notably not $M_A$ \cite{acy_geom_finite}.

One key ingredient of \cite{MMO2, acy_geom_finite} is some version of the following general isolation result (see e.g.\ \cite{MMO1}). Let $\Gamma\subset G$ be a Zariski dense Kleinian group with limit set $\Lambda$, and $C\subset S^2$ a circle whose stabilizer $\Gamma^C$ in $\Gamma$ is nonelementary, with limit set on both sides of $C$. If a sequence of distinct circles $C_n\to C$, then the closure of $\bigcup_n\Gamma\cdot C_n$ in $\mathcal{C}$ contains all circles meeting $\Lambda$. Roughly speaking, this means that if $M=\Gamma\backslash\mathbb{H}^3$ contains an immersed, totally geodesic surface with nonelementary fundamental group, then the dynamics of this large Fuchsian group forces nearby planes to become densely distributed in $M$. In the cases considered in \cite{MMO2, acy_geom_finite}, every closed geodesic plane intersecting $M^*$ has nonelementary fundamental group. So these planes are isolated from each other, unless they become densely distributed.

On the other hand, a closed geodesic plane in a general geometrically finite acylindrical 3-manifold can be elementary, so it is possible to have a sequence of closed planes limiting only on elementary planes. Indeed this does happen in the Apollonian orbifold $M_A$ as in Example~\ref{eg: ideal_quadrilateral}. As a matter of fact, we have
\begin{thm}\label{thm: pairs_of_pants}
There exists a sequence of closed geodesic planes with nonelementary fundamental group in the Apollonian orbifold $M_A$ limiting only on a finite union of elementary planes.
\end{thm}
An immediate corollary is that Theorem~\ref{thm: rigidity_homo} does not generalize to the geometrically finite setting. Indeed, if $\{C_n\}$ is a sequence of circles giving the sequence of geodesic planes in Theorem~\ref{thm: pairs_of_pants}, then $\bigcup_n\Gamma_A\cdot C_n$ is a $\Gamma_A$-invariant subset of the space of circles $\mathcal{C}^*$ that is neither closed nor dense in $\mathcal{C}^*$.

While Theorem~\ref{thm: rigidity_homo} does not generalize to $M_A$, the $H$-orbits in the counterexamples we have are \emph{locally closed}, i.e.\ open in their closures. Moreover, they only limit on elementary planes. We thus make the following conjecture:
\begin{conj}\label{conj: rigidity}
For the Apollonian orbifold $M_A$, any $H$-invariant subset $\mathcal{S}$ of $F^*M_A$ is either locally closed or dense in $F^*M_A$. Moreover, in the former case, $\overline{\mathcal{S}}-\mathcal{S}$ is a union of elementary planes.
\end{conj}

Finally, it remains unknown whether Theorem~\ref{rigidity_acy} generalizes.

\paragraph{Notes and references}
We refer to \cite{mohammadi2020isolations} for a quantitative description of the isolation phenomenon for closed geodesic planes with nonelementary fundamental groups in geometrically finite hyperbolic 3-manifolds.

The paper is organized as follows. In \S\ref{sec: prelim}, we discuss some preliminaries from topology and geometry of 3-manifolds, and prove some general results about elementary planes in geometrically finite acylindrical hyperbolic 3-manifolds. In \S\ref{sec: apollo_group} we list some properties of the Apollonian group $\Gamma_A$ and orbifold $M_A$. Some additional visualizations of $M_A$ and its manifold covers are included in the appendix.

In \S\ref{sec: diophantine}, we introduce a way to encode points on the Apollonian gasket with words in 3 letters, analogous to cutting sequences for $\psl(2,\mathbb{Z})$. This provides a way to encode closed geodesics on $M_A$ as well. In \S\ref{sec: marking_crown}, we describe markings of crowns, and how to calculate core geodesics from the markings. In \S\ref{sec: cutting_from_marking}, we reconcile these two descriptions by explaining how to obtain the coding of a core geodesic from the marking. In the process, we also classify all markings of elliptic and parabolic elementary planes.

Being an elementary plane puts combinatorial restrictions on the coding of its core geodesic. In \S\ref{sec: single_crown_class} and \S\ref{sec: double_crown_class}, we leverage these restrictions to classify all markings of hyperbolic elementary planes.

With the list of elementary planes at hand, we prove our main results in \S\ref{sec: geom_top_elem_planes}. In \S\ref{sec: beyond}, we explain how insights from elementary planes help us understand the behavior of some nonelementary planes as well. Finally, in \S\ref{sec: questions}, we propose several directions of further research.

This paper is adapted from part of the author's PhD thesis. I want to thank my advisor C.~McMullen for his continued support, insightful discussions, and  helpful suggestions. I also want to thank T. Torkaman and the anonymous referee for providing many comments and suggestions that improved the exposition. I acknowledge the support of Max Planck Institute for Mathematics, where the paper was finalized.

\section{Elementary planes in acylindrical hyperbolic 3-manifolds}\label{sec: prelim}
In this section we discuss some generalities of elementary planes in a geometrically finite acylindrical hyperbolic 3-manifold $M=\Gamma\backslash\mathbb{H}^3$ with Fuchsian boundary, before focusing on the Apollonian orbifold $M_A$ in later parts. The main goal is to prove Propositions~\ref{prop: elementary_circle_gen} and \ref{prop: closed_downstairs}.

\subsection{Acylindrical manifolds, \`a la Thurston}\label{sec: acy_def}
In this subsection we give a brief introduction to acylindrical manifolds, following the definition in \cite{hyperbolization1}.
\paragraph{Pared manifolds}
A \emph{pared manifold} is a pair $(N,P)$ of a compact oriented 3-manifold with boundary $N$ and a submanifold $P\subset\partial N$ satisfying the following conditions:
\begin{itemize}[topsep=0mm, itemsep=0mm]
\item $P$ consists of incompressible tori and annuli;
\item Every torus component of $\partial N$ is contained in $P$;
\item Any cylinder 
$$f:(S^1\times[0,1],S^1\times\{0,1\})\to (N,P)$$
whose boundary $f(S^1\times\{0,1\})$ gives essential curves of $P$ is homotopic rel boundary into $P$.
\end{itemize}
See \cite{hyperbolization1} for a more general definition in dimension $n$. By a \emph{hyperbolic structure} on the pared manifold $(N,P)$, we mean the realization of $(N,P)$ as a complete, oriented hyperbolic 3-manifold $M\cong\Gamma\backslash\mathbb{H}^3$. More precisely, let $M_{\text{cusp}}$ be the union of disjoint cuspidal neighborhoods for all cusps. Then there is an orientation-preserving homotopy equivalence on pairs
$$\phi:(N,P)\to (M,M_{\text{cusp}}).$$
In other words, $P\subset \partial N$ gives the parabolic locus, and the components of $P$ are designated cusps.

\paragraph{Acylindricality}
A pared 3-manifold $(N,P)$ is said to be \emph{acylindrical} if $\partial_0N:=\partial N-P$ is incompressible and if every cylinder
$$f:(S^1\times[0,1],S^1\times\{0,1\})\to(N,\partial_0N)$$
whose boundary components $f(S^1\times\{0\}),  f(S^1\times\{1\})$ are essential curves of $\partial_0N$ can be homotoped rel boundary into $\partial N$.

Finally, a hyperbolic 3-manifold is said to be acylindrical if it gives a hyperbolic structure on an acylindrical pared 3-manifold.

\paragraph{Acylindricality from limit sets}
Recall that a complete hyperbolic 3-manifold $M$ is said to be geometrically finite if the unit neighborhood of its convex core has finite volume. When $M$ is geometrically finite of infinite volume, one can recognize acylindricality from its limit set $\Lambda$: $M$ is acylindrical if and only if any component of the domain of discontinuity $\Omega=S^2-\Lambda$ is a Jordan domain, and the closures of any pair of connected components share at most one point.

In particular, any finite manifold cover of the Apollonian orbifold $M_A$ is acylindrical. In \S\ref{sec: manifold_cover}, we exhibit some explicit covers and show directly they are acylindrical in the sense of Thurston.

\subsection{Elementary planes in acylindrical manifolds}\label{sec: elem_acylindrical}
In this subsection, let $M=\Gamma\backslash\mathbb{H}^3$ be a geometrically finite acylindrical hyperbolic 3-manifold with Fuchsian boundary (i.e.\ $\core(M)$ has totally geodesic boundary)\footnote{Some of the statements proved in this subsection still hold for those with quasifuchsian boundary with some caveats, which we will note as appropriate.}. Let $\Lambda$ be the limit set of $\Gamma$, and $\Omega$ its domain of discontinuity. Recall that given a geodesic plane $P$, we write $\core(P)=P\cap\core(M)$.

Recall that given a circle $C$, the corresponding geodesic plane $P$ in $M$ is the image of an isometric immersion $f:\hull(C)\cong\mathbb{H}^2\to M$. Assume $P$ is closed in $M$. Let $\Gamma^C$ be the stabilizer of $\hull(C)$ in $\Gamma$. Then $f$ factors through the map
$$\hat f:S=\Gamma^C\backslash\mathbb{H}^2\to M.$$
This map is generically one-to-one. Take $\hat S:=(\hat f)^{-1}(\core(M))\subset S$. Then $\hat S$ is a convex subsurface of $S$, and has finite volume. The image of $\hat S$ under $\hat f$ is the portion of $P$ inside the convex core, which we denote by $\core(P)$.

For simplicity, we have assumed that $\core(M)$ has totally geodesic boundary. Then $\partial\hat S$ are complete geodesics of $S$. Depending on the geometry of $S$, we have the following possibilities\footnote{When $\partial\core(M)$ is not totally geodesic, we have essentially the same picture, except that complete geodesics may be replaced by piecewise geodesics on $\partial\core(M)$ bent according to the bending lamination.}:
\begin{itemize}[topsep=0mm, itemsep=0mm]
\item If $S\cong\mathbb{H}^2$, then $\hat S$ is an ideal polygon;
\item If $S\cong\langle \gamma\rangle\backslash\mathbb{H}^2$ for a parabolic element $\gamma$, then $\hat S$ is a punctured ideal polygon;
\item If $S\cong\langle\gamma\rangle\backslash\mathbb{H}^2$ for a hyperbolic element $\gamma$, then $\hat S$ is a single crown, or a double crown (i.e. two crowns glued along their core geodesic of the same length);
\item If $S$ is a nonelementary surface, then $\hat S$ is the convex core of $S$, with crowns attached to some of the geodesic boundary components.
\end{itemize}
Note that the first three items in the list above correspond to elementary planes, as mentioned in the introduction. In these cases, $\hat S$ \emph{must} have spikes coming from cusps of $M$, by the acylindrical assumption.

The main result of this section is the following general version of Proposition~\ref{prop: elementary_circle}, relating the geometry of an elementary plane with the intersection of its boundary circle and the limit set:
\begin{prop}\label{prop: elementary_circle_gen}
Let $M=\Gamma\backslash\mathbb{H}^3$ be a geometrically finite acylindrical hyperbolic 3-manifold with Fuchsian boundary, and $\Lambda$ its limit set. Let $C$ be a circle in $S^2$ and $P$ the corresponding geodesic plane in $M$. Then $\core(P)$ is a properly immersed
\begin{enumerate}[label=\normalfont{(\arabic*)}, topsep=0mm, itemsep=0mm]
\item  ideal polygon if and only if  $1<|C\cap\Lambda|<\infty$;
\item punctured ideal polygon if and only if $C\cap\Lambda$ has a unique accumulation point;
\item single crown if and only if $C\cap\Lambda$ has exactly two accumulation points $p,q$, and $\Lambda$ only intersects one component of $C-\{p, q\}$;
\item double crown if and only if $C\cap\Lambda$ has exactly two accumulation points $p,q$, and $\Lambda$ intersects both components of $C-\{p, q\}$.
\end{enumerate}
Assume furthermore that $M$ covers an arithmetic hyperbolic 3-manifold. Then $P$ is elementary if and only if $C\cap\Lambda$ is countable and contains at least $3$ points.
\end{prop}

\paragraph{Elementary circles and elementary planes}
Note that one direction of Proposition~\ref{prop: elementary_circle_gen} is easy: it is immediately clear that any elementary circle intersects the limit set $\Lambda$ in countably many points. As a matter of fact, it is easy to see what $C\cap\Lambda$ looks like depending on what type the corresponding elementary plane $P$ belongs to, as described in Proposition~\ref{prop: elementary_circle_gen}.

Conversely, if $C\cap\Lambda$ consists of countably many points, it is not \emph{a priori} known if the corresponding plane $P$ is closed or not. Assuming $P$ is closed, we then immediately obtain the other direction in each case of Proposition~\ref{prop: elementary_circle_gen}. So to finish the proof, it suffices to show $P$ is indeed closed in each case. We start with the following lemma
\begin{lm}
Suppose $1<|C\cap\Lambda|<\infty$. Then the corresponding geodesic plane $P$ is closed in $M$.
\end{lm}
\begin{proof}
Clearly $C$ intersects finitely many components of the domain of discontinuity, say $\Omega_1,\ldots,\Omega_d$. By assumption, the points in $C\cap\Lambda$ are exactly where these components touch, so these are parabolic fixed points by acylindricality.

Suppose $\gamma_k\cdot C\to C'$. Note that we can choose a component $\Omega_j$ so that $\gamma_k\Omega_j=\gamma_1\Omega_j$ for all $k\ge1$ after passing to a subsequence. Otherwise, $\diam(\gamma_k\cdot\Omega_j)\to0$ for $j=1,\ldots,d$, and so $\diam(\gamma_k\cdot C)\to 0$ as well.

Without loss of generality, assume $j=1$ and $\gamma_1$ is the identity. Then $\gamma_k\in\Gamma_1:=\stab_\Gamma(\Omega_1)$. Note that $\Gamma_1$ is a quasifuchsian group. The circle $C$ is divided into two parts $C_1$ and $C_2$, giving curves from cusp to cusp on the two Riemann surfaces at infinity of the corresponding quasifuchsian manifold. It then follows that $\gamma_k$ is eventually the identity, for otherwise $\diam(\gamma_k\cdot C_1)\to 0$ and $\diam(\gamma_k\cdot C_2)\to 0$. Therefore $C'=C$, as desired.
\end{proof}
The proof essentially follows the proof of \cite[Thm.\ 1.5]{khalil2019geodesic}. Note that we do not need the components of $\Omega$ to be round circles, so we can relax the assumption of Proposition~\ref{prop: elementary_circle_gen} to allow any geometrically finite acylindrical hyperbolic 3-manifold (even with quasifuchsian boundary) in this case, with the following caveat: we need to require that each point of $C\cap\Lambda$ is a point where $C$ crosses the limit set, not just touches it but stays in the same component of $\Omega$, to avoid exotic behavior described in \cite{exotic_plane}.

Next we consider parabolic and hyperbolic elementary planes. It turns out the method we adopt makes use of the properties of round circles in an essential way, so we stick to the setting where $\core(M)$ has totally geodesic boundary. We have
\begin{lm}
If $C\cap\Lambda$ is countable and has finitely many accumulation points, then the corresponding plane $P$ is closed in $M$.
\end{lm}
\begin{proof}
The accumulation points divides $C$ into finitely many segments $C_1,\ldots, C_l$. For each segment $C_j$, it passes through only parabolic fixed points. In particular, $C_j$ forms the same angle with each component of the domain of discontinuity $\Omega$ it intersects. In particular, there are only finitely many choices of angles $C$ forms with components of $\Omega$.

Suppose $\gamma_k\cdot C\to C'$. Note that $C'$ must intersect at least three different components $\Omega_1,\Omega_2,\Omega_3$ of $\Omega$. Indeed, if $C'$ is contained in the closure of one component, then the angle $\gamma_k\cdot C$ forms with that component goes to 0, a contradiction. It is also impossible for $C'$ to intersect just two components, for otherwise the boundary circles of those two components touch at two different points.

By passing to a subsequence if necessary, we may assume every $\gamma_k\cdot C$ intersects $\Omega_1,\Omega_2,\Omega_3$ as well. We may further assume the angle $\gamma_k\cdot C$ forms with $\Omega_1$ is the same for all $k$, and similarly for $\Omega_2,\Omega_3$. But there are only finitely many circles satisfying this. In particular, $\gamma_k\cdot C$ becomes constant for $k$ sufficiently large, and thus $C'\in\Gamma\cdot C$ as desired.
\end{proof}
The two lemmas then give (1)-(4) of Proposition~\ref{prop: elementary_circle_gen}, as discussed above.

\paragraph{Arithmeticity and elementary planes}
In this part, we focus on a smaller family of geometrically finite acylindrical manifolds with Fuchsian boundary, i.e.\ those covering an arithmetic manifold.

We call a hyperbolic 3-manifold \emph{arithmetic} if the corresponding Kleinian group is an arithmetic subgroup of $\psl(2,\mathbb{C})$. The general definition of arithmeticity is slightly technical and not really relevant to our discussion, so we refer to references \cite{arithmetic, morris2001introduction}. We remark that every arithmetic subgroup of $\psl(2,\mathbb{C})$ is a lattice, and every arithmetic subgroup containing parabolic elements is commensurable to a Bianchi group, i.e.\ $\psl(2,\mathcal{O}_d)$ where $\mathcal{O}_d$ is the ring of integers in a quadratic imaginary number field $\mathbb{Q}(\sqrt{-d})$. For details, we refer to \cite{arithmetic}.

Let $M=\Gamma\backslash\mathbb{H}^3$ be a geometrically finite, acylindrical hyperbolic 3-manifold with Fuchsian boundary so that there is a covering map 
$$\pi:M\to N$$
of infinite order to an arithmetic hyperbolic 3-manifold $N=\Gamma_0\backslash\mathbb{H}^3$. Examples include the Apollonian orbifold, as we will show that the Apollonian group $\Gamma_A$ is contained in $\psl(2,\mathbb{Z}[i])$ in the next section. We have
\begin{prop}\label{prop: closed_downstairs}
A geodesic plane $P\subset M$ intersecting the convex core is closed in $M$ if and only if its image $\pi(P)$ is closed in the arithmetic manifold $N$ under the covering map $\pi:M\to N$.
\end{prop}
\begin{proof}
One direction is purely topological and clear. For the other direction, suppose $P$ is closed in $M$. In $S^2$, let $C$ be a $\Gamma$-boundary circle of $P$. Let $\Gamma^C$ be the stabilizer of $C$ in $\Gamma$, and $\Gamma^C_0$ its stabilizer in $\Gamma_0$.

First assume $\Gamma^C$ is nonelementary. Then $\Gamma_0^C$ is also nonelementary, and in particular Zariski dense in $G_C$, the stabilizer of $C$ in $G=\psl(2,\mathbb{C})$. As $\Gamma_0^C$ is defined over a number field, so is $G_C$. Moreover, $\Gamma_0^C$ is simply the integer points of $G_C$, hence a lattice by Borel and Harish-Chandra \cite{arithmetic_lattice}.

Next assume $\Gamma^C$ is elementary. This is only possible when $\Gamma_0$ is commensurable to some $\psl(2,\mathcal{O}_d)$, so we may as well assume $\Gamma_0=\psl(2,\mathcal{O}_d)$. Note that by our discussion before, $C$ must pass through at least $3$ parabolic fixed points, and these are precisely points in $\mathbb{Q}(\sqrt{-d})$. The stabilizer of such a circle in $\psl(2,\mathcal{O}_d)$ is a lattice in $G_C$, see \cite[Cor.~9.6.4]{arithmetic}.
\end{proof}
The argument of this proof goes along the same lines as \cite[\S12.1]{acy_geom_finite}, except here there is the possibility of elementary surfaces. In their case, $\Gamma^C$ is automatically nonelementary when $P$ is closed. As mentioned in the introduction, this partly motivates the study of elementary planes.

We notice that if $C$ intersects the limit set $\Lambda$ of $\Gamma$ in countably many points (and crosses $\Lambda$), then it must give an elementary plane. Indeed, since $C\cap\Lambda$ is closed and countable, it must contain at least 3 isolated points. This is only possible when $N$ is commensurable to $\psl(2,\mathcal{O}_d)$, so we may assume these isolated points lie in $\mathbb{Q}(\sqrt{-d})$. Again by \cite[Cor.~9.6.4]{arithmetic}, this circle gives a closed plane even in $N$, and of course in $M$. This gives the final piece of Proposition~\ref{prop: elementary_circle_gen}.

Finally, we remark that the arguments for this part largely work even when $\partial\core(M)$ is not totally geodesic, with the same caveat from the last section that a point where $C$ touches but not crosses the limit set may lead to exotic behavior.

\section{The Apollonian gasket: group and geometry}\label{sec: apollo_group}

In this section, we collect some basic facts about the Apollonian gasket $\mathcal{A}$, the Apollonian group $\Gamma_A$ and the corresponding orbifold $M_A$.

\subsection{The Apollonian group}\label{sec: apollonian_group}

We first collect some basic facts about the Apollonian group $\Gamma_A$. This group has been investigated in detail in \cite{graham2005apollonian}, although there the authors focus on its action on Descartes configurations and view it as a subgroup of $\so(3,1)$.

First we have the following lemma:
\begin{lm}\label{lm: stab_of_H}
The stabilizer of the real line $\mathbb{R}$ in $\Gamma_A$ is $\psl(2,\mathbb{Z})$.
\end{lm}
\begin{proof}
First note that $\psl(2,\mathbb{Z})\subset\stab_{\Gamma_A}(\mathbb{R})$. Indeed, for any $\gamma\in \psl(2,\mathbb{Z})$, $\gamma$ fixes $y=0$ and rearranges circles tangent to $y=0$. Since $\mathcal{A}$ is uniquely determined by a triple of mutually tangent circles, $\gamma$ fixes $\mathcal{A}$ as well. Conversely, given $\gamma\in\stab_{\Gamma_A}(\mathbb{R})$, there exists $\eta\in \psl(2,\mathbb{Z})$ so that $\eta\gamma$ fixes $\infty,0,1$ and hence $\eta\gamma=1$.
\end{proof}

\begin{prop}
$\Gamma_A$ is generated by $\Gamma=\psl(2,\mathbb{Z})$ and $J=\begin{pmatrix}i&-1\\0&-i\end{pmatrix}$. In particular $\Gamma_A\subset \psl(2,\mathbb{Z}[i])$.
\end{prop}
\begin{proof}
It suffices to show that the group generated by $\Gamma$ and $J$ acts transitively on the circles in $\mathcal{A}$. Indeed, $J(\mathbb{R})$ is the line $y=-1$, and $\Gamma\cdot J(\mathbb{R})$ consists of circles tangent to the real axis. For any of the circle already obtained, say $\gamma(\mathbb{R})$ with $\gamma=\eta J$ for some $\eta\in\Gamma$, the orbit of $\mathbb{R}$ under $\gamma\Gamma\gamma^{-1}$ consists of all circles tangent to $\gamma(\mathbb{R})$. Repeating in this manner, we can obtain all the circles in $\mathcal{A}$.
\end{proof}

In particular, $\Gamma_A$ is discrete, and its limit set is precisely $\mathcal{A}$. We have the following immediate corollary:
\begin{cor}
The group $\Gamma_A$ acts transitively on the circles in $\mathcal{A}$, as well as the points of tangency in $\mathcal{A}$. In particular, any points of tangency in $\mathcal{A}$ is a Gaussian rational.
\end{cor}
\begin{proof}
The first part is clear from the proof of the previous proposition. Since $\psl(2,\mathbb{Z})$ acts on all the circles tangent to $\mathbb{R}$, the second part follows. Since $\infty$ is a point of tangency, $\Gamma_A\cdot\infty$ gives all the points of tangency. The last part then follows from $\Gamma_A\subset\psl(2,\mathbb{Z}[i])$.
\end{proof}
We remark that the converse is also true: every point in $\mathcal{A}\cap\mathbb{Q}(i)$ is a point of tangency; see Lemma~\ref{lm: rational_points}.

Since $\psl(2,\mathbb{Z})$ is finitely generated, $\Gamma_A$ is as well. There are many different choices of generators, but the following set of three parabolic elements will be particularly useful in later sections:
\begin{cor}
$\Gamma_A$ is generated by three parabolic elements
$$V_1=\begin{pmatrix}1&1\\0&1\end{pmatrix}, V_2=\begin{pmatrix}1&0\\1&1\end{pmatrix}, V_3=JV_2^{-1}J^{-1}=\begin{pmatrix}1-i&1\\1&1+i\end{pmatrix}.$$
\end{cor}
\begin{proof}
This follows from the fact that $V_1$ and $V_2$ generate $\psl(2,\mathbb{Z})$, and $V_3V_2^{-1}V_3=-J$.
\end{proof}

\subsection{The Apollonian orbifold}\label{sec: apollonian_orbifold}

In this subsection we study the geometry and topology of the Apollonian orbifold $M_A$, by giving a model for its convex core. This model is constructed by gluing faces of hyperbolic polyhedra. We remark that this is a common strategy for producing hyperbolic 3-manifolds with desired properties, see e.g.\ \cite[\S 3.3]{thurston_book}, \cite{poly1}, \cite{poly2}, \cite{gaster}, and \cite{double_lunchbox}.

\paragraph{The regular ideal hyperbolic octahedron}
The hyperbolic polyhedron of particular interest for our purpose here is the regular ideal hyperbolic octahedron. Explicitly, let $O$ be the ideal polyhedron with vertices at $\infty, 0, 1, 1-i, -i, 1/2-i/2$; see Figure~\ref{fig: octahedron}.
\begin{figure}[htp]
\centering
\includegraphics[width=0.5\linewidth]{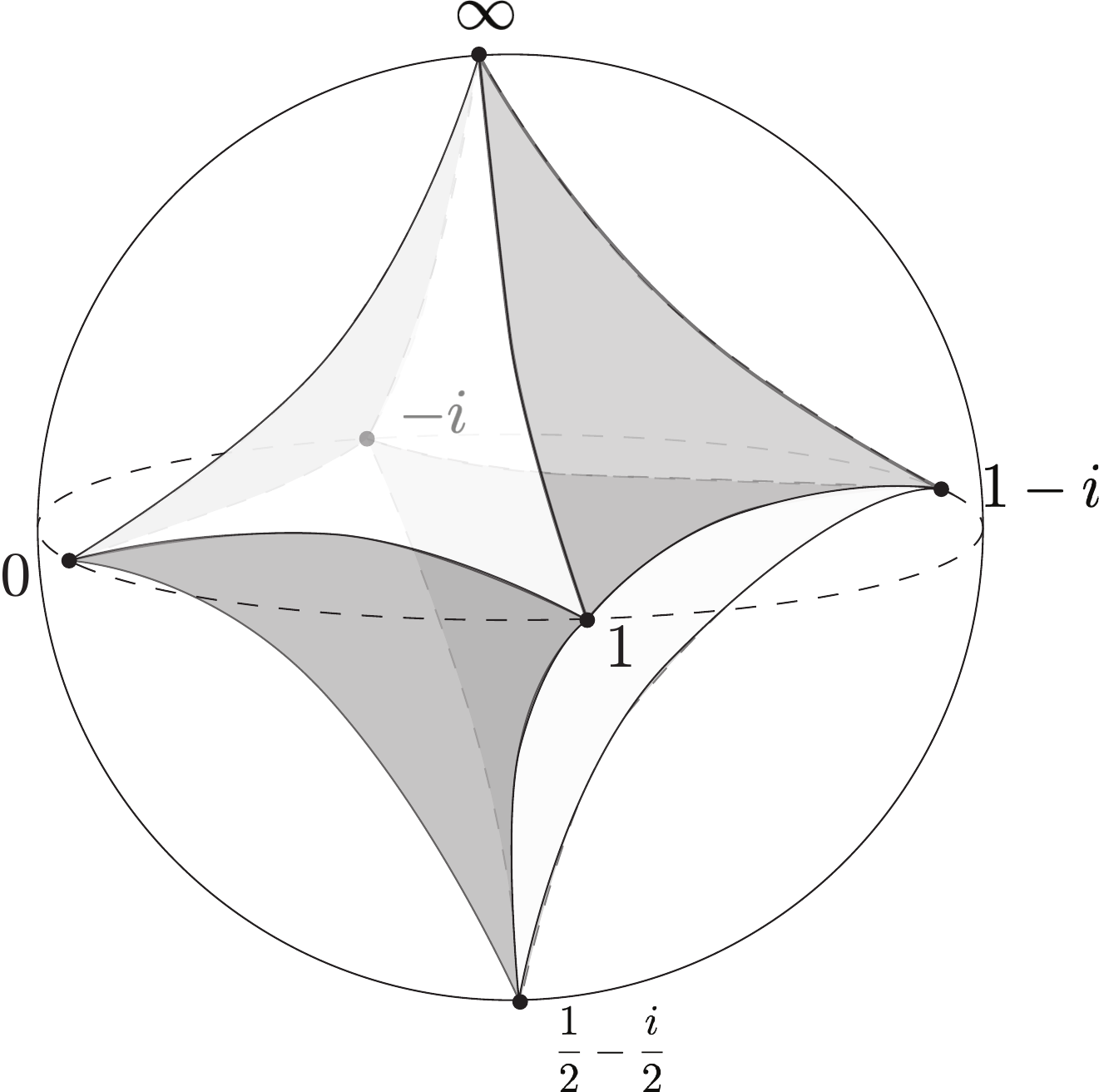}
\caption{The regular ideal hyperbolic octahedron $O$}
\label{fig: octahedron}
\end{figure}
Each face of $O$ is an ideal triangle; see Table~\ref{tab: faces} for the 8 triples of vertices and the circles at infinity of the faces they determine. All dihedral angles are $\pi/2$.

\begin{table}[htp]
\centering
\begin{tabular}{lll|l}
\hline
\multicolumn{3}{l|}{Vertices}&Circle at infinity\\
\hline
$0,$ & $1,$ & $\infty$ & $y=0$\\
$1,$ & $1-i,$ & $\infty$ & \cellcolor{gray!30}$x=1$\\
$1-i,$ & $-i,$ & $\infty$ & $y=-1$\\
$-i,$ & $0,$ & $\infty$ & \cellcolor{gray!30}$x=0$\\
$1-i,$ & $1,$ & $1/2-i/2$ & $(x-1)^2+(y+1/2)^2=1/4$\\
$1,$ & $0,$ & $1/2-i/2$ & \cellcolor{gray!30}$(x-1/2)^2+y^2=1/4$\\
$0,$ & $-i,$ & $1/2-i/2$ & $x^2+(y+1/2)^2=1/4$\\
$-i,$ & $1-i,$ & $1/2-i/2$ & \cellcolor{gray!30}$(x-1/2)^2+(y+1)^2=1/4$
\end{tabular}
\caption{Faces of $O$}
\label{tab: faces}
\end{table}

Note that each circle in Table~\ref{tab: faces} is either a circle in $\mathcal{A}$, or perpendicular to the circles it intersects in $\mathcal{A}$. We may color the latter type of circles and the corresponding faces gray. This gives a checkerboard coloring of $O$.

\paragraph{A model for $\core(M_A)$}
The polyhedron $O$ has many symmetries preserving the coloring. For example, it is symmetric across the planes determined by $x=1/2$, $y=-1/2$, or $(x-1)^2+y^2=1$. These three planes together with two faces of $O$ determined by $y=0$ and $x=1$ bound a slice $O'$ of $O$; see Figure~\ref{fig: slice_of_o}.

\begin{figure}[htp]
\centering
\includegraphics{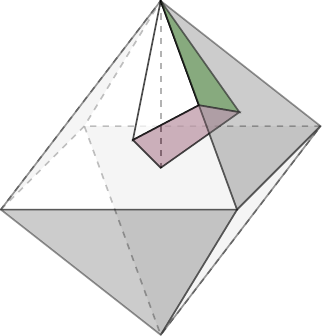}
\caption{A slice $O'$ of $O$}
\label{fig: slice_of_o}
\end{figure}

The dihedral angles of $O'$ are either $\pi/2$ or $\pi/3$, so we obtain a discrete subgroup of $\isom(\mathbb{H}^3)$, generated by reflections in all faces except the one coming from a white face of $O$. The index two subgroup of orientation preserving elements is then a Kleinian group, which we denote by $\Gamma_{O'}$. We claim:
\begin{lm}
The two Kleinian groups $\Gamma_A$ and $\Gamma_{O'}$ are identical.
\end{lm}
\begin{proof}
It is easy to see that $\Gamma_{O'}\supset\psl(2,\mathbb{Z})$. It is also not hard to see that one of the order $2$ elements contained in $\Gamma_{O'}$ is $J$. Finally, $\Gamma_{O'}$ preserves the Apollonian gasket $\mathcal{A}$.
\end{proof}

Let $\widetilde{O'}$ be the polyhedron (of infinite volume) obtained from $O'$ by extending across the white face to infinity. Two copies of $\widetilde{O'}$ with corresponding faces identified give a model for $M_A\cong\Gamma_A\backslash\mathbb{H}^3$. Its convex core is obtained from two copies of $O'$ by identifying corresponding faces except the white one. Thus $\core(M_A)$ has finite volume, and has totally geodesic boundary. That is, $M_A$ is geometrically finite with Fuchsian boundary.

This also gives the following topological model of $M_A$ as an orbifold, shown in Figure~\ref{fig: orbifold_model}.
\begin{figure}[htp]
\centering
\captionsetup{width=.6\linewidth}
\includegraphics[width=0.4\linewidth]{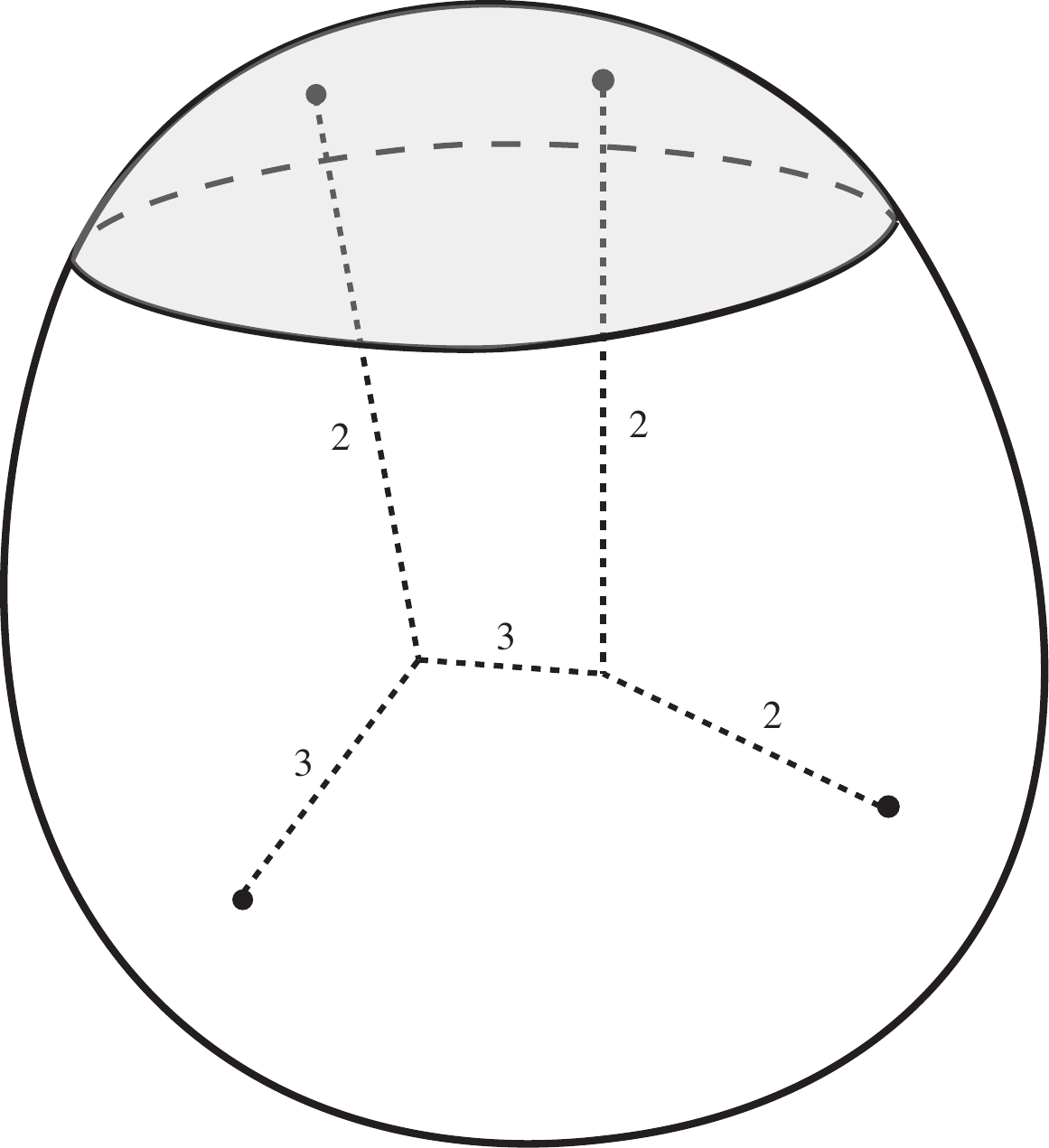}
\caption[A topological model for $M_A$]{A topological model for the orbifold $M_A$. The shaded region corresponds to the unique cusp of $M_A$}
\label{fig: orbifold_model}
\end{figure}

\section{Diophantine approximation on the Apollonian gasket}\label{sec: diophantine}
In this section we introduce a way to approximate points on $\mathcal{A}$ with its rational points in $\mathbb{Q}(i)$, analogous to Diophantine approximation of real numbers with convergents of their continued fractions. Along the way, we also introduce a way to encode points on $\mathcal{A}$ with words in $3$ letters, analogous to cutting sequences for real numbers with respect to the modular group $\psl(2,\mathbb{Z})$.

\subsection{Farey triangulation and Diophantine approximation}\label{sec: farey_diophantine}

In this subsection we recall the basics of continued fraction and Diophantine approximation on $\mathbb{R}$, and in particular highlight their connection to the modular group $\psl(2,\mathbb{Z})$. The Diophantine approximation on the Apollonian gasket to be introduced in the next section is very much inspired by these classical ideas. References include \cite{continued_fraction, continued_fraction_book}.

\paragraph{Farey triangulation}
In the upper half plane $\mathbb{H}$, consider the ideal triangle $\mathcal{F}$ with vertices $0,1,\infty$. The orbits of $\mathcal{F}$ under $\psl(2,\mathbb{Z})$ gives a triangulation of $\mathbb{H}$, called the \emph{Farey triangulation}. A pair of rational numbers $\alpha,\beta\in\mathbb{Q}\cup\{\infty\}$ are called \emph{Farey neighbors} if they form two of the three vertices of an ideal triangle in the Farey traingulation. If we write $\alpha=p/q$ and $\beta=r/s$ in lowest terms\footnote{To make this representation unique, we always require $q,s\ge0$ when $\alpha,\beta\neq\infty$, and set $\infty=1/0$.}, then $\alpha,\beta$ are Farey neighbors if and only if $ps-qr=\pm1$. In particular if $\alpha=\infty$, then $\beta\in\mathbb{Z}$.

Given a pair of \emph{finite} Farey neighbors $\alpha=p/q,\beta=r/s$ with $\alpha<\beta$, we have $ps-qr=-1$. There exists a unique rational number $\gamma$ in the interval $(\alpha,\beta)$ such that $\alpha,\beta,\gamma$ form the vertices of an ideal triangle in the triangulation. It turns out that $\gamma=(p+r)/(q+s)$. We call $\gamma$ the \emph{Farey midpoint} of $\alpha,\beta$. Conversely, given $\gamma\in\mathbb{Q}\backslash\mathbb{Z}$, there exists a unique pair $\alpha,\beta\in\mathbb{Q}$ so that $\alpha<\beta$ are Farey neighbors, and that $\gamma$ is the Farey midpoint of $\alpha$ and $\beta$. We will call $\alpha$ (resp. $\beta$) the left (resp. right) Farey neighbor of $\gamma$. For $\gamma\in\mathbb{Z}$, we adopt the following convention: when $\gamma>0$, take $\alpha=\gamma-1$ and $\beta=\infty$ and when $\gamma\le0$, take $\alpha=-\infty$ and $\beta=\gamma+1$.

\paragraph{Cutting sequences and continued fractions}
Given $x\in\mathbb{R}_{>0}$, we define the \emph{cutting sequence} of $x$ as follows. First draw the unique (oriented) geodesic $l_x$ based at $i\in\mathbb{H}$ towards $x$. For each ideal triangle the geodesic $l_x$ passes through, either (a) $l_x$ ends at one of the vertex, implying $x\in\mathbb{Q}$, or (b) $l_x$ cuts two sides of the ideal triangle, and the common vertex of these two sides lies on the left ($L$) or right ($R$) side of $l_x$. The cutting sequence is then the (finite or infinite) sequence of $L$ and $R$'s representing the position the vertex described in (b) as $l_x$ passes through each ideal triangle it intersects.

Given $x\in\mathbb{R}_{>0}$, we can also calculate its continued fraction $[a_0; a_1,\ldots]$ so that $x=a_0+\frac{1}{a_1+\frac{1}{\cdots}}$. It is an easy exercise that a real number $x\in(0,1)$ with continued fraction $[0; a_1,a_2,\ldots]$ has cutting sequence $R^{a_1}L^{a_2}\cdots$, and a real number $x\in(1,\infty)$ with continued fraction $[a_0;a_1,a_2,\ldots]$ has cutting sequence $L^{a_0}R^{a_1}L^{a_2}\cdots$. We remark that both cutting sequence and continued fraction can be defined for negative numbers as well.

It is sometimes useful to consider cutting sequences obtained from geodesics based at points other than $i$. If two cutting sequences obtained from different base points represents the same positive irrational number, then they eventually agree.

Let $\tau$ be an oriented closed geodesic on the orbifold $X:=\psl(2,\mathbb{Z})\backslash\mathbb{H}^2$, and let $l_\tau$ be a lift of $\tau$ in $\mathbb{H}^2$. As $l_\tau$ is stabilized by an element of $\psl(2,\mathbb{Z})$, the cutting sequence obtained by choosing an arbitrary point on $l_\tau$ as the base point is eventually periodic. The periodic part of the sequence is independent of the choice of base point or the choice of the lift. We call the periodic part, represented by a finite word $w(L,R)$ in $L$ and $R$, the cutting sequence of the oriented closed geodesic $\tau$, and write $\tau=\overline{w(L,R)}$.

\paragraph{Diophantine approximation}
Given a real number $x>0$, let $w(x)$ be its cutting sequence, and $w_n(x)$ be its initial part of length $n$. Let $m_n(x)$ be the matrix obtained by substituting $L$ with $V_1=\begin{pmatrix}1&1\\0&1\end{pmatrix}$ and $R$ with $V_2=\begin{pmatrix}1&0\\1&1\end{pmatrix}$. Then we have (see e.g.~\cite{continued_fraction}):
\begin{prop}[Diophantine approximation]
If $x$ is irrational, $m_n(x)\cdot z_0\to x$ for any $z_0\in\{z:\re(z)\ge0\}\cup\{\infty\}$. In particular, if we take $z_0=0$ or $\infty$, we obtain a sequence of rational numbers approximating $x$.
\end{prop}
This can be carried out using continued fractions. For concreteness, take $x\in(0,1)$, and so $x=[0;a_1,a_2,\ldots]$. One can easily check that (after substituting $L$ with $V_1$ and $R$ with $V_2$):
\begin{itemize}[topsep=0mm, itemsep=0mm]
\item $R^{a_1}L^{a_2}\cdots R^{a_{2k-1}}\cdot \infty=[0;a_1,a_2,\ldots,a_{2k-1}]$;
\item $R^{a_1}L^{a_2}\cdots R^{a_{2k-1}}L^{a_{2k}}\cdot 0=[0;a_1,a_2,\ldots,a_{2k}]$.

\end{itemize}
Set $p_k/q_k:=[0;a_1,\ldots,a_k]$, we then have $p_k/q_k\to x$.  Moreover, we observe that $p_n/q_n$ and $p_{n+1}/q_{n+1}$ are Farey neighbors, and $x$ lies between $p_n/q_n$ and $p_{n+1}/q_{n+1}$ on the real line. Since
$$\left|\frac{p_n}{q_n}-\frac{p_{n+1}}{q_{n+1}}\right|=\frac{|p_nq_{n+1}-q_np_{n+1}|}{q_nq_{n+1}}=\frac{1}{q_nq_{n+1}},$$
we conclude that $|x-p_n/q_n|<1/(q_nq_{n+1})$. This is the classical Dirichlet's Approximation Theorem (see e.g.~\cite{continued_fraction_book}). These are also \emph{best} rational approximations in some sense; we refer to \cite{continued_fraction_book} for detailed expositions.

\subsection{Diophantine approximation on the Apollonian gasket}\label{sec: diophantine_apollonian}
Diophantine approximation on the Apollonian gasket to be detailed here is the restriction of that for $\psl(2,\mathbb{Z}[i])$, developed by Schmidt \cite{diophantine}, to the subgroup $\Gamma_A$. However, as mentioned before, it is a very natural extension of the ideas detailed in the last subsection.

Recall that the Apollonian group $\Gamma_A$ is generated by three parabolic elements $V_1,V_2,V_3$. The action of these elements can be easily seen on the octahedron $O$: $V_1$ maps $(0,-i,\infty)$ to $(1,1-i,\infty)$, $V_2$ maps $(0,-i,\infty)$ to $(0,1/2-i/2,1)$, and $V_3$ maps $(0,-i,\infty)$ to $(1/2-i/2,-i,1-i)$. In the spirit of cutting sequences, note that the orbits of $O$ under $\Gamma_A$ gives a tessellation of the convex hull of $\mathcal{A}$. A geodesic ray ending in a point in $\mathcal{A}$ enters each ideal octahedron in a gray face, and leaves via one of the three other gray faces, corresponding to the actions of $V_1,V_2,V_3$. We can thus define a cutting sequence to encode the exit pattern along the geodesic ray, in three letters ($V_1,V_2,V_3$) instead of two ($L,R$). We now formalize this idea.

\paragraph{A triadic subdivision of $\mathcal{A}$}
Consider the right half plane $\mathbb{H}_*:=\{z:\re(z)>0\}$ bordered by the extended  imaginary axis $i\mathbb{R}\cup\{\infty\}$. It is easy to see that $V_1\cdot\mathbb{H}_*$ is the region $y>1$, $V_2\cdot\mathbb{H}_*$ is the disk $(x-1/2)^2+y^2<1/4$, and $V_3\cdot\mathbb{H}_*$ is the disk $(x-1/2)^2+(y+1)^2<1/4$. 

Note that $\mathcal{A}\cap\overline{\mathbb{H}_*}$ is contained in a triangular region with vertices $0,-i,\infty$, and $\mathcal{A}\cap\overline{\mathbb{H}_*}=\cup_{k=1}^3(\mathcal{A}\cap V_k(\overline{\mathbb{H}_*}))$. Moreover, each $\mathcal{A}\cap V_k(\overline{\mathbb{H}_*})$ is again contained in a triangular region with vertices in $\mathbb{Q}(i)\cup\{\infty\}$; see Figure~\ref{fig: triadic}

\begin{figure}[htp]
\centering
\includegraphics[width=0.9\textwidth]{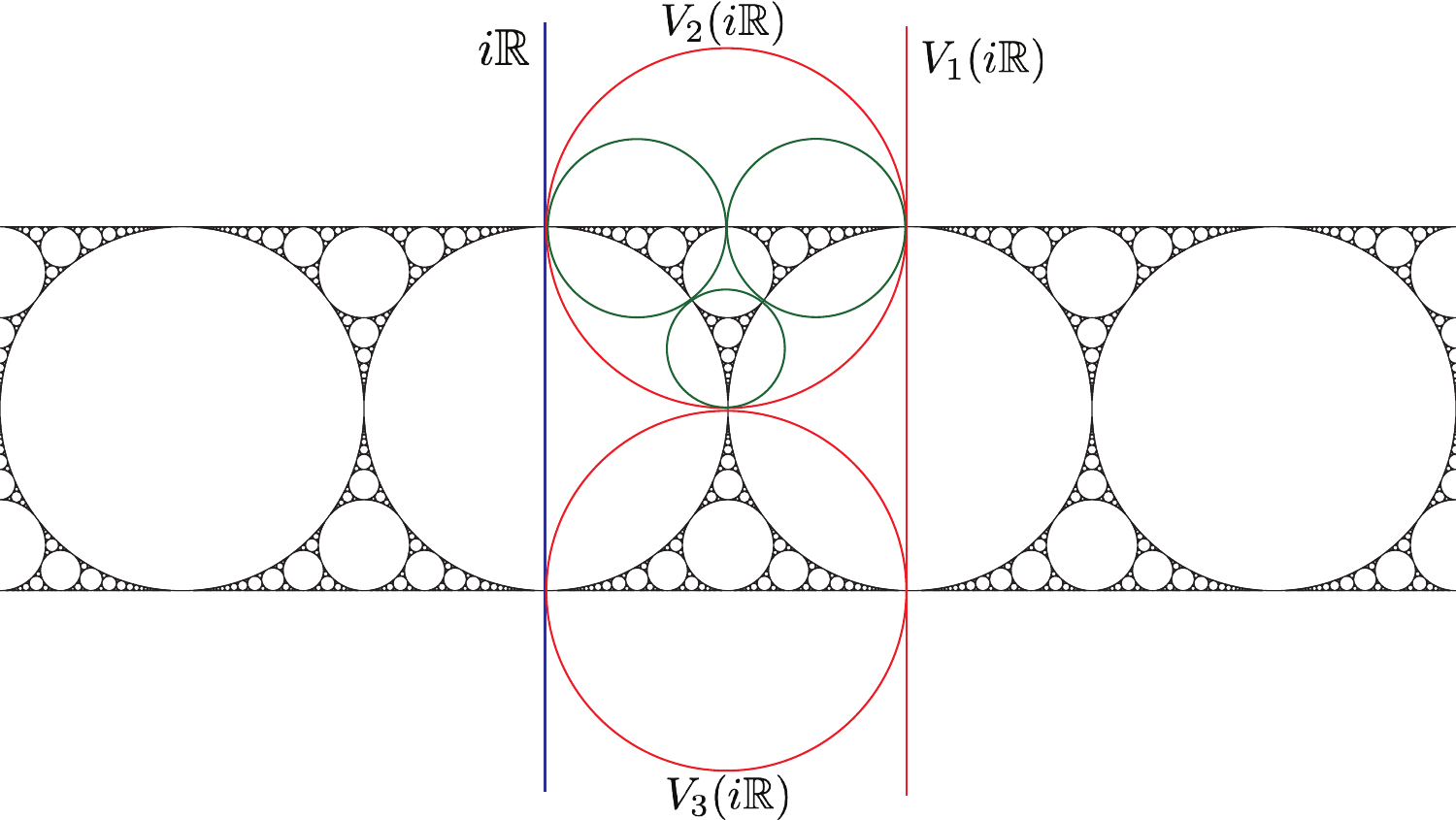}
\caption{Triadic subdivision of $\mathcal{A}$}
\label{fig: triadic}
\end{figure}

Let $\mathcal{W}_n$ be the collection of words in $V_1,V_2,V_3$ of length $n$. We also let $w$ denote the matrix calculated from $w$ by plugging in the three matrices $V_1,V_2,V_3$ represent. For any $w\in\mathcal{W}_n$, set $D_w=w\cdot\overline{\mathbb{H}_*}$ and $\mathcal{A}_w=\mathcal{A}\cap D_w$. Inductively, we note that $\mathcal{A}_w$ is contained in a triangular region bounded by three circular arcs tangent to each other, with vertices in $\mathbb{Q}(i)\cup\{\infty\}$. For simplicity, we will call them vertices of $\mathcal{A}_w$.
\begin{prop}\label{prop: triadic}
Let $\mathcal{A}_n=\{\mathcal{A}_w:w\in\mathcal{W}_n\}$. Then
\begin{enumerate}[label=\normalfont(\arabic*), topsep=0mm, itemsep=0mm]
\item $\mathcal{A}_n$ gives a subdivision of $\mathcal{A}\cap\overline{\mathbb{H}_*}$;
\item Given $w,w'\in\mathcal{W}_n$, $w\neq w'$,  $\mathcal{A}_w\cap \mathcal{A}_{w'}$ is either empty or a common vertex. Moreover, $\mathcal{A}_w\cap \mathcal{A}_{w'}\neq\emptyset$ if and only if there exists $j\neq k$ and a word $w''$ of length $n'<n$ so that $w=w''V_jV_k^t$ and $w'=w''V_kV_j^t$, where $t=n-n'-1$. In this case $\mathcal{A}_w\cap \mathcal{A}_{w'}=w''\cdot z$, where $z=1,1-i,1/2-1/2i$ when $\{j,k\}=\{1,2\},\{1,3\},\{2,3\}$ respectively;
\item\label{item: triadic_finer_division} $\mathcal{A}_{n+1}$ gives a finer subdivision, where each element of $\mathcal{A}_w\in\mathcal{A}_n$ is subdivided into $\cup_{k=1}^3\mathcal{A}_{wV_k}$. For $k=1,2,3$, $\mathcal{A}_{wV_k}$ shares a distinct vertex $z_{w,k}$ with $\mathcal{A}_w$. Suppose $w=w'V_j$, then $z_{w,j}=z_{w',j}$ and $z_{w,k}=\mathcal{A}_{w}\cap \mathcal{A}_{w'V_k}$ for $k\neq j$.
\end{enumerate}
\end{prop}
The proofs are immediate. Note that \ref{item: triadic_finer_division} gives an labelling of the vertices of $A_w$ and its subdivisions by $V_k$, and moreover, provides an inductive way to determine the labelling. Finally, note that we can also start with the region $\mathbb{H}_t=\{x+iy:y>-t\}$ for a positive integer $t$ to obtain subdivisions and orderings in the left-half plane. Equivalently, we allow words of the form $V_1^{-t}w$ for some $w\in\mathcal{W}_n$ not starting with $V_1$.

\paragraph{Vertices, Farey neighbors, and rational points}
As mentioned above, any vertex is a Gaussian rational. The converse is also true:
\begin{lm}\label{lm: rational_points}
Let $z\in\mathcal{A}$. Then $z$ is a vertex of some $\mathcal{A}_w$ if and only if $z\in\mathbb{Q}(i)\cup\{\infty\}$.
\end{lm}
\begin{proof}
When $z\in\mathbb{Q}(i)$, write $z=p/q$ with $p,q\in\mathbb{Z}[i]$ and relatively prime (we will call this \emph{in lowest terms}). By applying $V_1$ or its inverse, we may assume $z\in D_{V_2}$ or $D_{V_3}$. Suppose $z$ lies in the interior of $D_{V_2}$, then $	V_2^{-1}(z)=p/(q-p)$ lies in $\mathbb{H}_*$. Note that $|p/q|<1$ and so $|p|<|q|$. Also $|q-p|<|q|$, as $|p/q-1|^2<1$. Similarly, if $z$ lies in the interior of $D_{V_3}$, the maximum of the norms of the numerator and denominator decreases if one applies $V_3^{-1}$ to $z$. So eventually we must have $z$ lies on the boundary of $D_{V_2}$ or $D_{V_3}$, and is thus a vertex.

Alternatively, any rational point in $\mathcal{A}$ represents a lift of the unique cusp of $\psl(2,\mathbb{Z}[i])\backslash\mathbb{H}^3$. So it must be a lift of the unique cusp of $M_A$ upstairs as well. This implies that every rational point in $\mathcal{A}$ is a point of tangency of circles in $\mathcal{A}$, which are precisely vertices.
\end{proof}

The circles $w(i\mathbb{R})$ also divide the complement $S^2\backslash\mathcal{A}$ into ideal triangles, \`a la Farey triangulation. Analogously, a pair of rational numbers in $\mathcal{A}$ are called \emph{Farey neighbors} if they form two of the three vertices of an ideal triangle. It is clear that two numbers are Farey neighbors if and only if they form two of the three vertices of some $\mathcal{A}_w$.

Given $z,z'\in\mathcal{A}\cap\mathbb{Q}(i)$, write $z=p/q$ and $z'=r/s$ in lowest terms. Then they are Farey neighbors if and only if $ps-qr=\pm1$ or $\pm i$. The proof is immediate: indeed, there exists $w\in \Gamma_A$ so that $\{w(0),w(\infty)\}=\{z,z'\}$.

Suppose $z$ and $z'$ are Farey neighbors. Then they are two of the three vertices of some unique $\mathcal{A}_w$. To find the third vertex, we say the quadruple of Gaussian integers $(p,q,r,s)$ is \emph{positively presented} if $ps-qr=1$ and
\begin{itemize}[topsep=0mm, itemsep=0mm]
\item $\realpart(\bar q s)\ge0$, $\impart(\bar q s)\ge0$ when $\bar q s\neq0$;
\item $(p,q,r,s)=(1,0,n,1)$ or $(1+in, i, i, 0)$ for some integer $n$ when $\bar q s=0$.
\end{itemize}
\begin{lm}\label{lm: farey_triple}
Let $z, z'$ be a pair of Farey neighbors, and suppose they are two of the vertices of $\mathcal{A}_w$. Then
\begin{enumerate}[label=\normalfont(\arabic*), topsep=0mm, itemsep=0mm]
\item There exists a positively presented quadruple $(p,q,r,s)$, unique up to multiplication by $-1$, so that $\{z,z'\}=\{p/q, r/s\}$;
\item Given the positive presentation, the third vertex of $\mathcal{A}_w$ is given by $z''=\dfrac{-i p+r}{-i q+s}$;
\item The circle passing through $z, z', z''$ has radius $1/\realpart(\bar q s)$.
\end{enumerate}
\end{lm}
\begin{proof}
Clearly we may assume $z,z'\in\overline{\mathbb{H}_*}$. Then $z,z'$ is the image under $w$ of two of $\infty,0,-i$. Here $w$ is a word in $V_1$, $V_2$, and $V_3$. We prove (1) as well as the following statement by induction on the length of $w$: the quadruple $(p,q,r,s)$ in addition satisfies $|q|^2-\impart(\bar qs)\ge0$ and $|s|^2-\impart(\bar q s)\ge0$.

Let $S=\begin{pmatrix}0&-i\\-i&1\end{pmatrix}$. Note that $S(0)=-i$, $S(-i)=\infty$ and $S(\infty)=0$. Moreover, note that if $w(\infty)=p/q$ and $w(0)=r/s$ where $ps-qr=1$, then $w=\pm\begin{pmatrix}p&r\\q&s\end{pmatrix}$. It is then sufficient to show that
\begin{enumerate*}[label=(\alph*)]
\item the statements hold for the identity matrix;
\item if the statements hold for $w$, then they also hold for $wV_i, i=1,2,3$ and $wS$.
\end{enumerate*}
Part (a) is trivial, and we only verify the statements for $wV_3$; the others are similar.

Indeed, we have
$$wV_3=\begin{pmatrix}p&r\\q&s\end{pmatrix}\begin{pmatrix}1-i&1\\1&1+i\end{pmatrix}=\begin{pmatrix}p-pi+r&p+r+ri\\q-qi+s&q+s+si\end{pmatrix},$$
and so
$$(\overline{q-qi+s})\cdot(q+s+si)=|q|^2+|s|^2+i|q|^2+i|s|^2+2\bar qs i+\bar sq$$
whose real and imaginary parts are
$$|q|^2+|s|^2+\realpart(\overline sq)-2\impart(\bar qs),\quad |q|^2+|s|^2+2\realpart(\bar qs)-\impart(\bar q s)$$
respectively. Both are nonnegative by assumption. Moreover,
\begin{align*}
|q-qi+s|^2&=2|q|^2+|s|^2+2\realpart(\bar q s)-2\impart(\bar qs)\ge |q|^2+|s|^2+2\realpart(\bar qs)-\impart(\bar q s)\\
|q+s+si|^2&=|q|^2+2|s|^2+2\realpart(\bar qs)-2\impart(\bar qs)\ge |q|^2+|s|^2+2\realpart(\bar qs)-\impart(\bar q s)
\end{align*}
so we are done for this case.

For (2), given the positive presentation from (1), the third vertex is given by $\begin{pmatrix}p&r\\q&s\end{pmatrix}.(-i)$.

For (3), given $z=p/q, z'=r/s$ with positive presentation it is easy to check that
$$|z-z'|=\frac1{|qs|}, |z'-z''|=\frac1{|s(-iq+s)|}, |z-z''|=\frac1{|q(-iq+s)|}$$
Given the lengths of the three sides, it is then a routine exercise to calculate the radius of the circumscribed circle of a triangle.
\end{proof}

\paragraph{Cutting sequences and Diophantine approximation}
Given a point $z\in\mathcal{A}$, we define its cutting sequence as follows. First assume $\re(z)\ge0$. For each $n$, either $z$ belongs to a distinct $\mathcal{A}_w\in\mathcal{A}_n$, or $\{z\}=\mathcal{A}_w\cap \mathcal{A}_{w'}\subset\mathbb{Q}(i)\cup\{\infty\}$. We denote $w_n(z)=w$. In the latter case, we may choose $w_n(z)$ to be either $w$ or $w'$. Inductively, we either obtain a finite sequence for a Gaussian rational, or an infinite sequence (note that by Proposition~\ref{prop: triadic}, the first $n$ letters of $w_{n+1}(z)$ are simply those of $w_n(z)$). This finite or infinite sequence is called the \emph{cutting sequence} for $z$.

Intuitively, after we enter each $\mathcal{A}_w$, it is divided into three parts at the next level, and by \ref{item: triadic_finer_division} of Proposition~\ref{prop: triadic}, they are labelled by $V_1,V_2,V_3$. Since $z$ lies in one of the parts, we can then append the corresponding label at the end of $w$.

When $\re(z)<0$, let $t$ be the smallest nonnegative integer so that $z\in\overline{\mathbb{H}_t}$. Adding $V_1^{-t}$ to the start of the cutting sequence for $z+t$, we obtain the cutting sequence for $z$. In the spirit of comparison to the case of $\psl(2,\mathbb{Z})$, we can also emphasize the special role of the imaginary axis by saying that the \emph{base circle} of the cutting sequences is $i\mathbb{R}$. We have:

\begin{prop}\label{prop: diophantine_app}
Let $z\in\overline{\mathbb{H}_*}\cap\mathcal{A}$, and set $w(z)$ to be its cutting sequence (based at $i\mathbb{R}$). Then
\begin{enumerate}[label=\normalfont(\arabic*), topsep=0mm, itemsep=0mm]
\item\label{item: diophantine_finite} $w(z)$ is finite if and only if $z\in\mathbb{Q}(i)$;
\item\label{item: diophantine_app} {\normalfont{(Diophantine approximation)}} If $z$ is not a Gaussian rational, then $w_n(z)\cdot z_0\to z$ for any $z_0\in\overline{\mathbb{H}_*}$. In particular, taking $z_0=\infty$, $0$, or $-i$, we obtain a sequence of Gaussian rationals approximating $z$;
\item\label{item: diophantine_degenerate} $z$ lies on the boundary of some component of $S^2-\mathcal{A}$ if and only if the cutting sequence eventually only involves two of the three letters. In particular, if $z\in\mathbb{R}$, then $w(z)$ agree with its cutting sequence with respect to $\psl(2,\mathbb{Z})$ based at $i$, with $L$ replaced by $V_1$ and $R$ replaced by $V_2$.
\end{enumerate}
\end{prop}
\begin{proof}
For \ref{item: diophantine_finite}, note that $w(z)$ is finite if and only if $z$ is a vertex of some $\mathcal{A}_w$.

For \ref{item: diophantine_app}, note that $w_n(z)\cdot\overline{\mathbb{H}_*}\supset w_{n+1}(z)\cdot\overline{\mathbb{H}_*}$. Moreover, as $i\mathbb{R}$ corresponds to a closed plane in $M_A$, and its stabilizer in $\Gamma_A$ is finite, we must have the radii of the infinite sequence of disks $w_n(z)\cdot\overline{\mathbb{H}_*}$ tend to $0$.

For \ref{item: diophantine_degenerate}, each circle in $\mathcal{A}$ is divided into three parts by mutually tangent disks $D_{wV_1},D_{wV_2},D_{wV_3}$. We assume the point in consideration lies in $C_{wV_1}$; the other two cases are similar. Note that $z=wV_1\cdot z'$ for some $z'$ on the part of the circle $x^2+(y+1/2)^2=1/4$ in $\mathbb{H}_*$, and $w(z)=wV_1w(z')$. Clearly $w(z')$ only involves $V_2,V_3$, as desired.
\end{proof}

We can be more precise about the approximation in Part~\ref{item: diophantine_app}. Suppose $w(z)=V_{t_1}^{n_1}V_{t_2}^{n_2}\cdots$ where $t_j\neq t_{j+1}$. For each $j$, set $z_0^{(j)}=\begin{cases}\infty&t_{j+1}=1\\0&t_{j+1}=2\\-i&t_{j+1}=3\end{cases}$. Define
$$z_j=p_j/q_j=V_{t_1}^{n_1}\cdots V_{t_j}^{n_j}\cdot z_0^{(j)}$$
It is not hard to see that $z_j$ is closer to $z$ than the other two vertices of $\mathcal{A}_{V_{t_1}^{n_1}\cdots V_{t_j}^{n_j}}$, so is the ``best" approximation at this step. As a matter of fact, one can show the following:
\begin{prop}[Best rational approximations]
Let $z\in\mathcal{A}$ be an irrational point, and for each $j\ge1$ define $z_j=p_j/q_j$ as above. Then
$$|z-z_j|=\min_{p/q\in\mathcal{A}, |q|\le|q_j|}|z-p/q|.$$
\end{prop}

Finally, one can give an upper bound for $|z-z_j|$ as follows. Note that $z_j$ and $z_{j+1}$ are both vertices of $\mathcal{A}_{V_{t_1}^{n_1}\cdots V_{t_{j+1}}^{n_{j+1}}}$, and $z$ lies in the disk $D_{V_{t_1}^{n_1}\cdots V_{t_{j+1}}^{n_{j+1}}}$, whose diameter $r_j$ is calculated according to Lemma~\ref{lm: farey_triple}. Then $|z-z_j|<r_j$. In fact, it is not hard to see that $|z-z_j|<2/|q_jq_{j+1}|$ if we write $z_j=p_j/q_j$ in lowest terms, again drawing similarities to the case of $\psl(2,\mathbb{Z})$.

\paragraph{Cutting sequences and the action of $\Gamma_A$}
As in the case of $\psl(2,\mathbb{Z})$, cutting sequences are useful for recognizing points in the same $\Gamma_A$-orbit. Recall that Let $J=\begin{pmatrix}i&-1\\0&-i\end{pmatrix}$. We have:
\begin{lm}\label{lm: order_2}
Given $z\in\mathcal{A}\cap\mathbb{H}_*$, if $w(z)=V_1^{a_0}V_2w'$, then $w(Jz)=V_1^{-a_0-1}V_3w'(V_3,V_1,V_2)$. Similarly, if $w(z)=V_1^{a_0}V_3w'(V_1,V_2,V_3)$, then $w(Jz)=V_1^{-a_0-1} V_3 \allowbreak w'(V_2,V_3,V_1)$.
\end{lm}
Here $w'(V_3,V_1,V_2)$ denotes the word obtained from $w'$ after the substitution $V_1\to V_3$, $V_2\to V_1$, $V_3\to V_2$.
\begin{proof}
It suffices to check the correspondence between the vertices of $\mathcal{A}_w$ and $J\cdot \mathcal{A}_w$.
\end{proof}
\begin{lm}\label{lm: order_3}
Let $S=\begin{pmatrix}0&-i\\-i&1\end{pmatrix}$. Let $z\in \mathcal{A}\cap\mathbb{H}_*$ and suppose $w(V_1,V_2,V_3)=w(z)$ is the cutting sequence of $z$ based at $i\mathbb{R}$. Then $w(S^tz)=w(V_{1+t},V_{2+t},V_{3+t})$ for $t=0,1,2$, where the indices are understood modulo $3$.
\end{lm}
\begin{proof}
This follows from the fact that $S^tV_kS^{-t}=V_{k+t}$, and $S(i\mathbb{R})=i\mathbb{R}$.
\end{proof}
\begin{prop}\label{prop: equivalence}
Two points $z_1,z_2\in\mathcal{A}$ lie on the same $\Gamma_A$-orbit if and only if their cutting sequences (based at $i\mathbb{R}$) have the same tail up to a cyclic reordering of the letters $V_1,V_2,V_3$.
\end{prop}
\begin{proof}
By Lemma \ref{lm: order_2}, we can assume $z_1,z_2\in\mathbb{H}_*$. By Part \ref{item: diophantine_app} of Proposition~\ref{prop: diophantine_app}, dropping a finite segment of the cutting sequence at the beginning remains in the $\Gamma_A$-orbit. The proposition then follows from Lemma \ref{lm: order_3}.
\end{proof}

As in the case of $\psl(2,\mathbb{Z})$, we can also define cutting sequences based at an arbitrary circle $C$ in $\Gamma_A\cdot i\mathbb{R}$. Note that in the definition above, the two components of $S^2-i\mathbb{R}$ are distinguished. In general, we can either assign an orientation to $C$, or make a choice of a distinguished component bounded by $C$. We will adopt the second approach, and choose either the disk enclosed by $C$, or the region $y>t$ to be the ``positive component", denoted by $\mathbb{H}_C$.  Let $\gamma\in \Gamma_A$ such that $\mathbb{H}_C=\gamma\cdot\mathbb{H}_*$. Set $V^C_k=\gamma V_k\gamma^{-1}$. Then we can define everything in terms of $V^C_1,V^C_2,V^C_3$. Note that for each $\mathcal{A}_w$ defined above, this simply gives a new ordering of the vertices, respecting the cyclic order. Lemma~\ref{lm: order_3} and Proposition~\ref{prop: equivalence} can then be stated in terms of cutting sequences based at different circles.

\paragraph{Cutting sequences for closed geodesics}
We can also define cutting sequence for an oriented closed geodesic $\tau$ of $M_A=\Gamma_A\backslash\mathbb{H}^3$. Let $l_\tau$ be a lift of $\tau$ to $\mathbb{H}^3$. Suppose it crosses a plane $P_C$ determined by a circle $C$ in $\Gamma_A\cdot i\mathbb{R}$. Starting at $P_C$, along the direction of $l_\tau$, it crosses a sequence of planes, whose circles are labeled by words in $V_1^C,V_2^C,V_3^C$. As $l_\tau$ is stabilized by an element in $\Gamma_A$, the sequence is periodic. We call the periodic part the \emph{cutting sequence} of the oriented closed geodesic $\tau$. By our discussion above, choosing a different lift or a different base circle $C$ (and ignoring the superscript $C$ in the words) does not change the periodic part, up to a cyclic reordering of the letters. The following corollary follows immediately from Proposition~\ref{prop: diophantine_app}:

\begin{cor}
The closed geodesic $\tau$ lies on $\partial\core(M_A)$ if and only if the cutting sequence of $\tau$ only involves two letters.
\end{cor}

\section{Markings of oriented crowns}\label{sec: marking_crown}
In this section, to prepare for our classification of elementary planes, we first recall the ``markings" we assign to crowns in $M_A$, and then prove some of its properties.

\paragraph{Modular symbols}
We first recall some basic facts about modular symbols on the modular surface $X=\psl(2,\mathbb{Z})\backslash\mathbb{H}^2$. Classically, modular symbols form an abelian group and pair with modular forms (see e.g.~\cite[Ch.~IV]{modular_forms1} and \cite[Ch.~3]{stein2007modular}). Here we adopt the more geometric perspective in  \cite{modular_symbol}, which seems better suited for our purpose, especially when we talk about the topology of elementary surfaces later in \S\ref{sec: geom_top_elem_planes}.

A \emph{modular symbol of degree $d$} is a formal product $\gamma_1*\cdots*\gamma_d$, where $\gamma_1,\ldots,\gamma_d$ are complete geodesics on $X$ that start and end at the unique cusp. The lifts of the cusp are precisely $\mathbb{Q}\cup\{\infty\}$, and hence the lift of a modular symbol of degree $1$ is a complete geodesic whose ends are a pair of (extended) rationals, and we may represent the modular symbol by this pair. This representation is not unique: if $\{r_1,r_2\}$ represents a modular symbol, $\{\gamma\cdot r_1,\gamma\cdot r_2\}$ represents the same for all $\gamma\in \psl(2,\mathbb{Z})$. We may always take $r_1=\infty$, and $r_2\in[0,1)$, and hence the modular symbol can be represented by $[r_2]=\{\infty,r_2\}$. Therefore, a modular symbol of degree $d$ can be represented by a product $[r_1]*\cdots*[r_d]$ where $r_k\in\mathbb{Q}\cap[0,1)$. We denote by $\mathcal{S}^d$ the set of all degree $d$ modular symbols, and $\mathcal{S}=\bigcup_{d\ge1}\mathcal{S}^d$.

Sometimes it is convenient to have a version of modular symbols ``without base point". A \emph{cyclic modular symbol} is an equivalence class of modular symbols up to cyclic reordering. Let $\mathcal{S}^d_{\text{cyc}}$ be the set of degree $d$ cyclic modular symbols and $\mathcal{S}_{\text{cyc}}:=\bigcup_{d\ge1}\mathcal{S}^d_{\text{cyc}}$. 

\paragraph{Markings of oriented crowns}
An \emph{oriented crown} $\Sigma$ in $\core(M_A)$ is a crown with a choice of orientation for its core geodesic, together with a totally geodesic isometric immersion $f:\Sigma\to\core(M_A)$ so that the spikes of $\Sigma$ are mapped to complete geodesics from cusp to cusp on $\partial\core(M_A)$. We can extend the image $f(\Sigma)$ to a geodesic plane $\tilde\Sigma$. As a boundary circle $C$ of $\tilde\Sigma$ passes through countably many Gaussian rationals, $\tilde\Sigma$ is closed in $M_A$. We require that the immersion $f$ factors through
$$\Sigma\to\langle\gamma\rangle\backslash\hull(C)\to M_A$$
so that the first map is an embedding. Here $\gamma$ is the hyperbolic element corresponding to the core geodesic of $f(\Sigma)$. In this way, $f$ is generically one-to-one, unless $f(\Sigma)$ has torsion points or is nonorientable. With this caveat in mind, we usually identify $\Sigma$ with its image.

Given an oriented crown $\Sigma$, with respect to the chosen orientation, the spikes of the crown are represented by a modular symbol of degree equal to the number of spikes. Note that this modular symbol is only determined up to a cyclic reordering in its formal product representation, so we treat it as an element in $\mathcal{S}_{\text{cyc}}$.  To break the cycle, we define a \emph{marked} oriented crown as an oriented crown $\Sigma$ with one of the spikes marked. Its modular symbol is written so that it starts and ends at the marked spike.

Let $C$ be a boundary circle of the extension $\tilde\Sigma$ of $\Sigma$. We may choose $C$ so that a lift of the marked spike is at $\infty$. The circle $C$ intersects $y=-1$ and $y=0$ in two rational numbers $\alpha,\beta$; we may also choose $C$ so that $\alpha\to\infty\to\beta$ gives the orientation of the crown. Then the first and last components of the modular symbol of $\Sigma$ are given by $\{\infty,\beta\}$ and $\{-\alpha,\infty\}$.

Conversely, given a pair of rational numbers $\alpha,\beta$, consider the line $l(\alpha,\beta)$ passing through $\alpha-i,\beta$. This gives closed geodesic plane $\tilde\Sigma$ in $M_A$. Then either $\tilde\Sigma$ is a properly immersed ideal polygon (elliptic) or a punctured ideal polygon (parabolic), or a subsurface of $\tilde\Sigma$ is a marked oriented crown whose marked spike comes from $\infty$ and whose orientation is given by $\alpha-i\to\infty\to\beta$. We may even include elliptic (resp.\ parabolic) elementary planes by treating them as marked oriented crowns whose ``core curves" are elliptic (resp.\ parabolic). We call a marked oriented crown elliptic, parabolic, or hyperbolic depending on the type of its core curve.

Denote the marked oriented crown obtained from the ordered pair $(\alpha,\beta)$ by $\Cr(\alpha,\beta)$. Note that $\Cr(\alpha,\beta)=\Cr(\alpha+k,\beta+k)$ for any integer $k$. On the other hand, $\Cr(\alpha,\beta)\neq \Cr(\alpha,\beta+k)$ unless $k=0$, even though the first and last modular symbols $\{-\alpha,\infty\},\{\infty,\beta\}=\{\infty,\beta+k\}$ are the same. We call the ordered pair $(\alpha,\beta)\in(\mathbb{Q}\times\mathbb{Q})/\mathbb{Z}$ the \emph{marking symbol} of $\Cr(\alpha,\beta)$, or simply its marking. We will always choose a representative $(\alpha,\beta)\in (0,1)\times\mathbb{Q}$ or $\{0\}\times\mathbb{Q}_{\ge0}$ or $\{1\}\times\mathbb{Q}_{<1}$. It is clear that the choice is unique. Let $\symb:=((0,1)\times\mathbb{Q})\cup(\{0\}\times\mathbb{Q}_{\ge0})\cup(\{1\}\times\mathbb{Q}_{<1})$.

\paragraph{Change-of-marking map}
It would be interesting to understand the effect of changing the marking on an oriented crown. We define a change-of-marking map $T$ as follows. Let $(\alpha,\beta)\in\symb$. Consider the marked oriented crown $\Cr(\alpha,\beta)$. As it is oriented, it makes sense to talk about the ``next spike" from the marked one. Let $(\tilde\alpha,\tilde\beta)\in\symb$ be the marking symbol of the crown with the same orientation but with this next spike marked. Define $T(\alpha,\beta)=(\tilde\alpha,\tilde\beta)$. We have:
\begin{prop}\label{prop: renormal_map}
Let $\beta_1=p/q,\beta_2=r/s$ (where $p,q,r,s\in\mathbb{Z}$, $q,s\ge0$, $ps-qr=-1$) be the two Farey neighbors of $\beta=(p+r)/(q+s)$. Then $T(\alpha,\beta)=(\frac{q}{q+s},\alpha+\frac{q-p-r}{q+s})$.
\end{prop}
\begin{proof}
We first find an element in $\Gamma_A$ mapping $\beta$ to $\infty$, the line $y=0$ to $y=-1$, and $\infty$ to a point on $y=-1$ with $x$-coordinate in $[0,1)$. Let $\gamma_{\beta}=\begin{pmatrix}s&-r\\q+s&-p-r\end{pmatrix}\in \psl(2,\mathbb{Z})$. Then $V_1J\gamma_\beta=\begin{pmatrix}-iq-q-s&ip+p+r\\-i(q+s)&i(p+r)\end{pmatrix}$ is such an element. The line $l=l(\alpha,\beta)$ is mapped to a line $l'$ passing through $\frac{q}{q+s}-i,\frac{q}{q+s}+\alpha-\beta$. Here we used the fact that the angle between $l'$ and $y=-1$ is the same as that between $l$ and $y=0$. It is then clear by definition that $l'=l(\tilde\alpha,\tilde\beta)$ gives the desired lift of the new marked oriented crown $\Cr(\tilde\alpha,\tilde\beta)$ with $\tilde\alpha=q/(q+s)$ and $\tilde\beta=\alpha+\frac{q-p-r}{q+s}$.

Finally, when $\beta\notin\mathbb{Z}$, $\tilde\alpha\in(0,1)$; when $\beta$ is a positive integer, we have $\alpha<1$, $\tilde\alpha=1$, and $\tilde\beta=\alpha+1-k<1$; when $\beta$ is a nonpositive integer, we have $\tilde\alpha=0$ and $\tilde\beta=\alpha-\beta\ge0$. Thus $T$ is indeed a map from $\symb$ to $\symb$.
\end{proof}

In particular, note that $T(\alpha,\beta+k)=(\frac{q}{q+s},\frac{q}{q+s}+\alpha-\beta-k)$ for $k\in\mathbb{Z}$. Thus $\Cr(\alpha,\beta)$ and $\Cr(\alpha,\beta+k)$ have the same modular symbol representing their spikes.

\begin{lm}
$\{T^n(\alpha,\beta)\}$ is periodic for all $(\alpha,\beta)\in\symb$.
\end{lm}
\begin{proof}
Geometrically, a crown has finitely many spikes. Algebraically, at each iteration, the denominator of each component is a factor of the least common multiple of the denominator of $\alpha,\beta$ in their lowest terms.
\end{proof}

Suppose the period of $\{T^n(\alpha,\beta)\}$ is $d$, and set $(\alpha_k,\beta_k)=T^k(\alpha,\beta),0\le k\le d-1$. Then the modular symbol associated to the marked oriented crown $\Cr(\alpha,\beta)$ is then $[\beta]*\cdots*[\beta_{d-2}]*[\beta_{d-1}]$. In particular $\{\infty,\beta_{d-1}\}=\{-\alpha,\infty\}$. Note that the orbit $\beta_0\to\beta_1\to\cdots\to\beta_{d-1}$ contains all the information, so we sometimes simply refer to $\beta_k$ as the $k$-th marking symbol instead of the pair $(\alpha_k,\beta_k)$

For $\beta=(p+r)/(q+s)$ with $ps-qr=-1,q,s\ge0$, define $\eta_\beta=\begin{pmatrix}-iq-q-s&ip+p+r\\-i(q+s)&i(p+r)\end{pmatrix}.$ Note that $\eta_\beta=V_1J\gamma_\beta$ in the notation of the proof of Proposition~\ref{prop: renormal_map}.

\begin{prop}\label{prop: iteration}
The element $\eta_\beta^{-1}\eta_{\beta_1}^{-1}\cdots\eta_{\beta_{d-1}}^{-1}\in \Gamma_A$ represents the core curve of the crown.
\end{prop}
\begin{proof}
Indeed, by the proof of Proposition~\ref{prop: renormal_map}, $\eta_{\beta_{d-1}}\cdots\eta_{\beta_1}\eta_{\beta}$ represents the monodromy circling around the crown, in the direction opposite to its orientation.
\end{proof}
\begin{eg}
Consider the case $\alpha=\beta=1/n$, where $n\ge2$. It is easy to see that $T(1/n,1/n)=(1/n,1/n)$. In particular
$$\eta_{1/n}^{-1}=\begin{pmatrix}i&-1\\in&-i-n\end{pmatrix}$$
represents the core curve of the crown. When $n=2$, $\Cr(1/2,1/2)$ is parabolic, and hence extends to a parabolic elementary plane. When $n\ge3$, we note that $\eta_{1/n}$ stabilizes the disk $x^2+(y+1/2)^2\le 1/4$. In particular, $\Cr(1/n,1/n)$ is hyperbolic and extends to a hyperbolic elementary plane. This demonstrates how to apply Proposition~\ref{prop: iteration} to determine the type of the surface a given crown extends to (note that there is a unique way to extend a crown to a geodesic plane).\qed
\end{eg}

\section{Cutting sequences from markings}\label{sec: cutting_from_marking}

As an oriented crown is determined by its marking, the information about the core curve is encoded in the marking. As demonstrated in the previous section, we can calculate the matrix representing the curve and determine its type (elliptic, parabolic, or hyperbolic) from the trace of the matrix. In this section, we describe how to calculate the cutting sequence of the oriented core geodesic when the crown is hyperbolic. Recall that this cutting sequence is defined by a triadic subdivision of $\mathcal{A}$, labelled by $V_1,V_2,V_3$ at each level of division.

\begin{eg}\label{eg: cutting_sequence}
Consider again $\alpha=\beta=1/n$, where $n\ge3$. Denote the circle in $\mathcal{A}$ tangent to the real line at $1/n$ by $C_{1/n}$. The line $l=l(1/n,1/n)$ is given by $x=1/n$. In the upper half plane, $l$ goes from $\infty$ to $1/n$, entering $n-1$ circles labeled $V_2$, and intersects the real line at $1/n$. In $C_{1/n}$, the line should follow the path of the modular symbol $[1/n]$. We label the rational points on this circle as follows: the point $1/n$ is labeled $\infty$, the point of tangency with $C_0$ is labeled $0$, and that with $C_{1/(n-1)}$ is labeled $1$, and inductively label the other points using Farey midpoints. Equivalently, label a point with its image under $\eta_{1/n}$.  The line $x=1/n$ exits the circle $C_{1/n}$ at a point labeled $1/n$, and hence starting from $1/n$ (labeled $\infty$), it enters a circle labeled $V_3$, and then $n-2$ circles labeled $V_2$. We then repeat this process to figure out the path of $x=1/n$ in the next circle it enters. It follows that the cutting sequence of the attracting fixed point of the hyperbolic element $\eta_{1/n}^{-1}$ is given by $V_2\overline{V_2^{n-2}V_3}$, and hence $\overline{V_2^{n-2}V_3}$ gives a cutting sequence for the core geodesic. See Figure~\ref{fig: cutting_seq_eg} for an illustration of the process.\qed
\end{eg}

\begin{figure}[htp]
\centering
\includegraphics[trim={0cm 5cm 0cm 0cm},clip,width=\linewidth]{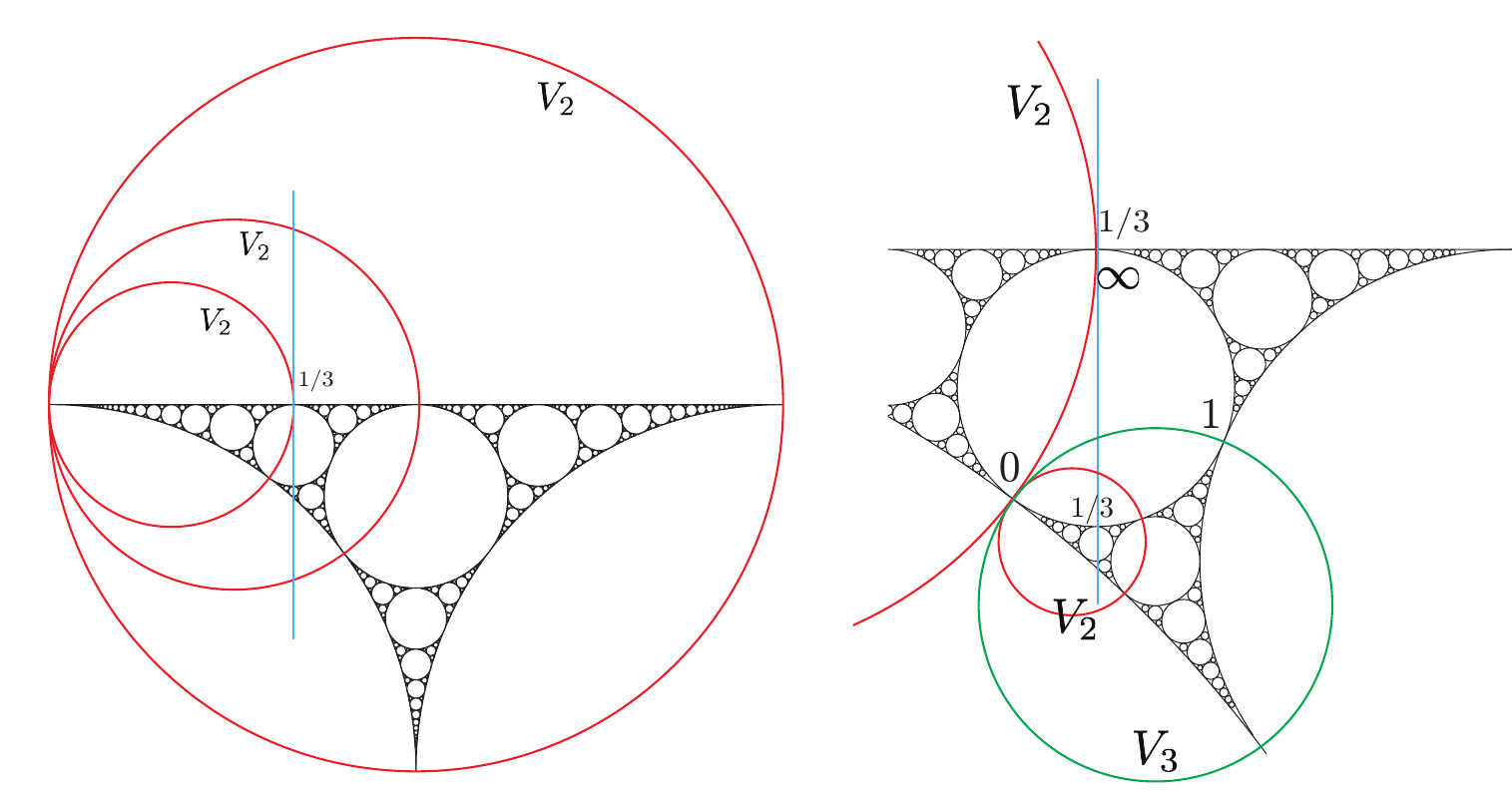}
\caption[Example~\ref{eg: cutting_sequence} with $n=1/3$]{Example~\ref{eg: cutting_sequence} with $n=1/3$. The right figure is a blow up of the left one near $1/3$}
\label{fig: cutting_seq_eg}
\end{figure}

To generalize this procedure, we make the following elementary observation: a circle passing through the point of tangency of two circles form the same angle with each of them. We need the following lemmas:
\begin{lm}\label{lm: cutting_from_symbol_1}
Given $(\alpha,\beta)\in\symb$, set $T(\alpha,\beta)=(\alpha_1,\beta_1)$. Write $\beta=[a_0;a_1,\ldots,a_k,1]$ in its continued fraction (note that we adopt the convention of ending the continued fraction of a rational number with $1$). Let $w=V_1^{a_0}V_2^{a_1}\cdots V_{1/2}^{a_k}$. Recall that $D_w=w\cdot\overline{\mathbb{H}_*}$, where $\mathbb{H}_*$ is the right half plane. If $\beta_1\neq0,1$, the attracting fixed point of the hyperbolic or parabolic element representing the core curve is contained in $D_w$.
\end{lm}
\begin{proof}
Write $w=V_1^{a_0}V_2^{a_1}\cdots V_{1/2}^{a_k}$. Write $\beta=(p+r)/(q+s)$ where $ps-qr=-1,q,s\ge0$. Consider the line $l=l(\alpha,\beta)$. As in Example~\ref{eg: cutting_sequence}, since $\beta\in D_w$, starting from $\infty$, along the orientation of the crown, $l$ enters the disk $D_w$. If the accumulation point is not contained in $D_w$, then either $l$ intersects $C_{p/q}$ or $C_{r/s}$ after intersecting finitely many circles in $\mathcal{A}$, or $l$ passes through the point of tangency of $C_{p/q}$ and $C_{r/s}$.

In the latter case, the crown $\Cr(\alpha,\beta)$ is parabolic, with the parabolic fixed point at the point of tangency, which is in particular on the boundary of $D_w$. For the former case, assume $l$ intersects $C_{p/q}$; the other possibility is similar. This can only happen if $l$ passes through finitely many tangency points and enters $C_{p/q}$. In particular, $l$ has the same angle with $y=0$, $C_\beta$ and $C_{p/q}$.  The matrix $\eta=\begin{pmatrix}q+s&-p-r\\q&-p\end{pmatrix}$ maps $\beta$ to $0$, and $p/q$ to $\infty$ and hence the circle $C_\beta$ to $C_0$ and $C_{p/q}$ to $y=-1$. The line $l$ is mapped to a circle $\eta(l)$ passing through $0$ and intersects $y=0$, $C_0$ and $y=-1$ in the same angle. This is only possible when the center of $\eta(l)$ lies on $y=-1/2$ and hence passes through $-i$. In particular, $\eta(l)$ passes through the tangency point of $y=-1$ and $C_0$ and hence $l$ passes through the tangency point of $C_{p/q}$ and $C_{\beta}$. This is only possible when $\beta_1=0$.
\end{proof}

Given $\beta=(p+r)/(q+s)$ where $ps-qr=-1,q,s\ge0$, define $\iota(\beta)=q/(q+s)$. Note that this is simply the first component of $T(\alpha,\beta)$. Moreover, on the interval $[0,1]$, $\iota^2=id$, as then the two Farey neighbors of $\iota(\beta)$ are $p/(p+r)$ and $(q-p)/(q+s-p-r)$. We also denote the word $w=V_1^{a_0}V_2^{a_1}\cdots V_{1/2}^{a_k}$ by $w(\beta)$. Finally, let $\beta_L, \beta_R$ be the left and right Farey neighbors of $\beta$.
\begin{lm}\label{lm: cutting_from_symbol_2}
Under the same assumptions and notations of the previous lemma and its proof, if $\beta_1=0$, then either $\alpha=\beta=0$, and $\Cr(\alpha,\beta)$ is elliptic, or $l$ passes through $p/q-i/q^2$. Moreover, in the latter case,
\begin{enumerate}[label=\normalfont{(\roman*)}, topsep=0mm, itemsep=0mm]
\item\label{case: betaLelliptic} if $\beta_L=0$, then $\alpha=0$, $\beta=1/n$ for some $n\ge1$, and the crown $\Cr(\alpha,\beta)$ is elliptic;
\item\label{case: betaL=0} Otherwise, if $\beta_3\neq0$, then the conclusion of Lemma~\ref{lm: cutting_from_symbol_1} holds with $w$ replaced by the word $w(\beta_L)V_3\widetilde w(\iota(\beta_L))$; here $\widetilde w(\beta)$ denotes the word obtained from $w(\beta)$ by dropping the first letter.
\end{enumerate}
\end{lm}
In Case~\ref{case: betaL=0}, since $\iota(\beta_L)\in[0,1)$, $w(\iota(\beta_L))$ starts with $V_2$. Hence $\widetilde w(\iota(\beta_L))$ starts with one less copy of $V_2$ than $w(\iota(\beta_L))$.
\begin{proof}
If $\beta=0$ then $\alpha=0$ and clearly $\Cr(0,0)$ is elliptic. Otherwise, since $\iota(\beta)>0=\beta_1$, we must have $0\le\alpha=\beta-\iota(\beta)<\beta$. The matrix $\eta:=\eta_0\eta_{\beta}=\begin{pmatrix}iq&-ip\\-q+iq+is&p-ip-ir\end{pmatrix}$ maps the tangency point of $C_{p/q}$ and $C_{\beta}$ to $\infty$, and the circle $C_{p/q}$ to the real axis. Moreover $\eta(l)$ intersects the real axis at $\iota(\beta)$. Hence $l$ passes through $\eta^{-1}(\iota(\beta))=p/q-i/q^2$.

The case $\beta_L=0$ is easily verified. Assume $\beta_L>0$. It is easy to calculate that $\beta_2=\iota(\beta)$, and $\beta_3=\{\beta\}-\iota(\beta)$, where $\{\cdot\}$ denotes the fractional part. Hence $\alpha-\beta_3$ is a nonnegative integer (actually $0$ or $1$). Note that it is automatically true $\beta_3\neq1$, as otherwise $\alpha=\beta_3=1$ and necessarily $\alpha>\beta$, a contradiction. Moreover, if we apply the matrix $\eta_{p/q}$, the point $p/q-i/q^2$ is mapped to $\iota(\beta_L)$. Since $\beta_L>0$, $\iota(\beta_L)\neq0$. If $\beta_L=n$ for some positive integer $n$, then as $l$ then passes through $n-i$ we have $\alpha=n\ge1$, forcing $\alpha=1$, contradiction again. Thus $\iota(\beta_L)\in(0,1)$.

Going along $l$ from $\infty$ to $\beta$ on the real line, and then to the point $p/q-i/q^2$, by an argument as in the previous lemma, it enters the disk $D_{w(\beta)}$, and then enters the circle labeled $V_3$ in $D_w$. Moreover, $\eta_{p/q}$ maps $D_{w(\beta)V_3}$ to $D_{V_2}$, and matches the circles within $D_{w(\beta)V_3}$ and $D_{V_2}$ according to their labels. Thus $l$ enters the disk $D_{w(\beta)V_3\widetilde w(\iota(\beta_L))}$ when it reaches $p/q-i/q^2$. Again, using the argument of the last lemma, we conclude that the accumulation point is contained in $D_{w(\beta)V_3\widetilde w(\iota(\beta_L))}$ if $\beta_3\neq0$. 
\end{proof}

\begin{lm}\label{lm: cutting_from_symbol_3}
Under the same assumptions and notations, suppose furthermore $\beta_3=\beta_1=0$, $\beta\neq0$ and $\beta_L\neq0$. Then $\alpha=0$, $\beta\in(0,1)$ and $\beta_2=\beta$. Moreover,
\begin{enumerate}[label=\normalfont{(\roman*)}, topsep=0mm, itemsep=0mm]
\item\label{part: elliptic} if $\iota(\beta_L)=1/n$ for some $n\ge2$, then $\beta=n/(n^2-1)$ and $\Cr(\alpha,\beta)$ is elliptic;
\item\label{part: parabolic} Otherwise, if $\iota(\beta_L)_L=1/2$, then $\beta=(2n+1)/(2n^2+2n)$ for some $n\ge1$, and $\Cr(\alpha,\beta)$ is parabolic;
\item\label{part: hypebolic} Otherwise, $\Cr(\alpha,\beta)$ is hyperbolic, and the attracting fixed point of the hyperbolic element has cutting sequence $w(\beta_L)\overline{V_3\widetilde w(\iota(\beta_L)_L)}$; in particular, the cutting sequence of the oriented geodesic is given by $\overline{V_3\widetilde w(\iota(\beta_L)_L)}$.
\end{enumerate}
\end{lm}
Note that we exclude Case~\ref{part: parabolic} from Case~\ref{part: hypebolic} so that $\widetilde w(\iota(\beta_L)_L)$ is not an empty word.
\begin{proof}
We have shown that $\alpha=\beta_3$ or $\alpha=\beta_3+1$. As argued in the previous lemma, $\alpha\neq1$. Therefore $\alpha=\beta_3=0$. This also implies $\{\beta\}=\beta$ and hence $\beta\in(0,1)$. We have also shown that $\alpha=\beta-\iota(\beta)$, and therefore $\beta_2=\iota(\beta)=\beta$. As $\beta=(p+r)/(q+s)$ and $\iota(\beta)=q/(q+s)$, we conclude that $p+r=q$.

Since $\beta\in(0,1)$ and $\beta_L\neq0$, we have $\beta_L\in(0,1)$, and hence $\iota^2(\beta_L)=\beta_L$. If $\iota(\beta_L)_L=0$,  we then have $\beta_L=\iota(\beta_L)=1/n$, for some $n\ge2$. In particular $p=1,q=n$, and hence $r=q-p=n-1$ and $s=(qr-1)/p=n^2-n-1$. This gives $\beta=n/(n^2-1)$, and it is easy to see that the crown $\Cr(0,\beta)$ is elliptic.

Otherwise, apply the matrix $\eta_{\beta_L}=\eta_{p/q}$; the point $p/q-i/q^2$ is mapped to $\iota(\beta_L)$, the disk $D_{w(\beta_L)V_3}$ is mapped to $D_{V_2}$, and the circles within $D_{w(\beta_L)V_3}$ and $D_{V_2}$ are matched according to their labels. The circle $\eta_{p/q}(l)$ enters $D_{V_2}$, passes through $\iota(\beta_L)$ and the tangency points between $C_{\iota(\beta_L)}$ and $C_{\iota(\beta_L)_L}$, and then exits the circle $C_{\iota(\beta_L)_L}$ through a rational point $\xi$. Since $\iota(\beta_L)_L\neq0$, $\xi$ is contained in $D_{w(\iota(\beta_L)_L)V_3}$. If we apply $\eta_{\iota(\beta_L)_L}$,  $\xi$ is mapped to $\iota(\beta_L)$ by periodicity. This implies that the accumulation point when we go along the line $l$ has word $w(\beta_L)\overline{V_3\widetilde w(\iota(\beta_L)_L)}$. In particular when $\widetilde w(\iota(\beta_L)_L)$ is an empty word, the crown is parabolic, and otherwise it is hyperbolic.

Finally, $\widetilde w(\iota(\beta_L)_L)$ is an empty word if and only if $\iota(\beta_L)_L=1/2$. Therefore $\iota(\beta_L)=(n+1)/(2n+1)$ for some $n\ge1$, and hence $\beta_L=2/(2n+1)$. In particular $p=2,q=2n+1$ and $r=q-p=2n-1$, $s=(qr-1)/p=2n^2-1$. Thus $\beta=(2n+1)/(2n^2+2n)$ for some $n\ge1$, as desired.
\end{proof}

We can formulate similar lemmas for $\beta_1=1$, with `$L$' replaced by `$R$' in the statements. We can prove these statements as above, \emph{mutatis mutandis}, or we can quote the following lemma:
\begin{lm}\label{lm: reflection}
Consider the involution $\tau:\symb\to\symb$ defined by $\tau(0,0)=(0,0)$ and $\tau(a,b)=(1-a,1-b)$ when $(a,b)\neq(0,0)$. Then $T\circ\tau=\tau\circ T$. Moreover, $\Cr(\alpha,\beta)$ and $\Cr(\tau(\alpha,\beta))$ has the same type. In fact, the cutting sequence for the core geodesic of $\Cr(\tau(\alpha,\beta))$ can be obtained from that of $\Cr(\alpha,\beta)$ by exchanging $V_1$ and $V_2$.
\end{lm}
\begin{proof}
Algebraically,  $T\circ\tau=\tau\circ T$ is equivalent to $\iota(1-\beta)=1-\iota(\beta)$. Geometrically, all statements follow from the fact that the Apollonian gasket is symmetric across $x=1/2$.
\end{proof}
In particular, when analyzing topological types of crowns and the surfaces they extend to, we may always assume $\beta\ge\alpha$. We apply the analysis above to obtain the following:
\begin{thm}\label{thm: elliptic}
The crown $\Cr(\alpha,\beta)$ is elliptic if and only if $(\alpha,\beta)$ or $\tau(\alpha,\beta)$ lies in one of the following $T$-orbits:
\begin{enumerate}[label=\normalfont{(\arabic*)}, topsep=0mm, itemsep=0mm]
\item $(0,0)$;
\item $(0,1/n)\to (1/n,0)$ for some integer $n\ge1$;
\item $(0,n/(n^2-1))\to(n/(n^2-1),0)$ for some integer $n\ge2$.
\end{enumerate}
\end{thm}
\begin{proof}
This follows from Lemmas \ref{lm: cutting_from_symbol_1} -- \ref{lm: cutting_from_symbol_3}, especially Part~\ref{case: betaLelliptic} of \ref{lm: cutting_from_symbol_2} and Part~\ref{part: elliptic} of \ref{lm: cutting_from_symbol_3}.
\end{proof}
Since $\Cr(\alpha,\beta)$ is elliptic if and only if it extends to an elliptic elementary plane, we also obtain all possible planes in $\core(M_A)$ that are ideal polygons.

Finally, for a hyperbolic crown, we wish to extract information about its core geodesic from its marking. For $\beta\neq 0,1$, define $\widetilde w(\beta)=\begin{cases}w(-1/\beta)&\beta\in(-\infty,0)\\w(\beta/(1-\beta))&\beta\in(0,1)\\w(\beta-1)&\beta\in(1,\infty)\end{cases}$. Note that when $\beta\in(0,1)$, this is consistent with the definition above (i.e. dropping the first letter of $w(\beta)$).

\begin{thm}\label{thm: algorithm}
Assume $(\alpha,\beta)\in\symb$ is not in the list from Theorem~\ref{thm: elliptic}. Suppose that the period of $\{T^n(\alpha,\beta)\}$ is $d$ and set $T^k(\alpha,\beta)=(\alpha_k,\beta_k)$. Then the following algorithm produces the cutting sequence for the attracting fixed point of the hyperbolic or parabolic element, whose periodic part gives the cutting sequence for the core geodesic.

\noindent\normalfont\textbf{Algorithm \protect\NoHyper\ref{thm: algorithm}\protect\endNoHyper} (Cutting sequences from markings)\textbf{.} In the algorithm, for $\beta\in\mathbb{Q}$, $w(\beta), \widetilde w(\beta)$ denote the corresponding words as defined above, but with $V_1$ replaced by $A_1$ and $V_2$ by $A_2$.
\begin{enumerate}[label=\normalfont{\arabic*.}, topsep=0mm, itemsep=2mm]
\item Set $k=1$, $w$ the empty word, the permutation $\sigma=\begin{pmatrix}1&2&3\\2&3&1\end{pmatrix}$, and $A_i=V_i$ for $i=1,2,3$;
\item \begin{enumerate}
\item If $\beta\in\mathbb{Z}_{\le0}$, set $A_i=A_{\sigma(i)}$; then set $c=V_1^\beta w(\beta_1)$, $k=3$, and $d=d+1$;
\end{enumerate}
Otherwise,
\begin{enumerate}[resume]
\item If $\beta_1\neq 0,1$, then set $c=w(\beta)$;
\item If $\beta_1=0$ and $\beta_3\neq0$, then set $c=w(\beta_L)A_3\widetilde w(\iota(\beta_L))$, and $k=3$;
\item If $\beta_1=\beta_3=0$, then set $c=w(\beta_L)$ and $w=A_3\widetilde w(\iota(\beta_L)_L)$ and go to Step 6;
\item If $\beta_1=1$ and $\beta_3\neq1$, then set $c=w(\beta_R)A_3\widetilde w(\iota(\beta_R))$, and $k=3$;
\item If $\beta_1=\beta_3=1$, then set $c=w(\beta_R)$ and $w=A_3\widetilde w(\iota(\beta_R)_R)$ and go to Step 6;
\end{enumerate}
\item Set $\beta=\beta_k$. If $\beta_k\in(-\infty,0)$, then $A_i=A_{\sigma^2(i)}$; if $\beta_k\in(1,\infty)$, then $A_i=A_{\sigma(i)}$;
\item \begin{enumerate}
\item If $\beta_{k+1}\neq0,1$, then set $w=wA_3\widetilde w(\beta)$ and $k=k+1$;
\item If $\beta_{k+1}=0$, then set $w=wA_3\widetilde w(\beta_L)A_3\widetilde w(\iota(\beta_L))$, and $k=k+3$;
\item If $\beta_{k+1}=1$, then set $w=wA_3\widetilde w(\beta_R)A_3\widetilde w(\iota(\beta_R))$, and $k=k+3$;
\end{enumerate}
\item If $k>d$, then go to Step 6, otherwise back to Step 3.
\item Return the hyperbolic or parabolic element $cwc^{-1}$, the cutting sequence for the attracting fixed point $c\overline{w}$ and that for the core geodesic $\overline{w}$.
\end{enumerate}
\end{thm}
\begin{proof}
This follows from inductive application of Lemmas~\ref{lm: cutting_from_symbol_1} -- \ref{lm: cutting_from_symbol_3}. Note that in Step 3 we apply the permutation to match the labels. Note also that Step 2(a) is distinguished as the cutting sequence by definition is not invariant under reflection across $x=1/2$, but invariant under translation $z\mapsto z+1$; slight modification is thus needed when $\beta$ is a nonpositive integer.
\end{proof}

\begin{eg}
Here is a slightly more complicated example. We start with $(\alpha,\beta)=(0,5/13)$. The period $d=6$. Below is a rundown of the algorithm:

\begin{table}[htp]
\centering
\begin{tabular}{c|c|c|c|c}
\hline
&Marking symb.&$A_1,A_2,A_3$&Continued fraction&Words\\
\hline
$\beta_0$&$5/13$&\multirow{2}{*}{$V_1,V_2,V_3$}&$\beta_0=[0;2,1,1,2]$&$c=V_2^2V_1V_2V_1$\\
\cline{1-2}\cline{4-5}
$\beta_1$&$3/13$&&$\beta_1=[0;4,3]$&$w=V_3V_2^3V_1^2$\\
\hline
$\beta_2$&$14/13$&$V_2,V_3,V_1$&$\beta_2=[1;13]$&$w=wV_1V_3^{12}$\\
\hline
$\beta_3$&$-4/13$&\multirow{4}{*}{$V_1,V_2,V_3$}&$-1/\beta_3=[3;4]$&$w=wV_3V_1^3V_2^3$\\
\cline{1-2}\cline{4-5}
$\beta_4$&$8/13$&&\multirow{3}{12em}{$(\beta_4)_L=3/5=[0;1,1,2]$, $\iota((\beta_4)_L)=2/5=[0;2,2]$}&\multirow{3}{*}{$w=wV_3V_1V_2V_3V_2V_1$}\\
\cline{1-2}
$\beta_5$&$0$&&&\\
\cline{1-2}
$\beta_6$&$5/13$&&&\\
\hline
\end{tabular}
\caption{A rundown of the algorithm with $(\alpha,\beta)=(0,5/13)$}
\end{table}

Thus we have the word $\overline{V_3V_2^3V_1^3V_3^{13}V_1^3V_2^3V_3V_1V_2V_3V_2V_1}$ for the core geodesic. This geodesic is not on the boundary of the orbifold, as it involves all three letters. The hyperbolic element we get is
$$cwc^{-1}=\begin{pmatrix}68040-184315i&70536i\\184315-479219i&-2496+184315i\end{pmatrix},$$
which agrees with the one calculated from Propostion~\ref{prop: iteration}, as easily checked.\qed
\end{eg}

Finally, we have the following:
\begin{thm}\label{thm: parabolic}
The crown $\Cr(\alpha,\beta)$ is parabolic if and only if $(\alpha,\beta)$ or $\tau(\alpha,\beta)$ lies in one of the following $T$ orbits:
\begin{enumerate}[label=\normalfont{(\arabic*)}, topsep=0mm, itemsep=0mm]
\item\label{orbits: parabolic1} $(1/2,1/2)$;
\item\label{orbits: parabolic2} $(0,(2n+1)/(2n^2+2n))\to((2n+1)/(2n^2+2n),0)$ for some integer $n\ge1$.
\end{enumerate}
\end{thm}
\begin{proof}
When $\beta_1=\beta_3=0$ or $1$, Part \ref{part: parabolic} of Lemma \ref{lm: cutting_from_symbol_3} gives the orbits in \ref{orbits: parabolic2} above. Otherwise, when we run the algorithm above, the recursive part should return $w=V_i^k$ for some $i\in\{1,2,3\}$ and $k\ge1$. This is only possible when $\beta_k\in[0,1]$. If $\beta_k\neq0,1$ for all $k=0,\ldots,d-1$, then we must have $\beta_k=1/2$ for all $k$. This gives Case~\ref{orbits: parabolic1} above. Otherwise, $(\beta_k)_L=1/2$ and $\beta_{k+1}=0$, or $(\beta_k)_R=1/2$ and $\beta_{k+1}=1$ for some $k$. In the former case, $\beta_k=(n+1)/(2n+1)$ for some $n\ge2$ (the case $n=1$ yields an elliptic crown, as easily checked). Then $\beta_{k+2}=2/(2n+1)$, $\beta_{k+3}=(k-1)/(2n+1)$. However then $\beta_{k+3}, (\beta_{k+3})_L, (\beta_{k+3})_R\neq1/2$, and hence the crown cannot be parabolic. A similar argument also rules out the other case.
\end{proof}

Since $\Cr(\alpha,\beta)$ is parabolic if and only if it extends to a parabolic elementary plane, we also obtain all possible closed planes in $\core(M)$ that are punctured ideal polygons.

\section{Hyperbolic elementary surfaces: Single crowns}\label{sec: single_crown_class}

In this section and the next, we classify elementary planes whose fundamental group contains a hyperbolic element, i.e. single crowns and double crowns. The proofs are technical but elementary; we refer to Theorem~\ref{thm: crown} and Theorem~\ref{thm: double_crown} for the complete lists.

First we observe that Algorithm \ref{thm: algorithm} gives a lot of information about the combinatorics of the cutting sequence of the core geodesic of a crown. In particular, it has the following features: (a) it can be divided into blocks sandwiched by single letters, which we will call \emph{lampposts}; note, however, that a block can be empty; (b) within each block, only two letters are involved; (c) in the block after a lamppost, the two letters involved are different from that of the lamppost.

We first apply these observations to classify elementary planes of Type~\ref{item: crown} in Proposition~\ref{prop: elementary_circle}, that is, elementary planes $P$ so that $\core(P)=P\cap\core(M_A)$ is a single crown.

Recall that given $(\alpha,\beta)\in\symb$, we denote by $d$ the period of $\{T^k(\alpha,\beta)\}$ and set $(\alpha_k,\beta_k)=T^k(\alpha,\beta)$. In this section we further assume that $\Cr(\alpha,\beta)$ is hyperbolic, and that its core geodesic lies on the convex core boundary of the Apollonian orbifold $M_A$. We have:

\begin{prop}\label{prop: crown_case1}
Suppose $\beta_k\notin[0,1]$ for all $0\le k<d$. Then $(\alpha,\beta)=(1/m,-1/n)$ or $(1-1/m,1+1/n)$ for some integers $m,n\ge1$.
\end{prop}

\begin{proof}
First notice that if $\beta_k>1\ge\alpha_k$, then $\beta_{k+1}=\iota(\beta_k)+\alpha_k-\beta_k<\iota(\beta_k)\le1$, and hence $\beta_{k+1}<0$ by assumption. Similarly if $\beta_k<0$ then $\beta_{k+1}>1$. By Lemma~\ref{lm: reflection}, we may assume $\alpha<\beta$, for otherwise we can replace them by $(1-\alpha,1-\beta)$. Then, as we run the algorithm, the word $w$ starts with a lamppost $V_2$, and the following lampposts alternate between $V_3$ and $V_2$. The block following a $V_2$ consists of $V_1$ and $V_3$, possibly empty. As we assume the core geodesic lies on the boundary of the Apollonian orbifold, it must consist entirely of $V_3$. Similarly the block following a $V_3$-lamppost must consist entirely of $V_2$. It then follows that $\beta_1=-1/m<0$, and $\beta_2=1+1/n>1$ for some integers $m,n\ge1$. It is then easy to check that $(\alpha,\beta)=(\alpha_2,\beta_2)=(1-1/m,1+1/n)$, and the corresponding cutting sequence for the core geodesic reads $w=\overline{V_2V_3^{m-1}V_3V_2^{n-1}}$.
\end{proof}

\begin{prop}\label{prop: crown_case2}
Suppose $\beta\in(0,1)$ and $\beta_k\neq0,1$ for all $0\le k<d$. Then $(\alpha,\beta)=(1/m,1/n)$ or $(1-1/m,1-1/n)$ for some integers $m,n\ge2$, $(m,n)\neq(2,2)$.
\end{prop}

\begin{proof}
The proof goes along the same line as the previous lemma. The first two lampposts are $V_3$, and the blocks consist of either all $V_1$ or all $V_2$. The cutting sequence for the core geodesic reads $w=\overline{V_3V_2^{m-2}V_3V_2^{n-2}}$ or $w=\overline{V_3V_1^{m-2}V_3V_1^{n-2}}$.
\end{proof}

\begin{prop}\label{prop: crown_case3}
Suppose $\beta_1=\beta_3=0$. Then $(\alpha,\beta)=(0,(mn+n+1)/(n^2(m+1)+2n))$ or $(0,(mn+m+n)/(m(n+1)^2+n^2-1))$ for some integers $m\ge2,n\ge1$.
\end{prop}

\begin{proof}
We must have $\iota(\beta_L)_L=m/(m+1)$ or $1/(m+1)$ for some $m\ge2$. The statement then follows from $\beta_L\in(0,1)$ and $\iota(\beta)=\beta$ (cf. the proof of Lemma~\ref{lm: cutting_from_symbol_3}).
\end{proof}

\begin{lm}
Suppose $\beta_1=0$ but $\beta_3\neq0$. Then $\alpha=\beta-\iota(\beta)\neq0$ and $d=4$ or $6$.
\end{lm}

\begin{proof}
Since $\alpha_1=\iota(\beta)$ and $\beta_1=0$, we immediately have $\alpha=\beta-\iota(\beta)$. Should $\alpha=0$, then $\beta_3=0$. We can easily calculate $\beta_3=\iota(\iota(\beta))-\iota(\beta)=\alpha-t$ where $t$ is an integer (actually $t=0,1$), and hence $\alpha_4=\iota(\alpha)$. Clearly $\beta_4=\iota(\alpha)+\iota(\beta)\neq0$. Suppose $\beta_4\neq1$. Then the part of the cutting sequence of the core geodesic coming from $\beta_3=\alpha-t$ is given by $V_3\widetilde w(\alpha)$ when $t=0$ and $V_2\widetilde w(\alpha-1)$ when $t=1$. In both cases we then have $\alpha=1/(n+1)$ or $n/(n+1)$ for some integer $n\ge1$. In particular $\alpha_4=\iota(\alpha)=\alpha$, and hence $d=4$.

Otherwise, $\beta_4=1$. Then $\beta_5=\iota(\alpha)$ and $\alpha_6=\iota(\iota(\alpha))=\alpha$. This implies $d=6$.
\end{proof}

\begin{prop}\label{prop: crown_case4}
Suppose $\beta_1=0$ but $\beta_3\neq0$, and $d=4$. Then $(\alpha,\beta)$ assume one of the following values:
\begin{enumerate}[label=\normalfont{(\roman*)}, topsep=0mm, itemsep=0mm]
\item $\left(\dfrac1n,\dfrac{\zeta+m}{m^2+m\zeta-1}\right)$, where $m$ is an integer $\ge2$, $\zeta$ is a positive divisor of $m^2-1$, and $n=m+(m^2-1)/\zeta$;
\item $\left(1-\dfrac1n,1-\dfrac{\zeta-m}{-m^2+m\zeta+1}\right)$, where $m$ is an integer $\ge3$, $\zeta\ge (m^2-1)/(m-2)$ is a positive divisor of $m^2-1$, and $n=m-(m^2-1)/\zeta$;
\item $\left(1-\dfrac1n,1+\dfrac{m-\zeta}{m^2-m\zeta-1}\right)$, where $m$ is an integer $\ge3$, $\zeta\le m-2$ is a positive divisor of $m^2-1$, and $n=-m+(m^2-1)/\zeta$.
\end{enumerate}
\end{prop}

\begin{proof}
First assume $\beta\in(0,1)$. Then the cutting sequence for the core geodesic reads
$$\overline{V_3\widetilde w(\beta_L)V_3\widetilde w(\iota(\beta_L))V_3\widetilde w(\alpha)}.$$
As discussed in the proof of the last lemma, we must have $\alpha=1/n$ or $(n-1)/n$ for some $n\ge2$. First consider the former case. Then $\widetilde w(\alpha)=V_2^{n-2}$. For the sequence to involve only two letters, we must have $\beta_L=1/m$ for some $m\ge2$. Then $\iota(\beta_L)=\beta_L$ and hence the cutting sequence only involves $V_2$ and $V_3$. Now $\beta=\dfrac{t+1}{mt+m-1}$ for some $t\ge1$, and $\iota(\beta)=\dfrac{m}{mt+m-1}$, and hence $1/n=\alpha=\beta-\iota(\beta)=\dfrac{t-m+1}{mt+m-1}$. This implies $n=m+\dfrac{m^2-1}{t-m+1}$. Set $\zeta=t-m+1>0$. Then $\zeta$ is a positive divisor of $m^2-1$, and $t=\zeta+m-1\ge1$. Finally $\beta=\dfrac{\zeta+m}{m^2+m\zeta-1}$, as desired. Note that as $n\ge m+1\ge3$, the cutting sequence does involve both letters.

The case $\alpha=(n-1)/n$ goes along the same way: $\beta_L=(m-1)/m$ for some $m\ge2$, and thus $\beta=\dfrac{(m-1)t+1}{mt+1}$ for some $t\ge1$. Therefore $(n-1)/n=\alpha=\dfrac{mt-m-t+1}{mt+1}$. Hence $n=m-\dfrac{m^2-1}{m+t}$. Set $\zeta=m+t$, then clearly $\zeta$ is a positive divisor of $m^2-1$. However since $n\ge2$, we must have $\zeta\ge\dfrac{m^2-1}{m-2}$, which in particular implies $m\ge3$, and hence the cutting sequence does involve both letters $V_3$ and $V_1$.

Finally assume $\beta\in(1,2)$. Then the cutting sequence reads
$$\overline{V_3\widetilde w(\beta_L)V_3\widetilde w(\iota(\beta_L))V_2\widetilde w(\alpha-1)}.$$
In particular we must have $\alpha=1-1/n$, and $\beta_L=1+1/m$ for some $m,n\ge2$. Going through the same calculation above, we have $n=-m+\dfrac{m^2-1}{m-t-1}$ for some $t\ge1$. Set $\zeta=m-t-1>0$. Then $t=m-1-\zeta\ge1$ and hence $\zeta\le m-2$. This implies $m\ge3$. Finally $n=-m+\dfrac{m^2-1}{\zeta}\ge-m+\dfrac{m^2-1}{m-2}>1$, as desired.
\end{proof}

\begin{prop}\label{prop: crown_case5}
Suppose $\beta_1=0$ but $\beta_3\neq0$, and $d=6$. Then
$$(\alpha,\beta)=\left(1-\frac{t}{2t^2-1},1+\frac{t}{2t^2-1}\right)$$
for some integer $t\ge2$.
\end{prop}
\begin{proof}
In this case, recall that we have $\beta_4=\iota(\alpha)+\iota(\beta)=1$ and $\beta_5=\iota(\alpha)$. First assume $\beta\in(0,1)$. Then $\beta_3=\alpha$. The cutting sequence for the core geodesic then reads
$$\overline{V_3\widetilde w(\beta_L)V_3\widetilde w(\iota(\beta_L))V_3\widetilde w(\alpha_R)V_3\widetilde w(\iota(\alpha_R))}.$$
Since $\iota(\alpha)+\iota(\beta)=1$, it is easy to see that $\alpha_R=1-\beta_L$. In particular we can get $\widetilde w(\alpha_R)$ by interchanging $V_1$ and $V_2$ in $\widetilde w(\beta_L)$. Therefore the cutting sequence above must involve all three letters unless $\beta_L=\alpha_R=1/2$. Thus $\beta=(t+1)/(2t+1)$ for some $t\ge1$ and $\alpha=t/(2t+1)$. However $\beta-\iota(\beta)=(t-1)/(2t+1)\neq\alpha$, a contradiction.

Therefore $\beta\in(1,2)$. Then $\beta_3=\alpha-1$, and the sequence now reads
$$\overline{V_3\widetilde w(\beta_L)V_3\widetilde w(\iota(\beta_L))V_2\widetilde w(\alpha_R-1)V_2\widetilde w(\iota(\alpha_R))}.$$
In particular, we must have $\beta_L=1+1/m$ and $\alpha_R=1-1/n$ for some $m, n\ge2$. Since $\iota(\alpha)+\iota(\beta)=1$, we have $\alpha_R+\beta_L-1=1$ and hence $m=n$. It then follows that $\beta=\dfrac{(m+1)t+m}{mt+m-1}$ and $\alpha=\dfrac{m-2+(m-1)t}{m-1+mt}$ for some $t\ge1$. Since $\beta-\iota(\beta)=\dfrac{(m+1)t}{mt+m-1}=\alpha$, we have $m=2t+2$. Hence $\alpha=\dfrac{2t^2+3t}{2t^2+4t+1}=1-\dfrac{t+1}{2t^2+4t+1}$ and $\beta=\dfrac{2t^2+5t+2}{2t^2+4t+1}=1+\dfrac{t+1}{2t^2+4t+1}$, as desired.
\end{proof}

Finally, collecting the cases, we have
\begin{thm}\label{thm: crown}
The crown $\Cr(\alpha,\beta)$ is hyperbolic with core geodesic on the boundary of $\core(M_A)$ if and only if $(\alpha,\beta)$ or $(1-\alpha,1-\beta)$ lies in one of the following $T$-orbits:
\begin{enumerate}[label=\normalfont{(\arabic*)}, topsep=0mm, itemsep=0mm]
\item $(1/m,-1/n)\to(1-1/n,1+1/m)$ for some integers $m,n\ge1$, with core geodesic $\overline{V_3^{m}V_1^n}$;
\item $(1/m,1/n)\to(1/n,1/m)$ for some integers $m,n\ge2$ and $(m,n)\neq(2,2)$, with core geodesic $\overline{V_3V_2^{m-2}V_3V_2^{n-2}}$;
\item $\left(0,\dfrac{mn+1}{n^2m+2n}\right)\to\left(\dfrac{mn+1}{n^2m+2n},0\right)$ for some integers $m\ge3,n\ge1$, with core geodesic $\overline{V_3V_1^{m-2}}$;
\item $\left(0,\dfrac{mn-1}{n^2m-2n}\right)\to\left(\dfrac{mn-1}{n^2m-2n},0\right)$ for some integers $m\ge3,n\ge2$, with core geodesic $\overline{V_3V_2^{m-2}}$;
\item $\left(\dfrac1n,\dfrac{\zeta+m}{m^2+m\zeta-1}\right)\to\left(\dfrac{m}{m^2+m\zeta-1},0\right)\to\left(0,\dfrac{m}{m^2+m\zeta-1}\right)\to\left(\dfrac{\zeta+m}{m^2+m\zeta-1},\dfrac1n\right)$ for some integers $m\ge2$, $\zeta$ a positive divisor of $m^2-1$, and $n=m+(m^2-1)/\zeta$, with core geodesic $\overline{V_3V_2^{m-2}V_3V_2^{m-2}V_3V_2^{n-2}}$;
\item $\left(\dfrac{n-1}n,1-\dfrac{\zeta-m}{-m^2+m\zeta+1}\right)\to\left(\dfrac{m}{-m^2+m\zeta+1},0\right)\to\left(0,\dfrac{m}{-m^2+m\zeta+1}\right)\to \linebreak \to  \left(1-\dfrac{\zeta-m}{-m^2+m\zeta+1},\dfrac{n-1}n\right)$ for some integers $m\ge3$, $\zeta\ge(m^2-1)/(m-2)$ a positive divisor of $m^2-1$, and $n=m-(m^2-1)/\zeta$, with core geodesic $\overline{V_3V_1^{m-2}V_3V_1^{m-2}V_3V_1^{n-2}}$;
\item $\left(\dfrac{n-1}n,1+\dfrac{m-\zeta}{m^2-m\zeta-1}\right)\to\left(\dfrac{m}{m^2-m\zeta-1},0\right)\to\left(0,\dfrac{m}{m^2-m\zeta-1}\right)\to\linebreak\to\left(\dfrac{m-\zeta}{m^2-m\zeta-1},-\dfrac1n\right)$ for some integers $m\ge3$, $\zeta\le m-2$ a positive divisor of $m^2-1$, and $n=-m+(m^2-1)/\zeta$, with core geodesic $\overline{V_3V_2^{m-1}V_3V_2^{m-1}V_3^{n-1}}$;
\item $\left(1-\dfrac{t}{2t^2-1},1+\dfrac{t}{2t^2-1}\right)\to\left(\dfrac{2t}{2t^2-1},0\right)\to\left(0,\dfrac{2t}{2t^2-1}\right)\to\left(\dfrac{t}{2t^2-1},-\dfrac{t}{2t^2-1}\right)\to\left(1-\dfrac{2t}{2t^2-1},1\right)\to\left(0,-\dfrac{2t}{2t^2-1}\right)$ for some integer $t\ge2$, with core geodesic $\overline{V_2^{2t-1}V_3V_2^{2t-1}}\allowbreak\overline{V_3^{2t-1}V_2V_3^{2t-1}}$.
\end{enumerate}
\end{thm}

\begin{rmk}
Note that the core geodesics above are given for the cases when $(\alpha,\beta)$ lies in the corresponding orbit listed; if $(1-\alpha,1-\beta)$ lies in the orbit, then we get the core geodesic by interchanging $V_1$ and $V_2$. Moreover, of course, the periodic part of the cutting sequence is only defined up to cyclic reordering. Finally, as discussed in previous sections, a cyclic reordering of $V_1,V_2,V_3$ gives the same conjugacy class of curves and hence the same geodesic.
\end{rmk}

\begin{proof}
Propositions~\ref{prop: crown_case1} and \ref{prop: crown_case2} give the cases $\beta_k\neq0,1$ for all $0\le k\le d-1$. Otherwise, by applying $\tau$ and $T$, we may assume $\beta_1=0$. Then Propositions~\ref{prop: crown_case3}, \ref{prop: crown_case4}, and \ref{prop: crown_case5} give the remaining cases.
\end{proof}

By looking through the list, we have the following observation:
\begin{cor}
There exists a closed geodesic on the boundary of $\core(M_A)$ that is not the core geodesic of a crown. On the other hand, there exists a closed geodesic on the boundary that is the core geodesic of two distinct crowns.
\end{cor}
\begin{proof}
It is easy to write down the periodic part of a closed geodesic that does not appear in the list above. For example $\overline{LRL^2RL^3RL^4R}$. On the other hand, $\Cr(1,-1)$ and $\Cr(0,4/5)$ both have core geodesic $\overline{V_3V_1}$ but they are distinct crowns.
\end{proof}

\section{Hyperbolic elementary surfaces: Double crowns}\label{sec: double_crown_class}
In this section we continue to classify elementary planes of Type~\ref{item: double_crown} in Proposition~\ref{prop: elementary_circle}, i.e.\ elementary planes $P$ so that $\core(P)$ is a double crown.

Since $\core(P)$ consists of two crowns, we may obtain the cutting sequence of the core geodesic, or that of the attracting fixed point of the corresponding hyperbolic element from either crown. Here is an example of such a surface.
\begin{eg}
Consider the crown $\Cr(1/3,2/3)$. Applying the algorithm, the cutting sequence for the attracting fixed point is $V_2\overline{V_1V_3V_2V_3}$. The line $l=l(1/3,2/3)$ passes through $1/2-i/2$, and it forms different angles with $y=0$ and $C_0$. Viewing $1/2-i/2$ as $\infty$ and the circle $C_1$ as the real line, the other side of the crown has marking $(-1,2)$. The cutting sequence obtained from this side is also $V_2\overline{V_1V_3V_2V_3}$. Therefore the crown extends to an elementary plane of Type~\ref{item: double_crown}.\qed
\end{eg}

We devote the remaining part of this section to obtaining marking symbols of elementary planes of Type~\ref{item: double_crown}. The discussions are technical but elementary; a reader only interested in the list may skip to Theorem~\ref{thm: double_crown}

 Throughout the section, assume $\Cr(\alpha,\beta)$ extends to a double crown. If $\Cr(\alpha,\beta)$ and $\Cr(\tilde\alpha,\tilde\beta)$ glue up to a double crown and the orientation on the core geodesics agree, we call $\Cr(\tilde\alpha,\tilde\beta)$ the \emph{complementary crown} of $\Cr(\alpha,\beta)$.

The fixed points of the hyperbolic element divides the line $l=l(\alpha,\beta)$ into two parts. For convenience, we call the part containing $\infty$ the ``outside", and the other part the ``inside". We say $l$ intersects $C_\gamma$ at $[[r]]$ if after applying $\eta_\gamma$, the intersection point becomes $r\in\mathbb{R}\cup\{\infty\}$. It may be helpful to extend this notation to other circles and other triple of points. Given a circle $C\subset\mathcal{A}$ and Farey neighbors $p_1,p_2,p_3\in C$, by $[[r]]_{p_1,p_2,p_3}^C$ we mean the point on $C$ which maps to $r\in\mathbb{R}\cup\{\infty\}$ after applying an element in $\Gamma_A$ sending $(p_1,p_2,p_3)$ on $C$ to $(\infty,0,1)$.

\begin{prop}
If $\beta_1=\beta_3=0$, then the period $d=2$, $\alpha=0$, and one of the following holds:
\begin{enumerate}[label=\normalfont{(\roman*)}, topsep=0mm, itemsep=0mm]
\item $\beta=\dfrac{nm^2-m-n}{m^2n^2-n^2-2mn+1}$, and a marking symbol for the complementary crown $\Cr(\tilde\alpha,\tilde\beta)$ is given by $\left(1,1-\dfrac{km^2-m-k}{m^2k^2-k^2-2mk+1}\right)$, where $n,k$ are integers $\ge2$ so that
$$m=\frac{(k+n)^2+2k^2n^2+\sqrt{(k^2-n^2)^2+4k^4n^4}}{2kn(n+k)}$$
is also an integer. The core geodesic for $\Cr(\alpha,\beta)$ has coding $\overline{V_1V_2^{m-2}V_3V_2^{m-2}}$.
\item $\beta=\dfrac{nm^2-m-n}{m^2n^2-n^2-2mn+1}$, and a marking symbol for the complementary crown $\Cr(\tilde\alpha,\tilde\beta)$ is given by $\left(0,\dfrac{km^2+m-k}{m^2k^2-k^2+2mk+1}\right)$, where $n>k$ are integers $\ge2$ so that
$$m=\frac{(k-n)^2+2k^2n^2+\sqrt{(k^2-n^2)^2+4k^4n^4}}{2kn(n-k)}$$
is also an integer. The core geodesic for $\Cr(\alpha,\beta)$ has coding $\overline{V_1V_2^{m-2}V_3V_2^{m-2}}$.
\end{enumerate}
\end{prop}
\begin{proof}
It is immediately clear that $d=2$ and $\alpha=0$. Moreover, we have $\beta=\iota(\beta)\in(0,1)$, and an application of the algorithm gives the cutting sequence of the attracting fixed point: $w(\beta_L)\overline{V_3\tilde w(\iota(\beta_L)_L)}$. The line $l:=l(\alpha,\beta)$ intersects the circle $C_{1/n}$ for some $n\ge1$. In order for $\Cr(\alpha,\beta)$ to be a double crown, the line $l$ must intersect $C_{1/n}$ at rational points. Moreover, we can obtain the cutting sequence from the inside. It reads $V_2^{n-1}V_1w(V_2,V_3)\cdots$ when $n\ge2$ and $\beta>1/n$ or $V_2^nw(V_1,V_3)\cdots$ when $\beta<1/n$.

\noindent\textbf{Case 1.}\quad We consider the former case first. If $w(V_2,V_3)$ contains any $V_3$, we must have $w(\beta_L)=V_2^{n-1}V_1V_2^{m-1}$. In particular $\beta_L=(m+1)/(mn+n-1)$ and hence $\iota(\beta_L)=n/(mn+n-1)$, giving $\iota(\beta_L)_L=1/(m+1)$. But then $\Cr(\alpha,\beta)$ is either elliptic or extends to a surface of Type~\ref{item: crown}. Therefore $w(V_2,V_3)=V_2^{m-2}$ for some integer $m\ge2$.

Thus the line $l$ either intersects the $C_{1/n}$ at $[[m]]$ for some integer $m\ge2$ (but the next marking symbol on the inside is not $1$), or at $[[\beta']]$ so that $\beta'_{L/R}=m$.

\noindent\textbf{\underline{Case 1.1.}}\quad First consider the former. Then it is easy to calculate that $\beta=\dfrac{nm^2-m+n}{n^2+n^2m^2-2mn}$, and $l$ intersects the circle $C_{1/n}$ also at $[[1/m]]$. Hence viewed from this point, the first marking symbol on the inside is $1-m/(m^2-1)$. More precisely, let $p_1=[[1/m]]$, $p_2=[[1/(m-1)]]$ and $p_2=[[0]]$, then $[[m]]=[[1-m/(m^2-1)]]_{p_1,p_2,p_3}^{C_{1/n}}$.

\noindent\textbf{\underline{Case 1.1.1.}}\quad Assume first that the next marking symbol on the inside lies in $(1,\infty)$ and before getting to the next lamppost, the next portion of the word consists entirely of $V_1$. Then either the next marking symbol is $1+1/k$, or the next marking symbol $\beta''$ satisfies $\beta_L''=1+1/k$, and the one after that is $0$.

For the former case, it is then easily calculated that the periodic orbit of marking symbols on the inside reads $(1-m/(m^2-1),1+1/k)\to(1/k,-m/(m^2-1))$.
Hence the cutting sequence for the attracting fixed point reads $V_2^{n-1}V_1V_2^{m-1}V_1^kV_2^{m-1}\overline{V_3V_2^{m-1}V_1^kV_2^{m-1}}$. In particular we have $\beta_L=[0;n-1,1,m-1,k,m]$, and hence $\beta=[0;n-1,1,m-1,k,m-1,1,t]$ for some integer $t\ge1$. The condition $\iota(\beta)=\beta$ gives $m(2+km)(-1+n-t)=0$, and thus $t=n-1$. On the other hand
$$\beta=\frac{nm^2-m+n}{n^2+n^2m^2-2mn}=\frac{1}{n-1+\frac{1}{1+\frac{1}{m-1+\frac{n-m}{nm}}}}.$$
Thus $\dfrac{nm}{n-m}=[k;m-1,1,n-1]=k+\dfrac{n}{nm-1}$. Since the right hand side is in lowest terms, we must have $nm-1$ divides $n-m$, but this is impossible when $m\ge2,n\ge2$.

Now consider the other case where $\beta''_L=1+1/k$. Then $\beta''=\dfrac{(k+1)l+k}{kl+k-1}$. This implies that we must have $1-\dfrac{m}{m^2-1}=\dfrac{(k+1)l}{kl+k-1}$. Given this, it is easy to calculate the cutting sequence, and hence $\beta_L=[0;n-1,1,m-1,k-1,1,k-1,m]$. The condition $\iota(\beta)=\beta$ then gives $\beta=[0;n-1,1,m-1,k-1,1,k-1,m-1,1,n-1]$. Arguing as above, we have $n+k(mn-1)$ divides $n-m$, again impossible.

\noindent\textbf{\underline{Case 1.1.2.}}\quad Assume next that the next marking symbol on the inside lies in $(1,\infty)$, but the portion before the next lamppost has a $V_3$. Then $\beta_L=[0;n-1,1,m-1,k]$. It is then easily calculated that the cutting sequence is $V_2^{n-1}V_1V_2^{m-1}V_1^{k-1}\overline{V_3V_1^{k-1}V_2^{m-1}}$, which cannot be the cutting sequence arising from the other side.

\noindent\textbf{\underline{Case 1.1.3.}}\quad Now assume the next marking symbol on this side lies in $(0,1)$ and before getting to the next lamppost, the next portion of the word consists entirely of $V_2$. Then either the next marking symbol is $1/k$, or the next marking symbol has left Farey neighbor $1/k$. As in previous cases, we can then calculate the cutting sequence. For the former case, we have $\beta_L=[0;n-1,1,m-2,1,k-2,1,m-1]$ and $\beta=[0;n-1,1,m-2,1,k-2,1,m-2,1,n-1]$. This implies $\dfrac{mn+n-m}{nm}=[0;1,k-2,1,m-2,1,n-1]$ and hence $kmn-n-k$ divides $mn$, which is impossible for $m,n,k\ge2$. For the latter case, $\beta_L=[0;n-1,1,m-2,1,k-2,1,k-2,1,m-1]$. Arguing as above, we must have $(k^2-1)(mn-1)-kn$ must divide $mn$, again impossible when $m,n,k\ge2$.

\noindent\textbf{\underline{Case 1.2.}}\quad We are left with the case that $l$ intersects $C_{1/n}$ at $[[\beta']]$ with $\beta_{L/R}'=m$, and the next marking symbol is $0/1$. It is easily checked that the next portion of the cutting sequence before the lamppost is complicated enough so that it has to contain a $V_3$. The cutting sequence then reads $V_2^{n-1}V_1V_2^{m-2}\overline{V_1V_2^{m-2}V_3V_2^{m-2}V_3V_2^{k-2}\cdots}$, which is not possible as the periodic part contains at least two $V_3$; or $V_2^{n-1}V_1V_2^{m-2}\overline{V_1V_2^{m-2}V_3V_2^{m-2}}$, and this implies $\beta=\dfrac{nm^2-m-n}{m^2n^2-n^2-2mn+1}$ from the sequence. On the other hand
$$\beta=\frac{n(m\pm1/k)^2+n-m\mp1/k}{n^2(m\pm1/k)^2-2n(m\pm1/k)+n^2}$$
from the fact that it intersects $C_{1/n}$ at $[[m\pm1/k]]$. Solving for $m$ we conclude
$$m=\frac{(k\mp n)^2+2k^2n^2+\sqrt{(k^2-n^2)^2+4k^4n^4}}{2kn(n\mp k)}.$$

\noindent\textbf{Case 2.}\quad The argument for $\beta<1/n$ and cutting sequence $V_2^nw(V_1,V_3)$ goes along the same way, so we just give a sketch. If $w(V_1,V_3)$ contains any $V_3$, we must have $w(\beta_L)=V_2^{n}V_1^{m-1}$ and hence $\beta_L=\dfrac{m}{nm+1}$. This implies $\beta=\frac{mn+1}{n^2m+2n}$, but then $\Cr(0,\beta)$ is a surface of Type~\ref{item: crown}.

Thus $w(V_1,V_3)=V_1^{m-1}$ for some $m\ge2$. So the line $l$ either intersects the circle at $[[-m]]$ for some integer $m\ge2$ (but the next marking symbol is not 0), or at $\beta'$ so that $\beta'_{L}=-m$.

The former case gives $\beta=\dfrac{m^2n+m+n}{m^2n^2+2mn+n^2}$. The next marking symbol on this side lies in $(-\infty,0)$. Suppose that this next marking symbol contributes some $V_3$ to the cutting sequence. Then $\beta_L=[0;n,m,k]$. But such a number cannot be the left neighbor of $\beta$ satisfying $\beta=\iota(\beta)$. So this next marking symbol contributes only $V_2$ to the cutting sequence. First assume the next marking symbol is $-1/k$ for some $k\ge1$. Then the marking symbols on this side reads $(1-1/k,1+m/(m^2-1))\to(m/(m^2-1),-1/k)$, and hence the cutting sequence reads $V_2^nV_1^{m-1}\overline{V_1V_2^{k}V_1^{m-1}V_3V_1^{m-2}}=V_2^nV_1^{m}V_2^kV_1^{m-1}\overline{V_3V_1^{m-1}V_2^kV_1^{m-1}}$ and thus $\beta_L=[0;n,m,k,m]$. This and $\beta=\iota(\beta)$ implies that $\beta=[0;n,m,k,m,n]$. With the other expression of $\beta$, we conclude that $\dfrac{mn}{m+n}=k+\dfrac{n}{mn+1}$. Therefore $mn+1$ must divide $m+n$, but this is impossible when $m,n\ge2$. Next assume this next marking symbol has right neighbor $-1/k$ and the marking symbol after that is $1$. Then we can calculate the cutting sequence to be $V_2^nV_1^{m-1}\overline{V_1V_2^{k-1}V_1V_2^{k-1}V_1^{m-1}V_3V_1^{m-2}}=V_2^nV_1^mV_2^{k-1}V_1V_2^{k-1}V_1^{m-1}\overline{V_3V_1^{m-1}V_2^{k-1}V_1V_2^{k-1}V_1^{m-1}}$. Thus $\beta_L=[0;n,m,k-1,1,k-1,m]$ and $\beta=[0;n,m,k-1,1,k-1,m,n]$. Thus $\dfrac{mn}{m+n}=k-\dfrac{mn+1}{kmn+k+n}$, and so $k(mn+1)+n$ divides $m+n$, again impossible.

We are left with the case $\beta'_L=-m$. As above, the cutting sequence then reads either $V_2^nV_1^{m-1}\overline{V_2V_1^{m-2}V_3V_1^{m-2}V_3V_1^{k-2}\cdots}$, which is impossible as there can only be one $V_3$ in the periodic part; or $V_2^nV_1^{m-1}\overline{V_2V_1^{m-2}V_3V_1^{m-2}}$, and thus $\beta=\dfrac{m^2n+m-n}{m^2n^2-n^2+2mn+1}$. On the other hand,
$$\beta=\frac{n(m-1/k)^2+m-1/k+n}{n^2(m-1/k)^2+2n(m-1/k)+n^2}.$$
Solving for $m$, we have the same expression as in the previous case (indeed, another way to think about this particular case is that if such a double crown exists, the other side would give the case $\beta>1/n$).
\end{proof}
\begin{rmk}
Unfortunately, we are not sure if there are actually positive integers $n,k$ giving integral $m$ via the formulae in the proposition above. In fact, a search among $n,k\le 100$ gives no example for which $(k^2-n^2)^2+4k^4n^4$ is a square unless $n=k$, and it is easy to see that when $n=k$, $m$ is not an integer. We make the following conjecture:
{\theoremstyle{plain}
\newtheorem*{conj_dioph}{Conjecture A}
\begin{conj_dioph}
Any integer solution to $(k^2-n^2)^2+4k^4n^4=t^2$ satisfies either $n=\pm k$, or $n=0$, or $k=0$.
\end{conj_dioph}}
Should the conjecture hold, it is clear that the cases described in the proposition above cannot happen. On the other hand, whether the conjecture is true or not does not affect our discussions on topology and geometry in the next section.
\end{rmk}

For the remaining cases, we assume $\beta_1$ and $\beta_3$ are not both $0$, nor are they both $1$. First we have the following three complementary propositions:
\begin{prop}
Assume
\begin{itemize}[topsep=0mm, itemsep=0mm]
\item $\beta_1$ and $\beta_3$ are not both $0$, nor are they both $1$;
\item For all $k$, either $\beta_k\in\mathbb{Z}$ or $\alpha_k=0,1$.
\end{itemize}
Then $(\alpha,\beta)=(0,3)$, $(1,-2)$, $(0,7/4)$, $(1,-3/4)$, $(1/4,2)$ or $(3/4,-1)$.
\end{prop}
\begin{proof}
First assume $\beta_k\in\mathbb{Z}$ for all $k$. Then $\alpha_k=0,1$. Applying $\tau$ if necessary, assume $\alpha=0$ and $\beta$ is a positive integer. If $\beta\ge4$, then the line $l(0,\beta)$ intersects $C_0,C_1,C_2$ in different angles, and in particular $\Cr(0,\beta)$ extends to a surface with at least three different ends. On the other hand $\beta=1$ and $\beta=2$ give elliptic elementary plane and single crown. Finally, it is easy to check that $\Cr(0,3)$ indeed extends to a double crown.

Now assume for some $k$, $\beta_k\notin\mathbb{Z}$. Applying $T$ if necessary, assume this is true for $\beta$. Then $\alpha=0,1$. Applying $\tau$ if necessary, assume $\alpha=0$. Then $\alpha_1=\iota(\beta)\neq0,1$, so $\beta_1=\iota(\beta)-\beta\in\mathbb{Z}$. Since $\iota(\beta)\in(0,1)$, and $\beta>0$, we have $\beta_1<1$. The case $\beta_1=0$ is ruled out assumption. If $\beta_1\le-2$, then $\beta=\iota(\beta)-\beta_1>2$. Then the line $l(0,\beta)$ intersects $C_0,C_1,C_2$ in different angles, and $\Cr(\alpha,\beta)$ extends to a nonelementary surface. So $\beta_1=-1$, and $\beta=\iota(\beta)+1$.

Set $\epsilon=\beta-1$. We can then calculate the cutting sequence $V_1\widetilde w(\epsilon)\overline{V_2V_3\widetilde w(\epsilon)}$. For this to be a double crown, $\widetilde w(\epsilon)$ must contain a $V_1$. This in turn implies that the first block of letters before the first lamppost (which has to be $V_1$) from the complementary crown must be entirely $V_2$.

If the line $l(0,\beta)$ intersects $C_1$ at $[[m+1]]$ for some integer $m\ge2$, then we can calculate $\beta=1+\dfrac{m+1}{m^2}$. Then $\iota(\beta)=\dfrac{m^2-m+1}{m^2}=\beta-1=\dfrac{m+1}{m^2}$, and thus $m=2$. It is easy to test that $\Cr(0,7/4)$ does give a double crown. If the line intersects $C_1$ at $[[m\pm1/k]]$ with the next marking symbol being $0,1$ respectively, one can check that either the periodic part contains at least two $V_3$, or the marking symbols of the complementary crown violates the first assumption in the proposition; neither is possible.
\end{proof}

\begin{prop}
Assume
\begin{itemize}[topsep=0mm, itemsep=0mm]
\item $\beta_1$ and $\beta_3$ are not both $0$, nor are they both $1$;
\item $\alpha\neq0,1$ and $\beta\notin\mathbb{Z}$;
\item The line $l:=l(\alpha,\beta)$ intersects two circles $C_k,C_{k+1}$ for some integer $k$.
\end{itemize}
Then we have $(\alpha,\beta)$ or $(1-\alpha,1-\beta)$ is one of the following:
\begin{enumerate*}[label=\normalfont{(\roman*)}]
\item $(1/n,1-1/n)$ for some integer $n\ge3$;
\item $(1-1/n,2+1/n)$ for some integer $n\ge2$;
\item $(2/7,5/7)$;
\item $(1/3,8/3)$.
\end{enumerate*}
\end{prop}
\begin{proof}
Note that $l$ must pass through the tangency points of $C_k$ and $C_{k+1}$. Indeed, either $l$ passes through $C_k$ or $C_{k+1}$ after finitely many circles after $y=0$ (in which case we have an elliptic crown), or $C_k$ and $C_{k+1}$ is on the other side forming the complementary crown. In particular, $l$ intersects $C_k$ and $C_{k+1}$ in the same angle. It is then easy to show that $l$ passes through the point of tangency $k+1/2-i/2$.

Notice that the complementary crown satisfies the same assumptions, except possibly violating the condition $\alpha\neq0,1$ and $\beta\notin\mathbb{Z}$. Checking the list in the previous proposition, we note that $(0,3)$ is such a pair, and a marking for its complementary crown is $(1/3,2/3)$.

For simplicity, we choose a different normalization than $\symb$, and assume $\alpha+\beta=1, \beta>\alpha$. Then $l(\alpha,\beta)$ intersects $C_0$ and $C_1$. Moreover, it is easy to see $\beta\in(1/2,2)$.

Again, as we have assumed that $\Cr(\alpha,\beta)$ is a double crown, the cutting sequence read from either side should agree. Then the cutting sequence on one side should give $V_2V_1^k$, $V_1V_2^k$ or $V_1V_2V_1^k$ before getting to the first lamppost. By symmetric nature of the case, we may assume this is on the side of $\infty$. Hence $\beta=1-1/n, 1+1/n, 2-1/n$ or $\beta_L=1-1/n,1+1/n,2-1/n$ for some $n\ge2$.

For the former, it is easy to check that each case for $\beta=1-1/n,n\ge3$, $\beta=1+1/n,n\ge2$, and $\beta=5/3$ indeed give double crowns, while the remaining cases are not. For the latter, since $\beta_1=0$, we can easily check that it is only possible that $\beta_L=1-1/n$. Thus $\beta=\dfrac{(n-1)t+1}{nt+1}$, and then $\iota(\beta)=2\beta-1$, which gives $t=\dfrac{n-1}{n-2}$, which is not an integer unless $n=3$. Then $t=2$, and $\beta=5/7$. It is easy to check this also gives a double crown.
\end{proof}

\begin{prop}
Assume
\begin{itemize}[topsep=0mm, itemsep=0mm]
\item $\beta_1$ and $\beta_3$ are not both $0$, nor are they both $1$;
\item $\alpha\neq0,1$ and $\beta\notin\mathbb{Z}$;
\item The line $l:=l(\alpha,\beta)$ intersects only $C_0$ or $C_1$.
\end{itemize}
Then $(\alpha,\beta)$ or $(1-\alpha,1-\beta)$ is one of the following:
\begin{enumerate}[label=\normalfont{(\roman*)}, topsep=0mm, itemsep=0mm]
\item $\left(\dfrac{n}{nm-1},-\dfrac{m}{nm-1}\right)$, $\left(\dfrac{m}{nm+1},\dfrac{n}{nm+1}\right)$ for some integers $n,m\ge2$;
\item $\left(\dfrac{2n+1}{4n},-\dfrac1n\right)$, $\left(\dfrac1{n}, \dfrac{2n+1}{4n}\right)$ for some integer $n\ge2$.
\end{enumerate}
\end{prop}
\begin{proof}
Applying $\tau$ If necessary, assume $l$ only intersects $C_0$. We assume $\beta<0$ and give a detailed analysis. The case $\beta>0$ is entirely analogous.

Again, the cutting sequences from both sides agree; in particular, the first block before a lamppost on one side must consist of only one letter. This is true for both the attracting and repelling fixed points. We have the following possibilities:

\noindent\textbf{Case 1.}\quad  Blocks of a single letter appear on the outside for both fixed points. Then either $(\alpha,\beta)=(1/m,-1/n)$, or $\alpha=1/m$ and $\beta_R=-1/n$, or $\alpha_L=1/m$ and $\beta=-1/n$, or $\alpha_L=1/m$ and $\beta_R=-1/n$. In all of these cases, one can check that $\Cr(\alpha,\beta)$ has core geodesic on $\partial\core(M_A)$. (cf. Theorem~\ref{thm: crown}).

\noindent\textbf{Case 2.}\quad Blocks of a single letter appear on the inside for both fixed points. We further have the following cases:

\noindent\textbf{\underline{Case 2.1.}}\quad The line $l$ intersects $C_0$ at $[[-n]]$ and $[[1/m]]$ for some integers $n,m\ge2$. Then it is easily calculated that $(\alpha,\beta)=(n/(nm-1),-m/(nm-1))$. The cutting sequence for the attracting fixed point is given by $V_1^{-1}V_2V_1^{n-2}\overline{V_2V_1^{m-1}V_3^{n-1}V_2V_3^{m-1}V_1^{n-1}}$.  A marking for the complementary crown is $(1-m/(nm+1),1-n/(nm+1))$, which gives the same cutting sequence.

\noindent\textbf{\underline{Case 2.2.}}\quad The line $l$ intersects $C_0$ at $[[-n]]$ and $\left[\left[\dfrac{1+k}{m+1+mk}\right]\right]$ for some integers $m,n,k\ge2$, and for the repelling fixed point, the next marking symbol is $1$ on the inside. Then one can calculate that $\alpha=\dfrac{n(k+1)}{(m+1+mk)n-1-k}$ and $\beta=-\dfrac{m+1+mk}{(m+1+mk)n-1-k}$ where $k+1=mn/(m-n)$.

Since $k$ is a positive integer, we have $k+1=dab$, $m=d(a-b)a$ and $n=d(a-b)b$ for some integers $d,a,b$ so that $a>b$ and $\gcd(a,b)=1$. Now the cutting sequence starts with $V_1^{-1}V_2V_1^{n-2}V_2V_1^{m-1}V_2^{t}$, where $t=k$ or $k-1$. On the inside, the marking symbols start with
$$\left(0,\frac{mn+1}{(m+1+mk)n+k+1}\right)\to\left(\frac{mk+m+1}{n(m+1+mk)+k+1},\frac{mk+m-nm}{n(m+1+mk)+k+1}\right)\to\cdots$$
The cutting sequence then reads $V_1^{-1}V_2V_1^{n-2}V_2w(V_1,V_3)\cdots$. Since
$$\dfrac{mk+m-nm}{n(m+1+mk)+k+1}=\dfrac{d(a-b)ab}{d^2(a-b)^2a^2b+2a-b},$$
the segment $w(V_1,V_3)$ thus reads either $V_1^{m-1}V_3\cdots$ (which does not agree with the cutting sequence from the outside), or $V_1^{m-2}V_2\cdots$ (which does not agree either), or $V_1^{m-1}V_2V_1^{m-1}\cdots$, which forces $k=2$, and hence $d=1,a=3,b=1$. This implies $m=6,n=2$, giving $(\alpha,\beta)=(6/35,-19/35)$. But it can be easily checked that $C(6/35,-19/35)$ does not extend to a double crown.

\noindent\textbf{\underline{Case 2.3.}}\quad The line $l$ intersects $C_0$ at $[[-n]]$ and $\left[\left[\dfrac{k+1}{mk+m-1}\right]\right]$ for some integers $m,n,k\ge2$, and for the repelling fixed point, the next marking symbol is $0$ on the inside. It then follows that $\alpha=\dfrac{n(k+1)}{(mk+m-1)n-k-1}$ and $\beta=-\dfrac{mk+m-1}{(mk+m-1)n-k-1}$ where $k+1=mn/(n-m)$.

As in the previous case, we have $k+1=dab,m=d(a-b)b$ and $n=d(a-b)a$ for some integers $d,a,b$ so that $a>b>0$ are relatively prime. Now the cutting sequence reads
$$V_1^{-1}V_2V_1^{n-2}V_2V_1^{m-2}V_2V_1^kw(V_2,V_3)\cdots.$$
The marking symbols on the inside start with
$$\left(1,\frac{(mk-1)n+k}{(mk+m-1)n+k+1}\right)\to\left(\frac{mk+m-1}{(mk+m-1)n+k+1},\frac{mk+m+mn}{(mk+m-1)n+k+1}\right)\to\cdots$$
Since $\dfrac{mk+m+mn}{(mk+m-1)n+k+1}=\dfrac{d(a-b)ab}{d^2(a-b)^2b^2a+2b-a}$, the cutting sequence on the inside reads $V_1^{-1}V_2V_1^{n-2}V_2w(V_1,V_3)$. This forces $k=m-1$ or $k=m$. The former is impossible, and the latter forces $b=d=1$. We also need $a-2$ divides $a(a-1)$, which forces $a=3$ or $a=4$. This gives $(m,n,k)=(2,6,2)$ or $(3,12,3)$, and hence $(\alpha,\beta)=(2/3,-5/27)$ or $(3/8,-11/128)$, neither of which gives a double crown.

\noindent\textbf{\underline{Case 2.4.}}\quad The line $l$ intersects $C_0$ at a point with left (resp. right) Farey neighbor $[[-n]]$ and another point with right (resp. left) Farey neighbor $[[1/m]]$. A different normalization gives that the marking symbols on the inside has period $2$ starting with $(0,\beta)$. This case is ruled out by the assumption.

\noindent\textbf{Case 3.}\quad Blocks of a single letter appear on different sides for the two fixed points. By symmetry, assume that this block appears on the outside for the attracting fixed point and inside for the repelling fixed point.

\noindent\textbf{\underline{Case 3.1.}}\quad The line $l$ intersects the real line at $-1/n$ and the circle $C_0$ at $[[1/m]]$ for some integers $n,m\ge2$. Then we have $\alpha=\dfrac{nm+1}{nm^2}$. The first few marking symbols on the outside are given by
$$\left(\frac{nm+1}{nm^2},-\frac1n\right)\to\left(\frac{n-1}n,1+\frac{nm+1}{nm^2}\right)\to\left(\frac{nm^2-nm+1}{nm^2},\frac{nm^2-2nm-m^2}{nm^2}\right)\to\cdots$$
Unless $nm^2=2nm+m^2$ (which is only possible if $n=m=3$, or $n=2$ and $m=4$, neither producing a double crown), the first few terms in the cutting sequence for the attracting fixed point is
$$V_1^{-1}V_2V_1^{n-1}V_3^{m-1}V_2V_3^{n-1}V_2^{m-1}\cdots$$
The marking symbols on the inside can be calculated as
$$\left(1-\frac{m}{2+mn},1-\frac{nm+1}{m(nm+2)}\right)\to\left(1-\frac{nm+1}{m(nm+2)},1-\frac{m}{nm+2}\right)$$
Comparing with the sequence on the outside, we must have $m=2$. It is easily checked that for any $n$, the cutting sequences from both sides agree:
$$V_1^{-1}V_2\overline{V_1^{n-1}V_3V_2V_3^{n-1}V_2V_3V_1}.$$

\noindent\textbf{\underline{Case 3.2.}}\quad The line $l$ intersects the real line at $\beta$ with $\beta_R=-1/n$, and the circle $C_0$ at $[[1/m]]$ for some integer $n,m\ge2$, and the next marking symbol on the outside is $1$. It is easily checked that this is impossible.

\noindent\textbf{\underline{Case 3.3.}}\quad The line $l$ intersects the real line at $-1/n$, and the circle $C_0$ at a point with left Farey neighbor $[[1/m]]$ or right Farey neighbor $[[1/m]]$. For the former case, assume the line intersects the circle $C_0$ at $\dfrac{k+1}{mk+m-1}$. Then as the next marking symbol on the inside for the repelling fixed point needs to be $0$, we have $n=\dfrac{(k+1)^2(m^2-1)}{(k+1-m)(km+m-1)}=\dfrac{t^2(m^2-1)}{(t-m)(tm-1)}$, where we set $t=k+1$. Since $n$ is an integer, and $\gcd(t,tm-1)=1$, it follows that $tm-1$ divides $m^2-1=(m+1)(m-1)$. Now as $\gcd(tm-1,m-1)=\gcd(t-1,m-1)$, we have $tm-1$ divides $(t-1)(m+1)$. However $tm-1-(t-1)(m+1)=t-m>0$, a contradiction. The other case is very similar and omitted.

\noindent\textbf{\underline{Case 3.4.}}\quad The line $l$ intersects the real line at $\beta$ with $\beta_R=-1/n$, and the circle $C_0$ at a point with left Farey neighbor $[[1/m]]$ or right Farey neighbor $[[1/m]]$. Then $\beta=\dfrac{-k-1}{nk+n-1}$, and $\beta_1=1$. Suppose $l$ intersects $C_0$ at a point $\dfrac{l+1}{ml+m\pm1}$. It follows that
$$k=-1+\frac{(ml+m\pm1)(1+l+n(ml+m\pm1))}{(1+l)^2+(ml+m\pm1)^2+n(1+l)(ml+m\pm1)}=-1+\frac{(mt\pm1)(t+(mt\pm1)n)}{t^2+(mt\pm1)(mt+nt\pm1)},$$
where we set $t=l+1$. Since $k$ is an integer, and $\gcd(mt\pm1,t^2+(mt\pm1)(mt+nt\pm1))=1$, we must have $t^2+(mt\pm1)(mt+nt\pm1)$ divides $t+(mt\pm1)n$. This cannot be, as the former is greater than the latter.
\end{proof}

Collecting all the propositions, we have
\begin{thm}\label{thm: double_crown}
The crown $\Cr(\alpha,\beta)$ extends to a double crown if and only if $(\alpha,\beta)$ or $(1-\alpha,1-\beta)$ lies in one of the following pairs of orbits (each pair gives complementary crowns):
\begin{enumerate}[label=\normalfont{(\arabic*)}, topsep=0mm, itemsep=0mm]
\item $(1/n,1-1/n)\to(1-1/n,n)$ for some integer $n\ge3$, and on the other side $(1-1/(n-2),2+1/(n-2))\to(1/(n-2),-1-1/(n-2))$, with core geodesic $\overline{V_3V_2^{n-2}V_3V_1^{n-2}}$;
\item $(2/7,5/7)\to(3/7,0)\to(0,3/7)\to(5/7,2/7)\to(4/7,1)\to(1,4/7)$, and on the other side $(1/3,8/3)\to(2/3,-5/3)$, with core geodesic $\overline{V_3V_2V_3V_2V_3V_1V_3V_1}$;
\item $\displaystyle\left(\frac{n}{nm-1},-\frac{m}{nm-1}\right)\to\left(1-\frac{n}{nm-1},1+\frac{m}{nm-1}\right)$ for some integers $m,n\ge2$, and on the other side $\displaystyle\left(1-\frac{m}{nm+1},1-\frac n{nm+1}\right)\to\left(\frac{m}{nm+1},\frac{n}{nm+1}\right)$, with core geodesic $\overline{V_2V_1^{m-1}}\allowbreak\overline{V_3^{n-1}}\allowbreak\overline{V_2V_3^{m-1}V_1^{n-1}}$;
\item $\displaystyle\left(\frac{2n+1}{4n},-\frac1n\right)\to\left(\frac{n-1}n,\frac{6n+1}{4n}\right)$ for some integer $n\ge1$, and on the other side $\displaystyle\left(\frac{n}{n+1},\right.\linebreak\left.\frac{2n+3}{4n+4}\right)\to\left(\frac{2n+3}{4n+4},\frac{n}{n+1}\right)$, with core geodesic $\overline{V_1^nV_3V_2V_3^{n-1}V_2V_3}$;
\item $\left(0,\dfrac{nm^2-m-n}{m^2n^2-n^2-2mn+1}\right)\to\left(\dfrac{nm^2-m-n}{m^2n^2-n^2-2mn+1},0\right)$, and on the other side\\$\left(1,1-\dfrac{km^2-m-k}{m^2k^2-k^2-2mk+1}\right)\to\left(1-\dfrac{km^2-m-k}{m^2k^2-k^2-2mk+1},1\right)$, where $n,k\ge2$ are integers so that
$$m=\frac{(k+n)^2+2k^2n^2+\sqrt{(k^2-n^2)^2+4k^4n^4}}{2kn(n+k)}$$
is also an integer. The core geodesic has coding $\overline{V_1V_2^{m-2}V_3V_2^{m-2}}$;
\item $\left(0,\dfrac{nm^2-m-n}{m^2n^2-n^2-2mn+1}\right)\to\left(\dfrac{nm^2-m-n}{m^2n^2-n^2-2mn+1},0\right)$, and on the other side\\$\left(0,\dfrac{km^2+m-k}{m^2k^2-k^2+2mk+1}\right)\to\left(\dfrac{km^2+m-k}{m^2k^2-k^2+2mk+1},0\right)$, where $n>k\ge2$ are integers so that
$$m=\frac{(k-n)^2+2k^2n^2+\sqrt{(k^2-n^2)^2+4k^4n^4}}{2kn(n-k)}$$
is also an integer. The core geodesic has coding $\overline{V_1V_2^{m-2}V_3V_2^{m-2}}$;
\end{enumerate}
\end{thm}

\section{Geometry and topology of elementary planes}\label{sec: geom_top_elem_planes}
In this section we discuss the geometry and topology of elementary planes in the Apollonian orbifold $M_A$, and prove Theorems~\ref{thm: topologies}, \ref{thm: elem_closed}, and \ref{thm: complexity}.

\subsection{Individual elementary planes}
We first discuss the geometry of individual elementary planes, and how they are immersed in $M_A$.
\paragraph{Topological types}
By going through the lists, we immediately have the topological types described in Theorem~\ref{thm: topologies}.
\paragraph{Area}
By Gauss-Bonnet, the hyperbolic area of an ideal triangle is $\pi$. From there we have (or we can apply Gauss-Bonnet again)
\begin{itemize}[topsep=0mm, itemsep=0mm]
\item The area of an ideal $n$-gon is $(n-2)\pi$;
\item The area of a punctured ideal $n$-gon is $n\pi$;
\item The area of a crown with $n$ tips is $n\pi$.
\end{itemize}
From the list in Theorem~\ref{thm: elliptic}, if we do not consider torsion, an elliptic crown in the Apollonian orbifold has area $\pi$, $2\pi$ or $4\pi$. From the list in Theorem~\ref{thm: parabolic}, again without considering torsion, a parabolic crown in $M_A$ has area $\pi$ or $2\pi$. From the list in Theorem~\ref{thm: crown}, a hyperbolic crown with core geodesic on $\partial\core(M_A)$ has area $\le 6\pi$. Finally from the list in Theorem~\ref{thm: double_crown}, a double crown has area $\le8\pi$. If we take into consideration torsion, the area would be even smaller. In particular, the area of any elementary plane in $\core(M_A)$ is bounded above by $8\pi$.

\paragraph{Complexity of markings / modular symbols}
By going through the lists, we note that the continued fractions of the rational numbers in the symbols of elementary planes have length $\le8$. Geometrically, this implies that each component of the corresponding modular symbols makes excursion into the unique cusp of $X=\partial\core(M_A)$ at most 8 times.

Indeed, each block $V_{i}^k$ represents a rotation around the cusp. To change from a letter to another, the geodesic has to return to the thick part of $X$. So the number of times letters change in the word gives the number of excursions the geodesic makes into the cusp. It is also (roughly) the number of times a geodesic intersects the thick part of $X$.

\paragraph{Complexity of core geodesics}
Given a word $w$ in $V_1$, $V_2$ and $V_3$, write $w=V_{i_1}^{a_1}V_{i_2}^{a_2}\cdots V_{i_d}^{a_d}$ so that $i_k\in\{1,2,3\}$ and $i_{k}\neq i_{k+1}$; its \emph{reduced length} is then defined to be $d$.

By going through the lists in Theorems~\ref{thm: crown} and \ref{thm: double_crown}, we note that the words for the core geodesics have reduced length $\le8$. Geometrically, this implies that each of these geodesics makes excursion into the unique cusp of $M_A$ at most 8 times, by a similar argument as above.

The discussions above give Theorem~\ref{thm: complexity}. Therefore, each elementary plane is properly immersed in $M_A$ in a very controlled way.

\subsection{Sociology of elementary planes}
In this subsection we discuss the collective behavior of elementary planes, and prove Theorem~\ref{thm: elem_closed}.

\subsubsection{Boundary data are closed}
To illustrate Theorem~\ref{thm: elem_closed}, we first show that the boundary data of the elementary planes are ``closed" in some sense.
\paragraph{Topology of modular symbols}
Recall that $\mathcal{S}^d$ is the space of degree $d$ modular symbols on $X=\psl(2,\mathbb{Z})\backslash\mathbb{H}^2$, and $\mathcal{S}=\bigcup_{d\ge1}\mathcal{S}^d$. Each modular symbol $\sigma$ of degree $d$ can be represented by a concatenation of $d$ complete oriented geodesics from the cusp to the cusp on the modular surface $X$, and gives a closed based loop $\sigma^*:[0,1]\to\overline{X}$ in the one-point compactification $\overline{X}=\psl(2,\mathbb{Z})\backslash(\mathbb{H}^2\cup\mathbb{Q})$, with base point at the ``filled in" cusp.

The space $\mathcal{S}$ carries a natural geometric topology, defined as follows: a sequence $\sigma_n\to\sigma\in\mathcal{S}$ if $\sigma_n^*\to\sigma^*$ as paths in the compact surface $\overline{X}$. More precisely, $\sigma_n\to\sigma$ if there exists homeomorphisms $\phi_n:[0,1]\to[0,1]$ fixing the endpoints so that $(\sigma_n^*\circ\phi_n)(t)\to\sigma^*(t)$ uniformly on $[0,1]$. For more properties of $\mathcal{S}$ together with this topology, we refer to \cite{modular_symbol}.

In this topology, we have the closure $\overline{\mathcal{S}^{d_0}}=\cup_{d\ge d_0}\mathcal{S}^d$. Indeed, if a sequence of geodesics make a deeper and deeper excursion into the cusp, eventually the limit will pass through the cusp, and break up into two geodesics. In particular $\mathcal{S}^1$ is dense in $\mathcal{S}$ with respect to the geometric topology.

Similarly, the space of cyclic modular symbols $\mathcal{S}_{\text{cyc}}=\bigcup_{d\ge1}\mathcal{S}^d_{\text{cyc}}$ also carries a geometric topology, making the natural projection $\mathcal{S}\to\mathcal{S}_{\text{cyc}}$ (sequentially) continuous.

\paragraph{Geometric topology in terms of continued fractions}
As discussed in \cite{modular_symbol}, taking a limit in $\mathcal{S}$ can be easily represented in terms of continued fractions. Recall that every modular symbol $\sigma$ of degree 1 on $X$ can be represented by a rational number $p/q\in[0,1]$ so that
$$\sigma=[p/q]=\{\infty,p/q\}.$$
Let $p/q=[0;a_1,a_2,\ldots, a_n]$ be its continued fraction. As in \cite{modular_symbol}, we then use the notation
$$\sigma=\langle a_1,a_2,\ldots,a_n\rangle.$$
In particular $[0]=\langle\rangle=[1]=\langle1\rangle$. Since continued fraction representation is not unique for rational numbers, this notation is also not unique.

Deeper excursion into the cusp corresponds to some $a_i\to\infty$. Therefore, we can represent modular symbols of higher degree using the notation above by allowing $a_i=\infty$. By replacing $\infty$ with $\ast$ and adding brackets as necessary, we can then convert it into a more familiar form. For example $\langle\infty\rangle=\langle\rangle\ast\langle\rangle=\langle1\rangle\ast\langle1\rangle$, $\langle 3,\infty,5/6,7,\infty,\infty,4\rangle=\langle3\rangle\ast\langle5/6,7\rangle\ast\langle\rangle\ast\langle4\rangle$, etc.

\paragraph{Closed geodesics}
We remark that we can add
$$\mathcal{S}^0:=\{\gamma:\gamma\text{ is an oriented closed geodesic on }X\}$$
to the picture, as ``modular symbols of degree 0". As discussed above, deeper and deeper excursions into the cusp produce geodesics passing through the cusp in the limit. So a sequence of closed geodesics limit on a modular symbol in this way. However, since there is no canonical way to choose a base point, the limit is really only defined up to cyclic reordering, or as an element of $\mathcal{S}_{\text{cyc}}$.

Moreover, the cutting sequence of $\gamma$ replaces the role of continued fraction in the discussion above. To take limit, we can replace letters that have infinite exponent with $\ast$, and read the rest as a continued fraction representation. For example, $\overline{L^nR}$ gives $[0]$ in the limit as $n\to\infty$, and $\overline{L^nRL^2R^nL^3R}$ gives $\langle1,2\rangle\ast\langle3,1\rangle$, etc. Indeed, this is immediately clear from the correspondence between cutting sequences and continued fractions, as explained in \S\ref{sec: farey_diophantine}.

\paragraph{Boundary data of elementary planes}
To give an illustration of Theorem~\ref{thm: elem_closed}, we first look at the collection of all boundary components of elementary planes. Define
$$\mathcal{E}:=\{\sigma:\sigma\subset X\text{ is a component of $\partial\core(M_A)\cap P$ for some elementary plane } P\}\subset\mathcal{S}^0\cup\mathcal{S}^1$$
where we identify $\partial\core(M_A)\cong X$. Set $S(\sigma):=\begin{cases}\text{the set of degree 1 components of }\sigma&\sigma\in\mathcal{S}\\\{\sigma\}&\sigma\in\mathcal{S}^0\end{cases}$. Then define
$$\mathcal{E}^*:=\bigcup_{\substack{\{\sigma_i\}\subset\mathcal{E}\\\sigma_i\to\sigma}}S(\sigma).$$
Clearly $\mathcal{E}\subset\mathcal{E}^*$ by taking $\{\sigma_i\}$ to be a constant sequence. We claim
\begin{prop}\label{prop: symbols_closed}
The two sets $\mathcal{E}^*$ and $\mathcal{E}$ are equal. In other words, the set of boundary data of elementary planes is stable under taking limit in the geometric topology.
\end{prop}
\begin{proof}
This can be checked case by case from the lists in Theorems~\ref{thm: elliptic}, \ref{thm: parabolic}, \ref{thm: crown}, and \ref{thm: double_crown}. For example, for every triple of integers $(m,n,\zeta)$ so that $m\ge2$,  $\zeta$ is a positive divisor of $m^2-1$, and $n=m+(m^2-1)/\zeta$, the cycle of symbols
\begin{equation}\label{eq: symbols}
\left(\dfrac1n,\dfrac{\zeta+m}{m^2+m\zeta-1}\right)\to\left(\dfrac{m}{m^2+m\zeta-1},0\right)\to\left(0,\dfrac{m}{m^2+m\zeta-1}\right)\to\left(\dfrac{\zeta+m}{m^2+m\zeta-1},\dfrac1n\right)\tag{$\star$}
\end{equation}
give a hyperbolic crown whose core geodesic lies on $\partial\core(M_A)$. The word for the core geodesic is given by $\overline{V_3V_2^{m-2}V_3V_2^{m-2}V_3V_2^{n-2}}$. To get an infinite sequence in this family, we need $m\to\infty$, which implies $n\to\infty$. We have
$$\dfrac{\zeta+m}{m^2+m\zeta-1}=[0;m-1,1,m+\zeta-1],\quad\dfrac{m}{m^2+m\zeta-1}=[0;m+\zeta-1,1,m-1].$$
Thus in the limit, every component is $\langle\rangle=[0]=[1]=\langle1\rangle$, which is clearly contained in $\mathcal{E}$ already. All the other cases can be checked similarly.
\end{proof}
We remark that the reason we only look at degree 1 components in the proposition above is that after taking limit, portions coming from different components may combine to form the cycle of symbols for the limit surface.

\subsubsection{The union of elementary planes is closed}
We are now in the position to prove Theorem~\ref{thm: elem_closed}.
\begin{proof}
Of course, we can prove it case by case as in Proposition~\ref{prop: symbols_closed}. Here we give a proof using Proposition~\ref{prop: symbols_closed}.

Let $P_n$ be a sequence of elementary planes in $M_A$, and $C_n$ a boundary circle of $P_n$. Suppose $C_n\to C$. If the stabilizer of $C$ in the Apollonian group $\Gamma_A$ is nonelementary, then by \cite[Cor.~3.4]{acy_geom_finite}, $\cup\Gamma_A\cdot C_n$ is dense in the set of circles intersecting $\mathcal{A}$. But then the collection of boundary components of $P_n$ would be dense in $X$ (as a matter of fact, in $T_1X$). But this contradicts Proposition~\ref{prop: symbols_closed}.

So either $C$ is an elementary circle, or the corresponding geodesic plane $P$ is not closed in $M_A$. Assume the latter. Then necessarily $P$ intersects $\partial\core(M_A)$ in a geodesic that is not closed. But this cannot be a limit of boundary data of elementary planes, again by Proposition~\ref{prop: symbols_closed}. It follows that the limit plane must be elementary, as desired.
\end{proof}

\begin{eg}
We again look at the sequence of crowns given by (\ref{eq: symbols}). Topologically, these are crowns with four tips. As $m\to\infty$, there are three deeper and deeper excursions into the cusp the core geodesics make. For the modular symbols, two components make two deeper and deeper excursions, and one make one, while the fourth one remains the same. As explained in the introduction, as boundary components go deeper into the cusp, portions of the surface are also pushed into the cusp, and the limit is obtained by pinching cross boundary components.

The end result is a collection of ideal triangles. Considering the symmetry of the planes, we have the pinching process must be as shown in Figure~\ref{fig: pinching_crown}.\qed
\end{eg}

\begin{figure}[htp]
\includegraphics[width=\linewidth]{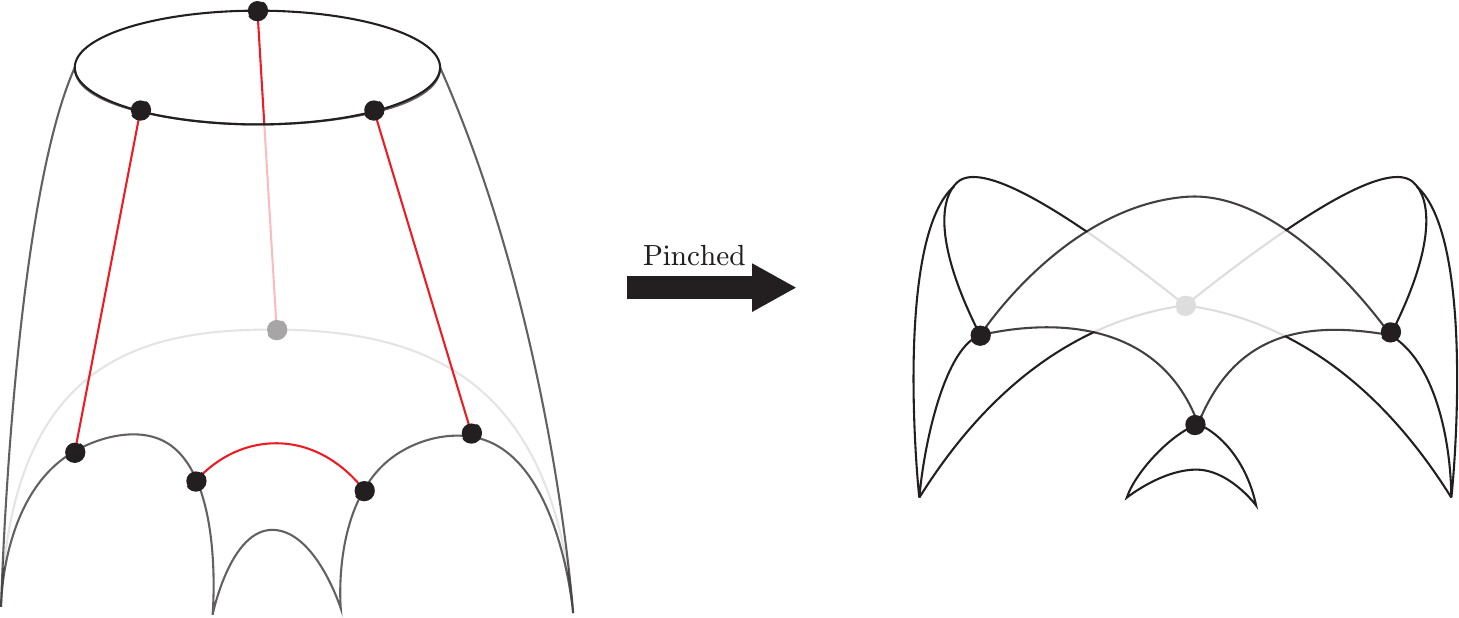}
\caption[As $m\to\infty$, the crown is pinched]{As $m\to\infty$, the crown is pinched along curves connecting boundary components}
\label{fig: pinching_crown}
\end{figure}

\section{Beyond elementary}\label{sec: beyond}
In this section we discuss how the scheme discribed in the previous section also applies to nonelementary surfaces (these surfaces are necessarily closed), and exhibit an example to prove Theorem~\ref{thm: pairs_of_pants}.

Note that to prove an analogue of Theorem~\ref{thm: elem_closed} for some other family of closed surfaces, we only need an analogue of Proposition~\ref{prop: symbols_closed} for the boundary data of these surfaces. For example, given $K>0$, let $\mathcal{F}_K$ be the set of closed geodesic planes $P$ in $M_A$ whose boundary data has uniformly bounded complexity, in the following sense: for each component $\sigma$ of $P\cap\partial\core(M_A)$, either $\sigma=[r]$ where the continued fraction of $r$ has length $\le K$, or $\sigma$ is a closed geodesic whose cutting sequence has reduced length $\le K$. Then
\begin{prop}
The union of geodesic planes in $\mathcal{F}_K$ is closed. Moreover, if a sequence of distinct planes in $\mathcal{F}_K$ limits on another plane $P$ in $\mathcal{F}_K$, then $P$ must be elementary.
\end{prop}

One immediately wonders if there exists a sequence of distinct nonelementary surfaces limiting to a (union of) elementary surface at all. The following example, discovered by reversing the pinching process, shows that there are sequences of pairs of pants limiting to crowns by pinching the seams.
\begin{eg}\label{eg: pair_of_pants}
For any $m\ge3,n\ge1$, $\Cr\left(0,\dfrac{mn+1}{n^2m+2n}\right)$ is hyperbolic, and its core geodesic lies on the boundary. The crown has two spikes. Taking $m\to\infty$, note that we get an ideal quadrilateral $\Cr(0,1/n)$. Then taking $n\to\infty$, we get two copies of the ideal triangle $\Cr(0,0)$. Reversing the process above, we can start with two copies of $\Cr(0,0)$, thicken and connect one corresponding pair of tips to get the ideal quadrilateral $\Cr(0,1/n)$, and then do the same for another pair of tips to get the crown $\Cr\left(0,\dfrac{mn+1}{n^2m+2n}\right)$. If we can connect the final pair of tips, we then get a pair of pants. This is indeed possible, as follows.

The line $l=l\left(0,\dfrac{mn+1}{n^2m+2n}\right)$ intersects the circle $C_{1/n}$ at a pair of conjugate quadratic rationals, the fixed points of the hyperbolic element $M_1=V_2V_1V_1^{m-2}V_3V_1^{-1}V_2^{-1}$. When connecting the final pair of tips, this geodesic is unchanged, so we may look at circles passing through the same pair of quadratic rationals. We hope for a sequence of circles that tends to $l$, and intersects the line $y=-1$ in a pair of conjugate quadratic rationals. The corresponding geodesic should have cutting sequence of the form $\overline{V_3V_1^k}$, so that it limits to the modular symbol $[0]$ when $k\to\infty$. We may take $M_2=V_1^{-1}V_3V_1^kV_1$, and take the desired pair of conjugate quadratic rationals as its fixed points. It is then calculated that $M_1M_2=\begin{pmatrix}-nm-1&-(k+2)(nm+1)-m\\-n(nm+2)&-n(k+2)(nm+2)-nm-1\end{pmatrix}$, which corresponds to the geodesic $\overline{V_2^nV_1^mV_2^nV_1^{k+2}}$. Note that this geodesic limits to the modular symbol $[(mn+1)/(n^2m+2n)]$ as $k\to\infty$.

Finally, note that $M_1M_2$ corresponds to a geodesic lying on $\partial\core(M_A)$ as the cutting sequence involves only two letters. So in the fundamental group of the corresponding closed surface, the product of the loops around two boundary components gives the loop around the third boundary component, so necessarily this corresponds to a pair of pants.\qed
\end{eg}
This example immediately gives Theorem~\ref{thm: pairs_of_pants}.

\section{Further questions}\label{sec: questions}
The analysis of elementary planes in $M_A$ naturally prompts the following questions, from different aspects of the discussions.
\paragraph{Controlling geodesic surfaces with boundary data}
How much do boundary data control the geometry and topology of closed geodesic planes in $M_A$? For example, recall the set $\mathcal{F}_K$ of geodesic planes with boundary data of uniformly bounded complexity. Do planes in $\mathcal{F}_K$ have finite area in the convex core? How are the number of boundary components and the genus of the surface related to the complexity of the boundary data?

A related question is to characterize the boundary data of surfaces of bounded genus with bounded number of boundary components. The answer could be framed in terms of modular symbols and cutting sequences, but another perspective is to give conjugacy classes of curves in the boundary of the genus 2 handlebody (see Figure~\ref{fig: handlebody} and surrounding discussions in the appendix).

\paragraph{Reversing the pinching process}
An immediate question following Example~\ref{eg: pair_of_pants} is how and when the reverse process to pinching works. Topologically, this reverse process is always possible, but there may not be a geodesic plane realizing the topological surface obtained. Also, from our lists of elementary surfaces, it seems that there are isolated ones that do not come in infinite families, which may not be obtained by reversing pinching. What about nonelementary geodesic planes? Do they always come in such families?

\paragraph{Generalization to other acylindrical manifolds}
Finally, how much do the results generalize to other geometrically finite acylindrical manifolds not covered by \cite{MMO2, acy_geom_finite}? The classification in Proposition~\ref{prop: elementary_circle} holds in general, but there is no guarantee of a symbolic coding scheme introduced in \S\ref{sec: diophantine_apollonian}. So a starting question may be: is there a way to prove our main results \emph{without} a complete list of elementary planes?

We have occasionally used the arithmeticity of $\psl(2,\mathbb{Z}[i])$ to easily rule out certain exotic possibilities. So one larger family of acylindrical manifolds to consider might be those covering the finite volume orbifolds whose fundamental groups are the Bianchi groups $\psl(2,\mathcal{O}_d)$, where $\mathcal{O}_d$ is the ring of integers in the imaginary quadratic field $\mathbb{Q}(\sqrt{-d})$.

\appendix
\section{Appendix: Manifold covers of the Apollonian orbifold}\label{sec: appendix}
\subsection{Manifold covers of the Apollonian orbifold}\label{sec: manifold_cover}
We now give some finite manifold covers of $M_A$ and show that they are acylindrical by Thurston's definition.

\paragraph{A manifold cover of $M_A$}
Recall that $O$ is the regular ideal hyperbolic octahedron (Figure~\ref{fig: octahedron}). Let $T_1=V_1=\begin{pmatrix}1&1\\0&1\end{pmatrix}$, and $T_2=\begin{pmatrix}2+i&-1\\2i&-i\end{pmatrix}$. Then $T_1$ maps $0,-i,\infty$ to $1,1-i,\infty$ respectively, and $T_2$ maps $0,1,1/2-i/2$ to $-i,1-i,1/2-i/2$. So $T_1$ and $T_2$ identifies the gray faces of the octahedron $O$ in pairs.

Since all dihedral angles of $O$ are $\pi/2$, it is easy to see that after identification, we get a manifold with totally geodesic boundary. It is the convex core of the hyperbolic 3-manifold $M_{A,1}$ given by the Kleinian group $\Gamma_{A,1}:=\langle T_1,T_2\rangle$. Since $T_1,T_2\in\Gamma_A$, and $\core(M_{A,1})$ has finite volume, we know that $M_{A,1}$ is a finite manifold cover of $M_A$, and in particular the limit set of $\Gamma_{A,1}$ is also the Apollonian gasket.

Topologically $M_{A,1}$ is homeomorphic to the interior of a handlebody $S$ of genus $2$.  By carefully examining the gluing procedure, it is easy to see that the Kleinian group $\Gamma_{A,1}=\langle T_1,T_2\rangle$ is a representation of $\pi_1(S)$ so that the following closed curves $\gamma_0,\gamma_1,\gamma_2$, as labeled in Figure~\ref{fig: handlebody}, are represented by parabolic elements.

\begin{figure}[htp]
\centering
\includegraphics[width=0.6\linewidth]{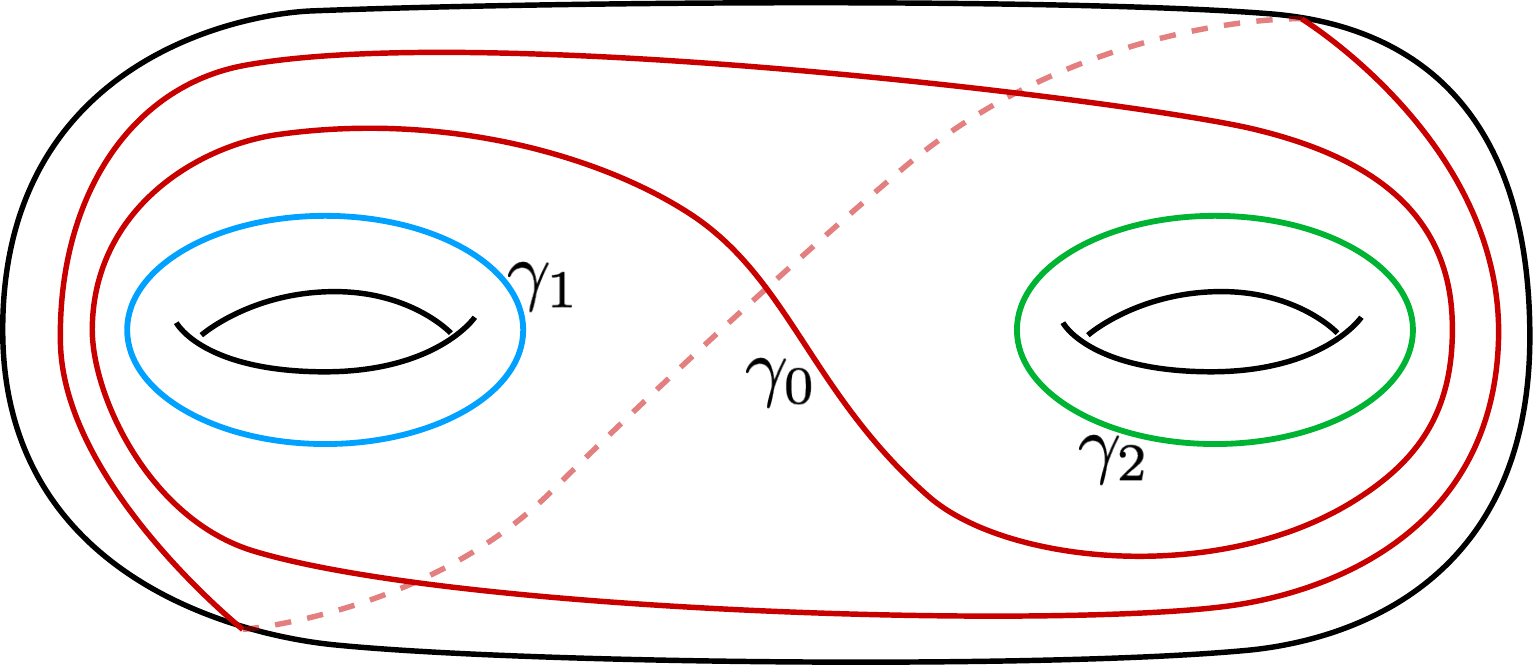}
\caption{$M_{A,1}$ as an acylindrical pared manifold}
\label{fig: handlebody}
\end{figure}

Moreover, the collection of curves $\gamma_0,\gamma_1,\gamma_2$ satisfies the assumptions of \cite[Rem.~1.5]{otal_currents}, so $(S,\cup \tilde\gamma_i)$ is acylindrical as a pared manifold, where $\tilde{\gamma}_i$ is a slight thickening of $\gamma_i$. The manifold $M_{A,1}$ is a hyperbolic realization of $(S,\cup\tilde\gamma_i)$, so it is acylindrical in the sense of Thurston.

We remark that this implies \emph{every} manifold cover of $M_A$ is acylindrical. Indeed, if not, a disk/annulus violating the definition projects down to one in $M_A$ and can then be lifted to one in $M_{A,1}$.

Here is another way of visualizing $M_{A,1}$: it is a solid cylinder with two linked cylinders drilled out, and the parabolic locus is precisely the core curves of the three cylinders; see Figure~\ref{fig: cylinder_rep}. Indeed, Figure~\ref{fig: handlebody2} is simply the handlebody in Figure~\ref{fig: handlebody} in a slightly different view, and from there it is easy to get Figure~\ref{fig: cylinder_rep}.
\begin{figure}[htp]
\centering
\begin{subfigure}[b]{0.35\linewidth}
\includegraphics[width=\textwidth]{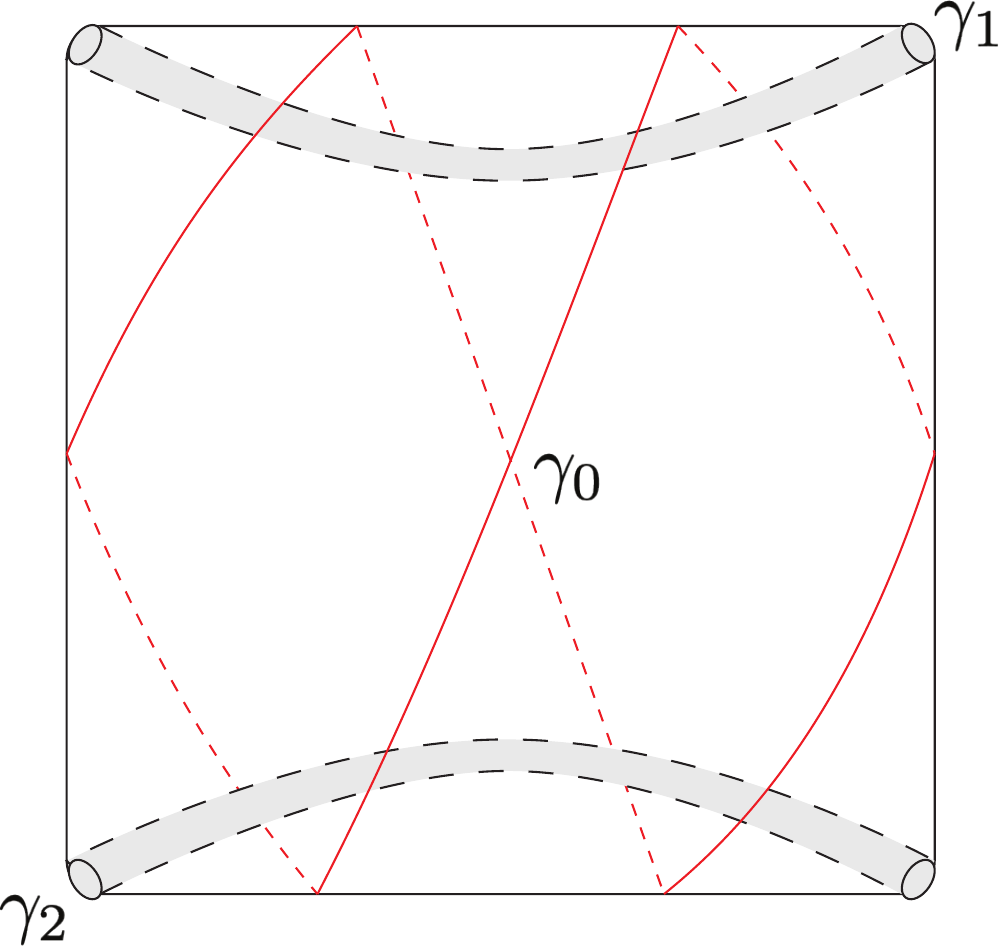}
\caption{A different view of Figure~\ref{fig: handlebody}}
\label{fig: handlebody2}
\end{subfigure}
\hspace{0.1\linewidth}
\begin{subfigure}[b]{0.5\linewidth}
\includegraphics[width=\textwidth]{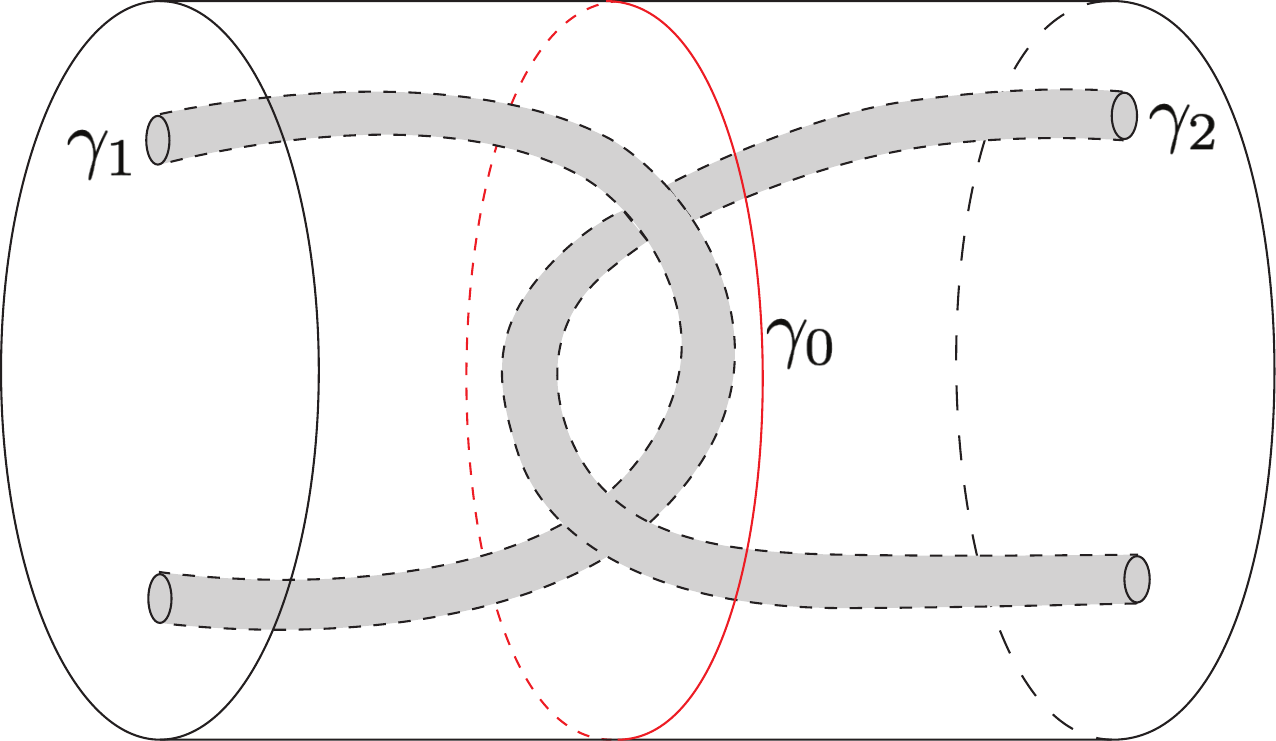}
\caption{A cylinder model for $M_{A,1}$}
\label{fig: cylinder_rep}
\end{subfigure}
\caption{Different topological models for $M_{A,1}$}
\label{fig: top_models}
\end{figure}

One advantage of this model: it is immediately clear that the boundary of $M_{A,1}$ consists of two copies of the three-punctured sphere. This manifold cover of $M_A$ also appears in \cite[Ex.~2.45]{hyperbolic_manifolds}.

\paragraph{Another manifold cover of $M_A$}
Another way to construct a manifold cover of $M_A$ is as follows. Refections in all the gray faces of $O$ generate a discrete subgroup of $\isom(\mathbb{H}^3)$, and the subgroup of orientation preserving elements is a Kleinian group $\Gamma_{A,2}$. As we have discussed before, one way to construct the corresponding hyperbolic manifold $M_{A,2}$ is to take two copies of the polyhedron obtained by extending $O$ across all white faces to infinity and glue corresponding faces. Then $\core(M_{A,2})$ is obtained by gluing two copies of $O$ along gray faces, and is a manifold with totally geodesic boundary.

Similar to the discussion above, we know that the limit set of $\Gamma_{A,2}$ is precisely the Apollonian gasket $\mathcal{A}$, and $M_{A,2}$ is a finite manifold cover of $M_A$. Topologically $M_{A,2}$ is a handlebody of genus 3, and the parabolic locus consists of 6 curves, as shown in Figure~\ref{fig: genus3}. As apparent from Figure~\ref{fig: genus3}, this cover is highly symmetric.
\begin{figure}[htp]
\centering
\includegraphics[width=0.4\textwidth]{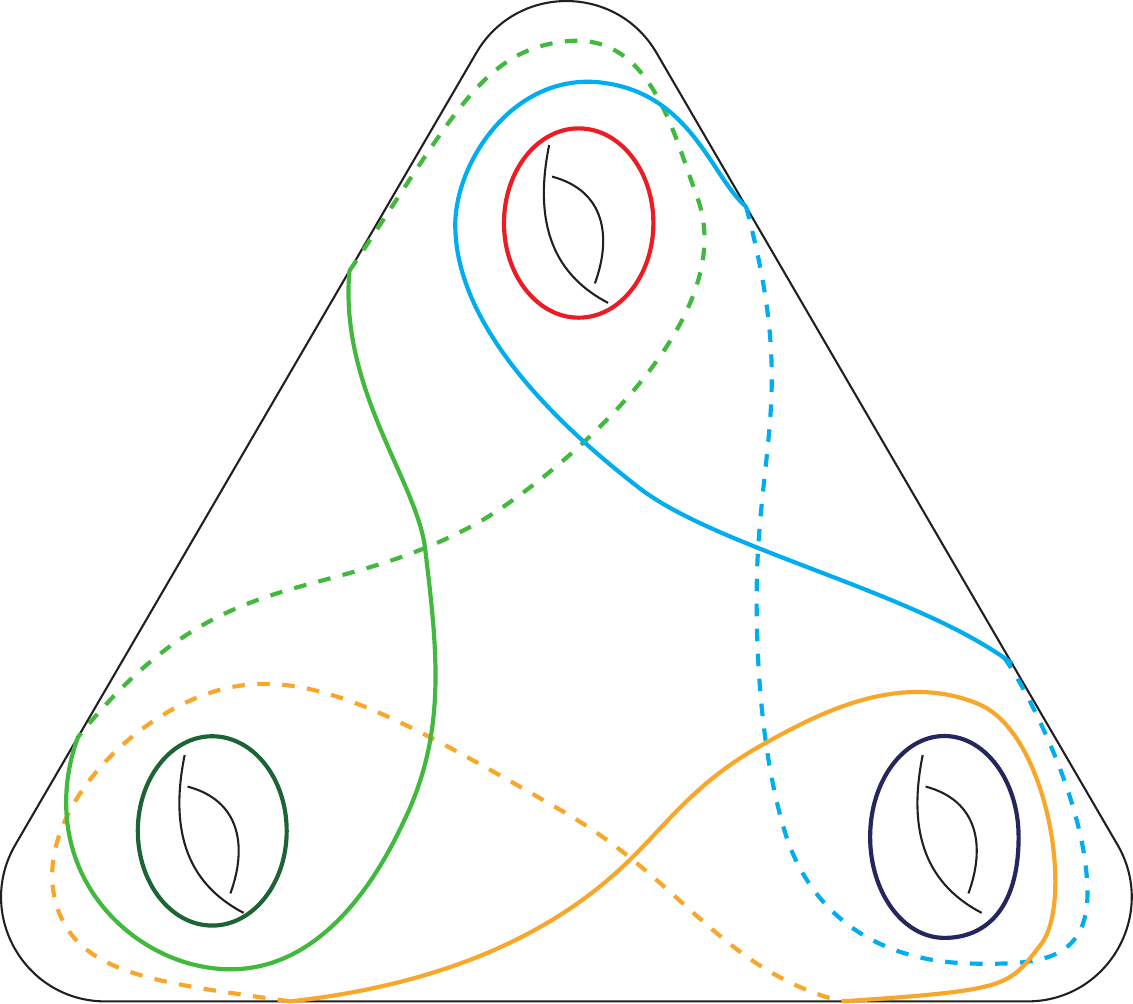}
\caption{$M_{A,2}$ as an acylindrical pared manifold}
\label{fig: genus3}
\end{figure}

\subsection{Some visualizations of elementary planes in the models}
We now draw a few simple examples of elementary planes in the topological models given above. Note that $\Cr(0,0)$ and $\Cr(0,1)$ are ideal polygons, so in the topological models they are disks whose boundary circle intersects the parabolic locus, see Figure~\ref{fig: elementary_in_top_polygon}. On the other hand $\Cr(1/2,1/2)$ is a punctured ideal polygon, so in the topological models it is a cylinder, and one of its boundary components lies in the parabolic locus, see Figure~\ref{fig: elementary_in_top_punctured}. We have drawn them in the model most convenient for the purpose; we can also convert the figures into another model via the topological process introduced in Figure~\ref{fig: top_models}.
\begin{figure}[htp]
\centering
\begin{subfigure}[b]{0.45\linewidth}
\includegraphics[width=\textwidth]{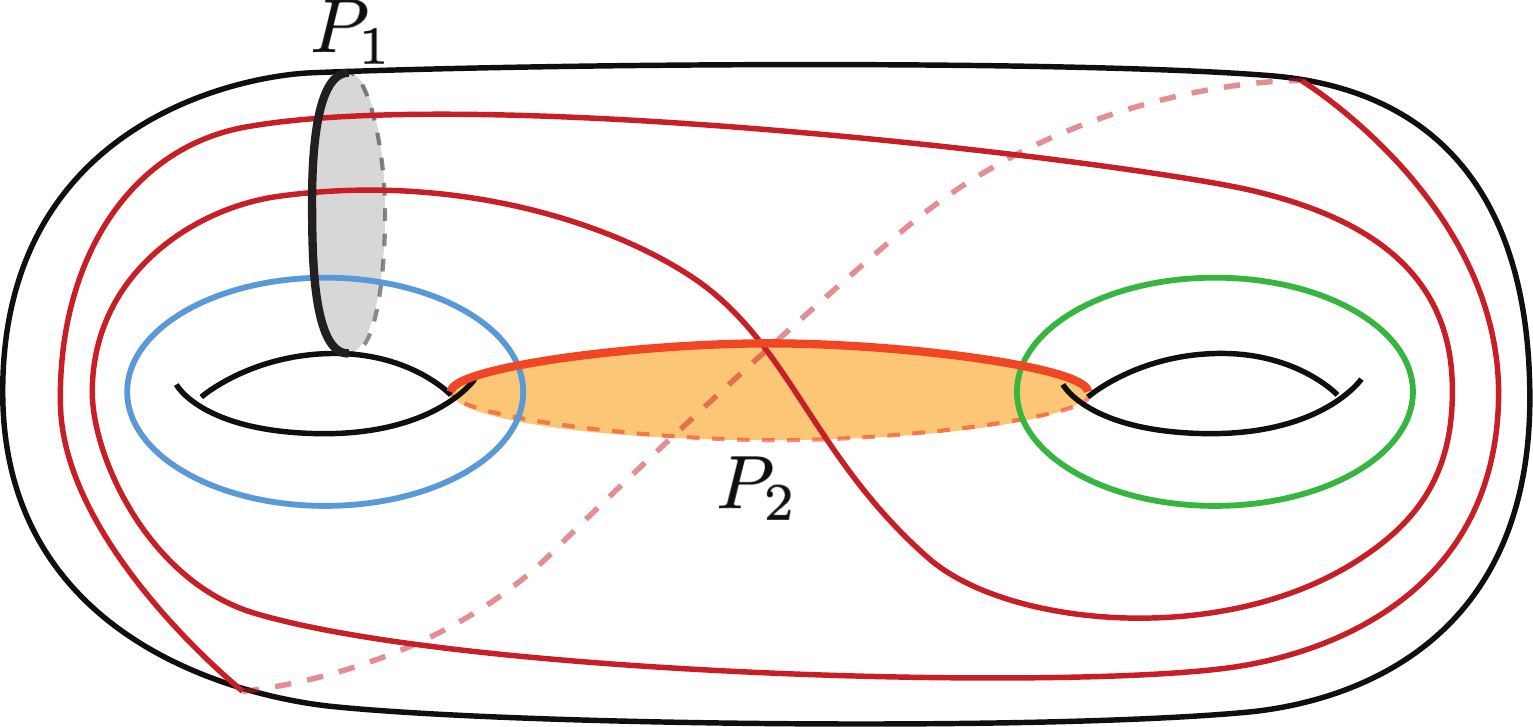}
\caption{Two ideal polygons in $M_A$}
\label{fig: elementary_in_top_polygon}
\end{subfigure}
\hspace{0.1\linewidth}
\begin{subfigure}[b]{0.4\linewidth}
\includegraphics[width=\textwidth]{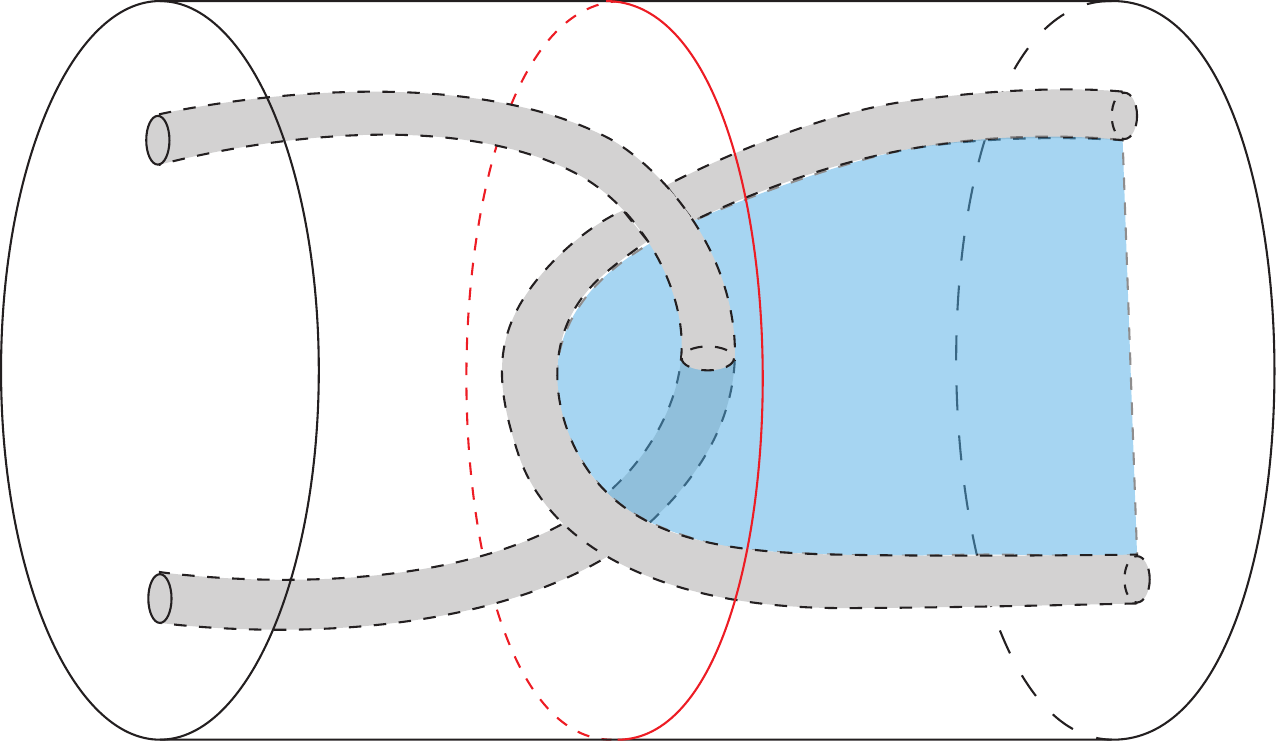}
\caption{A punctured ideal polygon in $M_A$}
\label{fig: elementary_in_top_punctured}
\end{subfigure}
\caption{Some examples of elementary planes in topological models}
\label{fig: elementary_in_top}
\end{figure}

\section{Appendix: Elementary planes in Apollonian chains}

Recall that $M_{A,1}$ is the manifold cover of $M_A$ with topological model depicted in Figures~\ref{fig: handlebody} and \ref{fig: top_models}. Its convex core consists of two components, each isometric to the thrice-punctured sphere.

We can glue $n$ copies of the convex core along the geodesic boundary in a chain; this is the convex core of a geometrically finite, acylindrical hyperbolic 3-manifold, which we call the \emph{Apollonian $n$-chain}, and denote by $M_A^{(n)}$. We have

\begin{thm}\label{thm: main_chain}
\begin{enumerate}[label=\normalfont{(\arabic*)}, topsep=0mm, itemsep=0mm]
\item The area of an elementary plane in the convex core of $M_A^{(n)}$ is uniformly bounded above;
\item The set of elementary planes in $M_A^{(n)}$ is closed.
\end{enumerate}
\end{thm}
In other words, our main results hold for any Apollonian $n$-chain. It is easy to see that these follow from the corresponding results for $M_A$, since an elementary plane intersects each piece in the chain in an elementary plane.

A closed geodesic plane of $M_A^{(n)}$ is called \emph{peripheral} if it only intersects the one of the two copies of $\core(M_{A,1})$ on either end of the chain. We have
\begin{prop}
\begin{enumerate}[label=\normalfont{(\arabic*)}, topsep=0mm, itemsep=0mm]
\item Every peripheral elementary plane is a punctured ideal polygon;
\item The only ideal polygon in $M_A^{(2)}$ comes from gluing two ideal triangles;
\item The only non-peripheral elementary planes in $M_A^{(3)}$ come from gluing three ideal triangles;
\item When $n\ge4$, every elementary plane in $M_A^{(n)}$ is peripheral.
\end{enumerate}
\end{prop}
\begin{proof}
For (1), if $l(\alpha,\beta)$ corresponds to a peripheral elementary plane, it cannot intersects either $x^2+(y-1/2)^2=1/4$ or $(x-1)^2+(y-1/2)^2=1/4$. This is only possible when $\alpha=\beta=1/2$, which gives a punctured ideal polygon; see Figure~\ref{fig: elementary_in_top_punctured}.

For (2), any ideal polygon in $M_A^{(2)}$ comes from gluing two ideal polygons in each piece. By going through the list, this is only possible for two copies of the ideal triangle corresponding to $l(0,0)$; see Figure~\ref{fig: elementary_in_MA2}. All the other possibilities produce either a cusp, or a closed geodesic.
\begin{figure}[htp]
\centering
\includegraphics[width=0.8\textwidth]{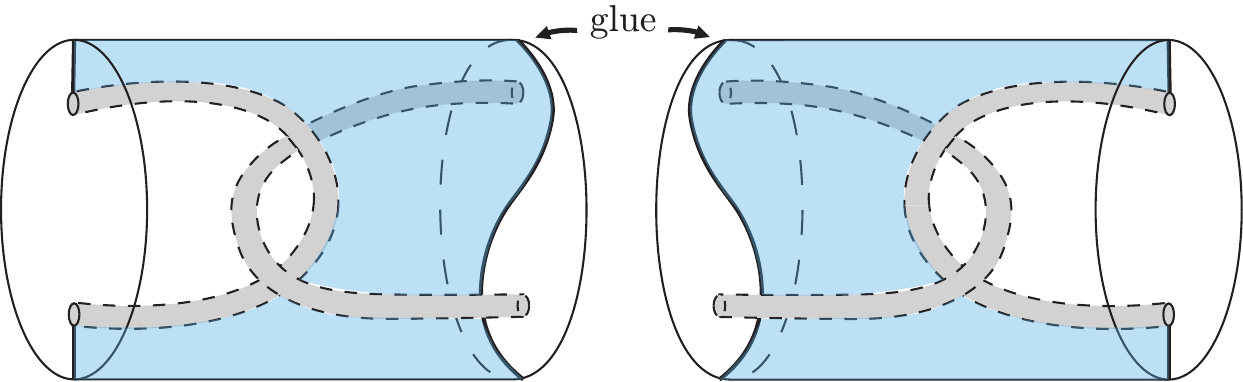}
\caption{An ideal polygon in $M_A^{(2)}$}
\label{fig: elementary_in_MA2}
\end{figure}

For (3), by adding a third piece in Figure~\ref{fig: elementary_in_MA2}, we can see that gluing three pieces of the ideal triangle gives a non-peripheral elementary plane. In all other cases, the two gluings produce two parabolic/hyperbolic elements, and they generate a non-elementary group.

For (4), note that even gluing $\ge4$ pieces of ideal triangles gives a non-elementary surface.
\end{proof}
Note that the proof implies that when $n\ge3$, $M_A^{(n)}$ contains only finitely many elementary planes (so Theorem~\ref{thm: main_chain} is trivially true!). In $M_A^{(2)}$, two quadrilaterals corresponding to $l(0,1/n)$ glue up to a double crown, so infinite sequences of elementary planes still exist.

\bibliographystyle{math}
\bibliography{biblio}

\newcommand{\etalchar}[1]{$^{#1}$}
\begin{thebibliography}{{Zha}2}

\bibitem[Apo]{modular_functions_apostol}
T.~M. Apostol.
\newblock {\em Modular Functions and Dirichlet Series in Number Theory},
  volume~41 of {\em Graduate Texts in Mathematics}.
\newblock Springer, New York, NY, 1976.

\bibitem[BO]{acy_geom_finite}
Y.~{Benoist} and H.~{Oh}.
\newblock {{Geodesic planes in geometrically finite acylindrical 3-manifolds}}.
\newblock {\em Ergodic Theory Dynam. Systems} {\bf 42}(2022), 514--553.

\bibitem[BHC]{arithmetic_lattice}
A.~Borel and Harish-Chandra.
\newblock {Arithmetic subgroups of algebraic groups}.
\newblock {\em Ann. Math.} {\bf 75}(1962), 485--535.

\bibitem[For]{ford1938fractions}
L.~R. Ford.
\newblock {Fractions}.
\newblock {\em Amer. Math. Monthly} {\bf 45}(1938), 586--601.

\bibitem[Fri]{poly2}
R.~Frigerio.
\newblock {An infinite family of hyperbolic graph complements in {$S^3$}}.
\newblock {\em J. Knot Theory Ramifications} {\bf 14}(2005), 479--496.

\bibitem[Gas]{gaster}
J.~Gaster.
\newblock {A family of non-injective skinning maps with critical points}.
\newblock {\em Trans. Amer. Math. Soc.} {\bf 368}(2016), 1911--1940.

\bibitem[GLM{\etalchar{+}}]{graham2005apollonian}
R.~L. Graham, J.~C. Lagarias, C.~L. Mallows, A.~R. Wilks, and C.~H. Yan.
\newblock {Apollonian circle packings: geometry and group theory I. The
  Apollonian group}.
\newblock {\em Discrete Comput. Geom.} {\bf 34}(2005), 547--585.

\bibitem[Kha]{khalil2019geodesic}
Osama Khalil.
\newblock {Geodesic planes in geometrically finite manifolds}.
\newblock {\em Discrete Contin. Dyn. Syst. Ser. A} {\bf 39}(2019), 881--903.

\bibitem[Khi]{continued_fraction_book}
A.~Ya. Khinchin.
\newblock {\em Continued Fractions}.
\newblock P. Noordhoff, Groningen, 1963.

\bibitem[Lan]{modular_forms1}
S.~Lang.
\newblock {\em Introduction to Modular Forms}.
\newblock Number 222 in Grundlehren der mathematischen Wissenschaften.
  Springer-Verlag, Berlin Heidelberg, 1995.

\bibitem[MR]{arithmetic}
C.~Maclachlan and A.~W. Reid.
\newblock {\em The Arithmetic of Hyperbolic 3-Manifolds}.
\newblock Number 219 in Graduate Texts in Mathematics. Springer-Verlag, 2003.

\bibitem[MT]{hyperbolic_manifolds}
K.~Matsuzaki and M.~Taniguchi.
\newblock {\em Hyperbolic Manifolds and Kleinian Groups}.
\newblock Oxford Mathematical Monographs. Oxford University Press, 1998.

\bibitem[McM]{modular_symbol}
C.~T. McMullen.
\newblock {Modular symbols for {Teichm\"uller} curves}.
\newblock {\em J. Reine Angew. Math.} {\bf 777}(2021), 89--125.

\bibitem[MMO1]{MMO1}
C.~T. McMullen, A.~Mohammadi, and H.~Oh.
\newblock {Geodesic planes in hyperbolic 3-manifolds}.
\newblock {\em Invent. Math.} {\bf 209}(2017), 425--461.

\bibitem[MMO2]{MMO2}
C.~T. {McMullen}, A.~{Mohammadi}, and H.~{Oh}.
\newblock {{Geodesic planes in the convex core of an acylindrical 3-manifold}}.
\newblock {\em Duke Math. J.} {\bf 171}(2022), 1029--1060.

\bibitem[MO]{mohammadi2020isolations}
A.~Mohammadi and H.~Oh.
\newblock {Isolations of geodesic planes in the frame bundle of a hyperbolic
  3-manifold}.
\newblock {\em Preprint}.

\bibitem[Mor]{morris2001introduction}
D.~W. Morris.
\newblock {\em Introduction to arithmetic groups}.
\newblock arXiv math/0106063, 2001.

\bibitem[Ota]{otal_currents}
J.-P. Otal.
\newblock {\em Courants g{\'e}od{\'e}siques et surfaces}.
\newblock PhD thesis, Universit{'e} de Paris-Sud, Orsay, 1988.

\bibitem[PZ]{poly1}
L.~Paoluzzi and B.~Zimmermann.
\newblock {On a class of hyperbolic 3-manifolds and groups with one defining
  relation}.
\newblock {\em Geom. Dedicata} {\bf 60}(1996), 113--123.

\bibitem[Rat]{ratner}
M.~Ratner.
\newblock {Raghunathan's topological conjecture and distributions of unipotent
  flows}.
\newblock {\em Duke Math. J.} {\bf 63}(1991), 235--280.

\bibitem[Sch]{diophantine}
A.~L. Schmidt.
\newblock {Diophantine approximation of complex numbers}.
\newblock {\em Acta Math.} {\bf 134}(1975), 1--85.

\bibitem[Ser]{continued_fraction}
C.~Series.
\newblock {The modular surface and continued fractions}.
\newblock {\em J. London Math. Soc.} {\bf s2-31}(1985), 69--80.

\bibitem[Sha]{shah}
N.~Shah.
\newblock {Closures of totally geodesic immersions in manifolds of constant
  negative curvature}.
\newblock In {\em Group Theory from a Geometrical Viewpoint (Trieste, 1990)},
  pages 718--732. World Scientific, 1991.

\bibitem[Ste]{stein2007modular}
W.~A. Stein.
\newblock {\em Modular forms, a computational approach}, volume~79 of {\em
  Graduate Studies in Mathematics}.
\newblock American Mathematical Soc., 2007.

\bibitem[Thu1]{hyperbolization1}
W.~P. Thurston.
\newblock {Hyperbolic structures on 3-manifolds I: Deformation of acylindrical
  manifolds}.
\newblock {\em Ann. of Math.} {\bf 124}(1986), 203--246.

\bibitem[Thu2]{thurston_book}
W.~P. Thurston.
\newblock {\em Three-Dimensional Geometry and Topology}, volume~1.
\newblock Princeton University Press, 1997.

\bibitem[{Zha}1]{double_lunchbox}
Y.~{Zhang}.
\newblock {Construction of acylindrical hyperbolic 3-manifolds with
  quasifuchsian boundary}.
\newblock {\em Exp. Math.} {\bf 31}(2022), 883--896.

\bibitem[{Zha}2]{exotic_plane}
Y.~{Zhang}.
\newblock {Existence of an exotic plane in an acylindrical 3-manifold}.
\newblock {\em To appear, Math. Res. Lett.}

\end{thebibliography}

\end{document}